\documentclass[final,3p,times]{elsarticle}
\usepackage{amsmath}
\usepackage{amssymb}
\usepackage{longtable}
\usepackage{amsthm}
\usepackage{hyperref}
\usepackage{algorithm}
\usepackage{algpseudocode}

\theoremstyle{definition}
\newtheorem{defn}{Definition}[section]
\newtheorem{conj}{Conjecture}[section]

\newtheorem{ques}{Question}[section]
\newtheorem{lem}{Lemma}[section]
\newtheorem{cor}{Corollary}[lem]
\newtheorem{thm}{Theorem}[section]

\newcommand{\set}[1]{\{#1\}}
\renewcommand{\leq}{\le} 

\newcommand{\ol}{\overline}

\newcommand{\NN}{\mathbb{N}}
\newcommand{\NNZ}{\NN_0}
\newcommand{\ZZ}{\mathbb{Z}}
\newcommand{\RR}{\mathbb{R}}
\providecommand{\abs}[1]{\lvert #1 \rvert}
\newcommand{\ceil}[1]{\lceil #1 \rceil}

\newcommand{\OKlibrary}{\texttt{OKlibrary}}
\newcommand{\OKinternet}{\url{http://www.ok-sat-library.org}}

\newcommand{\tawsolver}{\texttt{tawSolver}}
\newcommand{\otawsolver}{\tawsolver-1.0}
\newcommand{\ntawsolver}{\tawsolver-2.6}
\newcommand{\ttawsolver}{$\tau$\texttt{awSolver}-2.6}

\newcommand{\oksolver}{\texttt{OKsolver}}
\newcommand{\satz}{\texttt{satz}}
\newcommand{\march}{\texttt{march\_pl}}
\newcommand{\minisat}{\texttt{MiniSat}}
\newcommand{\cryptomini}{\texttt{CryptoMiniSat}}
\newcommand{\glucose}{\texttt{Glucose}}
\newcommand{\picosat}{\texttt{PicoSAT}}
\newcommand{\precosat}{\texttt{PrecoSAT}}
\newcommand{\lingeling}{\texttt{Lingeling}}
\newcommand{\cubeconq}{\texttt{Cube\:\&\:Conquer}}

\DeclareMathOperator{\epos}{EP1S}
\DeclareMathOperator{\np}{\mathit{n}_p}
\DeclareMathOperator{\nv}{\mathit{n}_v}
\DeclareMathOperator{\vdw}{w}
\DeclareMathOperator{\vdwpd}{pdw}
\DeclareMathOperator{\pdg}{pdg}
\DeclareMathOperator{\pds}{pds}
\DeclareMathOperator{\arithp}{ap}
\DeclareMathOperator{\pdarithp}{pdap}
\DeclareMathOperator{\Fpd}{\mathit{F}^{pd}}
\DeclareMathOperator{\mir}{m}
\DeclareMathOperator{\ldeg}{ld} 

\journal{}

\begin{document}

\begin{frontmatter}
\title{On the van der Waerden numbers $\vdw(2;3,t)$}
\author[ahmed_address]{Tanbir Ahmed}
\address[ahmed_address]{Department of Computer Science and Software Engineering,
Concordia University, Montr\'eal, Canada. \\
ta\_ahmed@cs.concordia.ca}

\author[kullmann_address]{Oliver Kullmann}
\address[kullmann_address]{Computer Science Department,
Swansea University, Swansea, UK. \\
O.Kullmann@Swansea.ac.uk}

\author[snevily_address]{Hunter Snevily\fnref{fn1}}
\address[snevily_address]{Department of Mathematics,
University of Idaho - Moscow, Idaho, USA.}
\fntext[fn1]{Hunter Snevily passed away on November 11, 2013 after his long struggle with
Parkinson’s disease. He was an inspiring mathematician. 
We have lost a great friend and colleague. He will be heavily missed
and fondly remembered}

\begin{abstract}
In this paper
we present results and conjectures on the ordinary \href{http://en.wikipedia.org/wiki/Van_der_Waerden_number}{van der Waerden numbers} $\vdw(2; 3, t)$ and on the new \emph{palindromic van der Waerden numbers} $\vdwpd(2; 3, t)$.
We have computed the exact value of the previously unknown number $\vdw(2; 3, 19) = 349$, and we provide new lower bounds for $20 \leq t \leq 39$, where for $20 \leq t \leq 30$ we conjecture these bounds to be exact.
The lower bounds for $\vdw(2; 3, t)$ with $24 \leq t \leq 30$ refute the conjecture that $\vdw(2; 3, t)\leq t^2$ as suggested in \cite{blr2008}.
Based on the known values of $\vdw(2; 3, t)$, we investigate regularities to better understand the lower bounds of $\vdw(2; 3, t)$.
Motivated by such regularities, we introduce palindromic van der Waerden numbers $\vdwpd(k; t_0,\dots,t_{k-1})$, which are defined as the ordinary numbers $\vdw(k;t_0,\dots,t_{k-1})$, but where only palindromic solutions are considered, reading the same from both ends. Different from the situation for ordinary van der Waerden numbers, these ``numbers'' need actually to be pairs of numbers.
We compute $\vdwpd(2;3,t)$ for $3 \leq t \leq 27$, and we provide bounds for $t \leq 39$, which we believe to be exact for $t \leq 35$.
All computations are based on SAT solving, and we discuss the various relations between SAT solving and Ramsey theory. Especially we introduce a novel (open-source) SAT solver, the \tawsolver, which performs best on the SAT instances studied here, and which is actually the original DLL-solver (\cite{dpll}), but with an efficient implementation and a modern heuristic typical for look-ahead solvers (applying the theory developed in \cite{Kullmann2007HandbuchTau}).
\end{abstract}

\end{frontmatter}

\tableofcontents

\section{Introduction}
\label{sec:intro}

We consider \href{http://en.wikipedia.org/wiki/Ramsey_theory}{Ramsey theory} and its connections to computer science (see \cite{Rosta2004RamseySurvey}
for a survey)
by exploring a rather recent link, especially to algorithms and formal methods, namely to \href{http://en.wikipedia.org/wiki/Boolean_satisfiability_problem}{SAT solving}.
SAT is the problem of finding a satisfying assignment for a propositional formula.
Since Ramsey problems can naturally be formulated as SAT problems,
SAT solvers can be used to compute numbers from Ramsey theory.
In the present article, we consider van der Waerden numbers (\cite{vanderWaerden1927Baudet}),
where SAT had its biggest success in Ramsey theory, namely the determination of $\vdw(2; 6,6) = 1132$ in \cite{KourilPaulW26}, the first new diagonal van der Waerden (short ``vdW'') number after almost 30 years.

\begin{defn}\label{def:vdwn}
  We use $\NN = \set{x \in \ZZ : x \ge 1}$, $\NNZ = \NN \cup \set{0}$. An \emph{arithmetic progression} of length $t \in \NN$ is a subset $p \subset \NN$ of length $\abs{p} = t$ and of the form $p = \set{a + i \cdot d : i \in \set{0, \dots, t-1}}$ for some $a, d \in \NN$. A \emph{block partition} of length $k \in \NN$ of a set $X$ is a tuple $(P_0,\dots, P_{k-1})$ of length $k$ of subsets of $X$ (possibly empty) which are pairwise disjoint ($P_i \cap P_j = \emptyset$ for $i \ne j$) and with $P_0 \cup \dots \cup P_{k-1} = X$. The \emph{van der Waerden number} $\vdw(k; t_0, t_1, \dots, t_{k-1}) \in \NN$ for $k, t_0, \dots, t_{k-1} \in \NN$ is the smallest $n \in \NN$ such that for any block partition $(P_0,\dots,P_{k-1})$ of length $k$ of $\set{1,\dots,n}$ there exists a $j \in \set{0,\dots,k-1}$ such that $P_j$ contains an arithmetic progression of length $t_j$.
\end{defn}
That we have $\vdw(k; t_0, t_1, \ldots, t_{k-1}) > n$ can be certified by an appropriate block partition of $\set{1,\dots,n}$; such partitions are the solutions of the SAT problems to be constructed, and we call them ``good partitions'':
\begin{defn}\label{def:goodpart}
  A \emph{good partition} of $\set{1,\dots,n}$ (where $n \in \NNZ$) w.r.t.\ parameters $t_0, t_1, \ldots, t_{k-1}$ is a block partition $(P_0,\dots,P_{k-1})$ of $\set{1,\dots,n}$ containing no block $P_j$ with an arithmetic progression of length $t_j$ (for any $j$).
\end{defn}
So there exists a good partition of $\set{1,\dots,n}$ if and only if $n < \vdw(k; t_0, t_1, \dots, t_{k-1})$. For every $k, t_0, \dots, t_{k-1} \in \NN$ the only block partition of $\set{1,\dots,0} = \emptyset$ is $(\emptyset,\dots,\emptyset)$, and this is a good partition. In this paper, we are interested in the specific van der Waerden numbers $\vdw(2; 3, t)$, $t\geq 3$. Specialising the general definition we obtain:
\begin{center}
  $\vdw(2;3,t)$ is the smallest $n \in \NN$, such that\\
  for all $P_0, P_1 \subseteq \set{1,\dots,n}$ with $P_0 \cap P_1 = \emptyset$ and $P_0 \cup P_1 = \set{1,\dots,n}$\\
  either $P_0$ has an arithmetic progression of size $3$ or $P_1$ has an arithmetic progression of size $t$, or both.
\end{center}
The known exact values of $\vdw(2; 3, t)$ are shown in Table \ref{tab:known3t} (with our contribution in bold).
\begin{table}[H]
  \setlength{\tabcolsep}{5pt}
  \centering
  \begin{tabular}{c|*{17}{c}}
    $t$ & 3 & 4 & 5 & 6 & 7 & 8 & 9 & 10 & 11 & 12 & 13 & 14 & 15 & 16 & 17 & 18 & \textbf{19}\\
    \hline
    $\vdw(2;3,t)$ & 9 & 18 & 22 & 32 & 46 & 58 & 77 & 97 & 114 & 135 & 160 & 186 & 218 & 238 & 279 & 312 & \textbf{349}
  \end{tabular}
  \caption{Known values for $\vdw(2;3,t)$}
  \label{tab:known3t}
\end{table}

As references and for relevant information on the above numbers,
see Chv\'atal \cite{chvatal1970}, Brown \cite{brown1974},
Beeler and O'Neil \cite{beelneil1979}, Kouril \cite{KourilPaulW26},
Landman, Robertson and Culver \cite{lrc2005}, and Ahmed \cite{ahmed2009, ahmed2010, ahmed2011, ahmed2013}.\footnote{This sequence is \url{http://oeis.org/A007783} in the ``On-Line Encyclopedia of Integer Sequences''.}
Recently, Kullmann \cite{kullmann0}\footnote{the conference article \cite{kullmann} contains only material related to Green-Tao numbers and SAT}
reported the following lower bounds
$$\vdw(2; 3, 19)\geq 349, \vdw(2; 3, 20)\geq 389, \vdw(2;3, 21)\geq 416.$$

We confirm the exact value of $\vdw(2; 3, 19) = 349$, and we extend
the list of lower bounds up to $t=39$. Brown, Landman, and Robertson \cite{blr2008},
showed the lower bound $\vdw(2; 3, t) > t^{2- 1 / \log \log t}$ for $t \geq 4 \cdot 10^{316}$, and
observed that $\vdw(2; 3, t)\leq t^2$ for $5\leq t\leq 16$, suggesting that this might hold for all $t$.
Our lower bounds in Subsection \ref{sec:Somenewconjectures} however prove that there are $t$ with  $\vdw(2; 3, t) > t^2$. We provide an improved upper bound $1.675 t^2$ in Subsection \ref{sec:conjupb} (satisfying all known values and lower bounds of $\vdw(2; 3, t)$).

We also present a new type of van-der-Waerden-like numbers, namely \emph{palindromic number-pairs}, obtained by the constraint on good partitions that they must be symmetric under reflection at the mid-point of the interval $\set{1,\dots,n}$. Perceived originally only as a heuristic tool for studying ordinary van der Waerden numbers, it turned out that these numbers are interesting objects on their own. An interesting phenomenon is that we no longer have the standard behaviour of the SAT instances with increasing $n$, where
\begin{itemize}
\item first all instances are satisfiable (for $n < \vdw(k;t_0,\dots,t_{k-1})$), and from a certain point on (the van der Waerden number) all instances are unsatisfiable (for $n \geq \vdw(k;t_0,\dots,t_{k-1})$),
\item but now first again all instances are satisfiable (for $n \leq p$), then we have a region with strict alternation between unsatisfiability and satisfiability, and only from a second point on all instances are unsatisfiable (for $n \geq q$).
\end{itemize}
These two turning points constitute the palindromic ``number'' $\vdwpd(2;3,t) = (p,q)$ as pairs of natural numbers. We were able to compute $\vdwpd(2;3,t)$ for $t \leq 27$. We also provide (conjectured) values for $t \leq 39$.\footnote{The sequence $\vdwpd(2;3,t)$ is \url{http://oeis.org/A198684}, \url{http://oeis.org/A198685} in the ``On-Line Encyclopedia of Integer Sequences'' (the first and the second components).} The full definition is in Section \ref{sec:Palindromes}, while the special case experimentally studied in this paper is defined as follows:
\begin{center}
  In $\vdwpd(2;3,t) = (p,q)$,\\
  the number $q$ is the smallest number such that for all $n \geq q$ and\\
  for all $P_0, P_1 \subseteq \set{1,\dots,n}$ with $P_0 \cap P_1 = \emptyset$ and $P_0 \cup P_1 = \set{1,\dots,n}$ with the property,\\
  that for all $v \in \set{1,\dots,n}$ we have $v \in P_0 \Leftrightarrow n+1-v \in P_0$ and $v \in P_1 \Leftrightarrow n+1-v \in P_1$,\\
  either $P_0$ has an arithmetic progression of size $3$ or $P_1$ has an arithmetic progression of size $t$, or both.

  While $p$ is the largest number such that for all $n \leq p$ and for all such $(P_0, P_1)$\\
  neither $P_0$ has an arithmetic progression of size $3$ nor $P_1$ has an arithmetic progression of size $t$.
\end{center}
In the ordinary case of plain partitions (without the additional symmetry condition) we have $p+1 = q$, and thus one uses just one number (instead of a pair), however here we can have a ``palindromic span'', that is, $p+1 < q$ can happen for the palindromic case. The reason is that from a good partition of $\set{1,\dots,n}$ we obtain a good partition of $\set{1,\dots,n-1}$ by simple removing $n$, however for ``good palindromic partitions'' besides removing $n$ we also need to remove the corresponding vertex $1$ (due to the palindromicity condition).

Apparently the most advanced special algorithm (and implementation) for computing (mixed) van der Waerden numbers is the algorithm/implementation developed in \cite{Schweitzer2009Ramsey}. For computing $\vdw(2;3,17) = 279$, with this special algorithm a run-time of 552 days is reported (page 113); the machine used should be at most $30\%$ slower than the machine used in our experiments, and so this should translate into at least 400 days on our machine. As we can see in Table \ref{tab:complsolvervdw}, the \ntawsolver{} used is 85-times faster, while Table \ref{tab:CCvdW} shows, that \cubeconq{} is around 40-times faster. These algorithms know nothing about the specific problem, and are just given the generic SAT formulation of the underlying hypergraph colouring problem. So it seems that SAT solving does a good job here.\footnote{As discussed in Subsection \ref{sec:tawsbasic}, for enumerating all solutions for $n=\vdw(2;3,17)-1 = 278$ with \ntawsolver{} we need at most the time needed for determining unsatisfiability; in this special case we have actually precisely one solution.}

\subsection{Using SAT solvers}
\label{sec:usingsat}

As explored in Dransfield et al.\ \cite{dransfield}, Herwig et al.\ \cite{herwig}, Kouril \cite{KourilPaulW26,Kouril2012W34}, Ahmed \cite{ahmed2009, ahmed2010}, and Kullmann \cite{kullmann0,kullmann}, we can generate an instance $F(t_0,\dots,t_{k-1};n)$ of the satisfiability problem (for definition, see any of the above references) corresponding to $\vdw(k; t_0, t_1,\ldots, t_{k-1})$ and integer $n$, such that $F(t_0,\dots,t_{k-1};n)$ is satisfiable if and only if $n<\vdw(k; t_0, t_1,\ldots, t_{k-1})$.
In particular, the instance $F(3,t;n)$ corresponding to $\vdw(2; 3, t)$ with $n$ variables consists of the following clauses:
\begin{itemize}
\item[(a)] $\set{x_{a}, x_{a+d}, x_{a+2d}}$ with
         $a\geq{1}, d\geq{1}, a+2d\leq{n}$, and
\item[(b)] $\set{\ol{x}_{a}, \ol{x}_{a+d}, \cdots, \ol{x}_{a+d(t-1)}}$ with
         $a\geq{1}, d\geq{1}, a+d(t-1)\leq{n}$,
\end{itemize}
where an assignment $x_i = \varepsilon$ encodes $i \in P_{\varepsilon}$ for $\varepsilon \in \set{0,1}$
(if $x_i$ is not assigned but the formula is satisfied, then $i$ can be arbitrarily placed in either of the
blocks of the partition). The (``positive'') clauses (a) (consisting only of variables), constructed from all arithmetic progressions of length $3$ in $\set{1,\dots,n}$, prohibit the existence of an arithmetic progression of length $3$ in $P_0$. And the (``negative'') clauses (b) (consisting only of negated variables), constructed from all arithmetic progressions of length $t$ in $\set{1,\dots,n}$, prohibit the existence of an arithmetic progression of length $t$ in $P_1$. To check the satisfiability of the generated instance, we need to use a ``SAT solver''. A complete SAT solver finds a satisfying assignment if one exists, and otherwise correctly says that no satisfying assignment exists and the formula is unsatisfiable. One of the earliest complete algorithms is the DLL algorithm (\cite{dpll}), and our algorithm for computing $\vdw(2; 3, 19) \leq 349$, discussed in Section \ref{sec:tawsolver}, actually implements this very basic scheme, using modern heuristics.

SAT solving has progressed much beyond this simple algorithm, and the handbook \cite{2008HandbuchSAT} gives an overview (where \cite{Zha09HBSAT} discusses some applications of SAT to combinatorics). There in \cite{DP09HBSAT} we find a general overview on complete SAT algorithms, while \cite{KSS09HBSAT} gives an overview on incomplete algorithms. For complete algorithms especially the algorithms derived from the DLL algorithm are of importance, and there are two families, namely the (earlier) ``look-ahead solvers'' outlined in \cite{HvM09HBSAT}, and the (later) ``conflict-driven solvers'' (or ``CDCL'' like ``conflict-driven clause-learning'') outlined in \cite{MSLM09HBSAT}. In Section \ref{sec:remsat} we will discuss how general SAT solvers perform on the problems from this article. The motivation for our choice of the most basic DLL algorithm for tackling the unsatisfiability of the instance $F(3,19; 349)$, already employed in \cite{ahmed2010} and discussed in Subsection \ref{sec:comp349}, is, that on these special problems classes this basic algorithm together with a modern heuristic is very competitive --- best on ordinary problem instances, and beaten on palindromic instances only by the the \cubeconq{} method.\footnote{The \cubeconq{} method, developed originally on the instances of this article, combines a look-ahead solver with a conflict-driven solver, and is faster by a factor of two on palindromic instances.} And then it is also instructive to use such an algorithm, which due to its simplicity might enable greater insight. Another advantage of its simplicity is, that it can also count and enumerate the solutions, but in this article we focus mostly on mere SAT solving; see \cite{GSS09HBSAT} for an overview on counting solutions.

Local-search based incomplete algorithms (see Ubcsat-suite \cite{ubcsat})
are generally faster than a DLL-like algorithm in finding a satisfying assignment (on such combinatorial problems), and this is also the case for the instances of this article. However they may fail to deliver a satisfying assignment when there exists one, and they can not prove unsatisfiability. If they succeed on our instances, then they deliver a good partition, and thus a lower bound for a certain van der Waerden number. So such incomplete algorithms are used for obtaining good partitions and improving lower bounds of van der Waerden numbers. When they fail to improve the lower bound any further, we need to turn to a complete algorithm.

\subsubsection{Informed versus uninformed SAT solving}
\label{sec:infvsuninf}

We use general SAT solvers, and the new solvers developed by us are also general SAT solvers, which can run without modification on any SAT problem; these solvers just run on the naked and natural SAT formulation of the problem, without giving them further information. More specifically, to show unsatisfiability we have developed the \tawsolver{} (Section \ref{sec:tawsolver}) and the \cubeconq-method (Subsection \ref{sec:cubeconq}), while to find satisfying assignments we have selected local-search algorithms (Subsection \ref{sec:remsatincomp}).

On the other end of the spectrum is \cite{KourilPaulW26,Kouril2012W34}, which uses a highly specialised method, which involves a variety of specialised SAT solvers on specialised hardware, in combination with some special insights into the problem domain. For finding satisfying assignments we have the methods developed in \cite{herwig,HeuleWalsh2010SolutionsSymmetry,HeuleWalsh2010SolutionsSymmetryW}. For more examples on informed search to compute van der Waerden numbers, see also Section 2 of \cite{ahmed2013}.

Our ``uninformed approach'' has stronger bearings on general SAT solving, while the informed approach can be more efficient for producing numerical results (however it seems to need a lot of effort to beat general SAT solvers (by specialised SAT solvers); as we have already reported, our general methods are at least on the instances of this paper much faster than the dedicated (non-SAT-based) method in  \cite{Schweitzer2009Ramsey}).

\subsubsection{Parallel/distributed SAT solving}
\label{sec:parallelSAT}

The problems we consider are computationally hard, and for the hardest of them in this paper, computation of $\vdw(2; 3, 19) = 349$, a single processor, even when run for a long time,
is not enough. Hence some form of parallelisation or distribution of the work is needed.
Four levels of parallelisation have been considered for general-purpose SAT solving (in a variety of schemes):
\begin{enumerate}
\item[(i)] Processor-level parallelisation: This helps only for very special algorithms,
and can only achieve some relatively small speed-up; see \cite{HeuleMaaren2008Parallel} for an
example which exploits parallel bit-operations. It seems to play no role for the problems we are considering.
\item[(ii)] Computer-level parallelisation: Here it is exploited that currently a single (standard)
computer can contain up to, say, 16 relatively independent processing units, working on shared memory.
So threads (or processes) can run in parallel, using one (or more) of the following general forms of collaboration:
  \begin{enumerate}
  \item[(a)] Partitioning the work via partitioning the instance (see below); \cite{ZBH96,JurkowiakLiUtard2005PSatz} are ``classical'' examples.
  \item[(b)] Using the same algorithm running in various nodes on the same problem,
exploiting randomisation and/or sharing of learned results; see
\cite{HyvaerinenJunttilaNiemelae2009ClauseLearning,GuoHamadiJabbourSais2010DiversificationIntensification} for recent examples.
  \item[(c)] Using some portfolio approach, running different algorithms on the same problem,
exploiting that various algorithms can behave very differently and unpredictably; see \cite{Gu1999} for the first example.
  \end{enumerate}
  Often these approaches are combined in various ways; see
\cite{SchubertLewisBecker2009PaMiraXT, GilFloresSilveira2009PMSat, HyvaerinenJunttilaNiemelae2009PartitioningSearchSpaces, HyvaerinenJunttilaNiemelae2010PartitioningInstances}
for recent examples. Approaches (b) and (c) do not seem to be of much use for the well-specified problem domain of hard instances from Ramsey theory.
Only (a) is relevant, but in a more extreme form (see below). In the context of (ii), still only relatively ``easy'' problems (compared to the hard problems from Ramsey theory) are tackled.
\item[(iii)] Parallelisation on a cluster of computers: Here up to, say, 100 computers are considered,
with restricted communication (though typically still non-trivial). In this case, the approach (ii)(a) becomes more dominant,
but other considerations of (ii) are still relevant. For hard problems this form of computation is a common approach.
\item[(iv)] Internet computation, with completely independent computers, and only very basic communication between the centre and the machines: In principle, the number of computers is unbounded. Since progress must be guaranteed, and the instances for which Internet computation is applied would be very hard, at the global level only (ii)(a) is applicable (while at a local level all the other schemes can in principle be applied). Yet there is no real example for a SAT computation at this level.
\end{enumerate}
We remark that the classical area of ``high performance computing'' seems to be of no relevance for SAT solving,
since the basic SAT algorithms like unit-clause propagation are, different from typical forms of numerical computation, inherently sequential (compare also our remarks to (i)). However using dedicated hardware with specialised algorithms has been utilised in \cite{KourilPaulW26,Kouril2012W34}, yielding the currently most efficient machinery for computing van der Waerden numbers.

A major advantage of the DLL solver architecture (which has been further developed into so-called ``look-ahead'' SAT solvers) is that the computation is easily parallelisable and distributable: Just compute the tree only up to a certain depth $d$, and solve the (up to) $2^d$ sub-problems at level $d$. Only minimal interaction is required: The sub-problems are solved independently, and in case one sub-problem has been found satisfiable, then the whole search can be aborted (for the purpose of mere SAT-solving; for counting all solutions of course the search needs to be completed). And the sub-problems are accessible via the partial assignment constituting the path from the root to the corresponding leaf, and thus also require only small storage space. This is the core of method (ii)(a) from above, and will be further considered in Subsection \ref{sec:comp349} (for our special example class).

In the subsequent subsection we will discuss the general merits of applying SAT solving to (hard) Ramsey problems. One spin-off of this combination lies in pushing the frontier of large computations. As a first example we have developed in \cite{HeuleKullmannWieringaBiere2011Cubism,TakHeuleBiere2012CC}, motivated by the considerations of the present article, an improved method for (ii)(a) called ``\cubeconq'', which is also relevant for industrial problems (typically from the verification area). One aspect exploited here is that for extremely hard problems, splitting into millions of sub-instances is needed. In the literature until now (see above for examples) only splitting as required, by at most hundreds of processors, has been performed, while it turned out that the above ``extreme splitting'', when combined with ``modern'' (CDCL) SAT solvers, is even beneficial when considered as a (hybrid) solver on a single-processor, and this for a large range of problem instances.

\subsubsection{Synergies between Ramsey theory and SAT}
\label{sec:merits}

For Ramsey-numbers (see \cite{Radziszowski2006RamseySurvey} for an overview on exact results),
relatively precise asymptotic bounds exist, and due to the inherent symmetry, relatively
specialised methods for solving concrete instances have an advantage. Van-der-Waerden-like numbers
seem harder to tackle, both asymptotically and exactly, and perhaps the only way ever to know the
precise values is by computation (and perhaps this is also true for Ramsey-numbers,
only more structures are to be exploited). SAT solvers are especially suited for the task,
since the computational problems are hypergraph-colouring problems, which, at least for two colours,
have canonical translations into SAT problems (as only considered in this paper).
For more colours, see the approach started in \cite{kullmann}, while for a general theory of multi-valued SAT close to
hypergraph-colouring, see \cite{Kullmann2007ClausalFormZI,Kullmann2007ClausalFormZII}.

Through applying and improving SAT solvers (as in the present article), Ramsey theory itself acquires an applied side.
Perhaps unknown to many mathematicians is the fact, that whenever for example a recent microchip is employed, this likely involves SAT solving, playing an important (though typically hidden) role in its development, by providing the underlying ``engines'' for its verification; see the recent handbook \cite{2008HandbuchSAT} to get some impression of this astounding development.
Now we believe that problem instances from Ramsey theory are good benchmarks, serving to improve SAT solvers on hard instances:

\begin{itemize}
\item Unlike with random instances (see \cite{Ach09HBSAT} for an overview), instances from Ramsey theory are ``structured'' in various ways. One special structure which one finds in all these instances is that they are layered by the number of vertices (the same structural pattern is repeated again and again, on growing scales).
\item A major advantage of random instances is their scalability, that is, we can create relatively easily instances of the same ``structure'' and different sizes. With instances from Ramsey theory we can also vary the parameters, however due to the possibly large and unknown growth of Ramsey-like numbers, controlling satisfiability and hardness is more complicated here. This possible disadvantage can be overcome through computational studies like in this paper,
which serve to calibrate the scale via precise numerical data, so that the field of SAT instances from Ramsey-theory becomes accessible (one knows for initial parameter values the satisfiability status and (apparent) solving complexity, and gets a feeling what happens beyond that).
\item In this paper, we consider two instance classes: instances related to ordinary van der Waerden numbers $\vdw(2;3,t)$
and instances related to the palindromic forms $\vdwpd(2;3,t)$. Now already with these two classes, the two main types of complete
SAT solvers, ``look-ahead'' (see \cite{HvM09HBSAT}) and ``conflict-driven'' (see \cite{MSLM09HBSAT}), are covered in the
sense that they dominate on one class each (and are (relatively) efficient); see Section \ref{sec:remsat} for further details.
On the other hand, for random instances only look-ahead solvers are efficient (for complete solvers).
\item Especially for local-search methods (see \cite{KSS09HBSAT} for an overview), these problems are hard,
but not overwhelmingly so (for the ranges considered), and thus all the given lower bounds can trigger further progress
(and insight) into the solution process in a relatively simple engineering-like manner (by studying which algorithms work best where).
\item On the other hand, for upper bounds we need to show unsatisfiability, which is much harder (we can only solve much smaller instances). All applications of SAT solving
in hardware verification are ``unsatisfiability-driven'' (see \cite{Bie09HBSAT,Kro09HBSAT} for introductions). So future progress
in solving hard Ramsey instances might trigger a breakthrough in tackling unsatisfiability,
and should then also improve these industrial applications.
\end{itemize}
We believe that for better SAT solving, established hard problem instances are needed in a great variety, and we
believe that Ramsey theory offers this potential. To begin the process of applying Ramsey theory in this direction,
problem instances from this paper (as well as related to \cite{kullmann}) have been used in the SAT 2011 competition
(\url{http://www.satcompetition.org/2011/}). As already mentioned in the previous subsection, the first fruits of the
collaboration between SAT and Ramsey theory appeared in \cite{HeuleKullmannWieringaBiere2011Cubism,TakHeuleBiere2012CC},
yielding a method for tackling hard problems with strong scalability.

Finally, the interaction between Ramsey theory and SAT should yield new insights for Ramsey theory itself:
\begin{enumerate}
\item The numerical data can yield conjectures on growth rates; see Subsection \ref{sec:conjupb}.
\item The good partitions found can yield conjectures on patterns; see Section \ref{sec-pat}.
\item New forms of Ramsey problems can be found through algorithmic considerations; see Section \ref{sec:Palindromes}.
\item The SAT solving process, considered \emph{in detail}, acts like a microscope, enabling insights into the structure of
the problem instances which are out of sight for Ramsey theory yet. For approaches towards structures in SAT instances,
which we hope to study in the future, see \cite{SS09HBSAT,Kullmann2007HandbuchMU}.
\end{enumerate}

\subsection{The results of this paper}
\label{sec:results}

In Section \ref{sec:tawsolver}, we present the new SAT solver, \ntawsolver, with superior performance on the instances considered in this paper (only for palindromic instances the new hybrid method \cubeconq{} is superior).
Section \ref{sec-comp} contains our results on the numbers $\vdw(2; 3, t)$. We discuss the computation of the one new van der Waerden number, and present further conjectures regarding precise values\footnote{to establish these conjectures will require major advances in SAT solving} and the growth rate. In Section \ref{sec-pat}, we investigate some patterns we found in the good partitions (establishing the lower bounds). In Section \ref{sec:Palindromes}, we introduce palindromic problems and the corresponding palindromic number-pairs. Finally in Section \ref{sec:remsat}, we discuss the observations on the use of the various SAT solvers involved.

In this paper, we represent partitions of $\vdw(2; 3, t)$ as bitstrings.
For example, the partition $P_0=\set{1,4,5,8}$ and $P_1=\set{2,3,6,7}$, which is an example of a good partition of $\set{1,2,\dots,8}$, where $8 = \vdw(2; 3, 3)-1$, is represented as $01100110$, or more compactly as $01^20^21^20$, using exponentiation to denote repetition of bits.

\section{The \tawsolver}
\label{sec:tawsolver}

We now discuss the \tawsolver, an open-source SAT solver, created by the first author with a special focus on van der Waerden problems (version 1.0), and improved by the second author through an improved branching heuristic (version 2.6).\footnote{\url{http://sourceforge.net/projects/tawsolver/}, and in the \texttt{OKlibrary}:\\\url{https://github.com/OKullmann/oklibrary/blob/master/Satisfiability/Solvers/TawSolver/tawSolver.cpp}} Algorithm \ref{alg-revised-dpll} shows that the basic algorithm of the \tawsolver{} is the simplest possible (reasonable) DLL-scheme, just branching on a variable plus unit-clause propagation. As we can see in Section \ref{sec:remsat}, it is the strongest SAT solver on the instances considered in this paper, only beaten on palindromic problems by the new hybrid scheme \cubeconq, which came out as a result on research on the instances of this paper.

\subsection{The basic structure}
\label{sec:tawsbasic}

Algorithm \ref{alg-revised-dpll} specifies the \tawsolver, which for input $F$ (a formula or ``clause-set'') decides satisfiability:
\begin{enumerate}
\item Lines 3-5 is ``unit-clause propagation'' (UCP), denoted by the function $r_1$, which sets literals $x$ in the current $F$ to true while there are unit-clauses $\set{x} \in F$.
  \begin{enumerate}
  \item Setting a literal $x$ to true in a clause-set $F$ is performed by first removing all clauses from $F$ containing $x$, and removing the element $\ol{x}$ from the remaining clauses.
\item $r_1$ finds a contradiction (Line 4) by finding two unit-clauses $\set{v}$ and $\set{\ol{v}}$ (i.e., $v \wedge \neg v$).
\item While $r_1$ finds a satisfying assignment (Line 5) if all clauses vanished (have been satisfied).
  \end{enumerate}
\item Lines 6-7 give the branching heuristic, which yields the branching literal $x$, first set to true, then to false, in the recursive call of the \tawsolver.
  \begin{enumerate}
  \item $p(a,b) \in \RR_{>0}$ for $a,b \in \RR_{>0}$ in Line 6 is the ``projection'', and we consider three choices $p_+, p_*, p_{\tau}$.
  \item $w_F(x)$ for literal $x$ is a heuristical value, measuring in a sense the ``progress achieved'' when setting $x$ to FALSE (``progress'' in the sense of the instance becoming more constrained, so that more unit-clause propagations are to be expected).
  \item The details are specified in Subsections \ref{sec:tawsolver1020}, \ref{sec:tauproj}.
  \end{enumerate}
\item The implementation is discussed in Subsection \ref{sec:tawsolverimpl}.
\item The tree of recursive calls made by the solver is called the {\it DLL-tree} of $F$.
\end{enumerate}
Besides the choice of the heuristic, this is the basic SAT solver as published in \cite{dpll}. The implementation is optimised for the needs of the branching heuristic, which requires to know from each (original) clause in the input $F$ whether it has been satisfied meanwhile, and if not, what is its current length.

\begin{algorithm}[H]
\caption{\tawsolver}
\begin{algorithmic}[1]
  \State Global variable $F$, initialised by the input.
  \Function{DLL}{\ } : returns SAT or UNSAT for the current $F$
  \State Update $F$ to $r_1(F)$
  \State If contradiction found via $r_1$, then goto 12
  \State If satisfying assignment found via $r_1$, then return SAT
  \State Choose variable $v$ with maximal $p(w_F(v), w_F(\ol{v}))$
  \State If $w_F(v) \ge w_F(\ol{v})$, then $x := v$, else $x := \ol{v}$
  \State Set $x$ to TRUE in $F$; if \Call{DLL}{\ } = SAT, then return SAT
  \State Undo assignment of $x$
  \State Set $\ol{x}$ to TRUE in $F$; if \Call{DLL}{\ } = SAT, then return SAT
  \State Undo assignment of $\ol{x}$
  \State Undo assignments made by $r_1$
  \State Return UNSAT
\EndFunction
\end{algorithmic}
\label{alg-revised-dpll}
\end{algorithm}

With a small modification, namely just continuing when a satisfying assignment was found, the \tawsolver{} can also count all satisfying assignments, or output them; this is available as a compile-time option for the solver. In Section \ref{sec-pat}, we will discuss some patterns which we found in satisfying assignments for $F(3,t;n)$ with $n < \vdw(2;3,t)$. We do not report run-times for determining (or counting) all solutions in Section \ref{sec:remsat}, but for $n=\vdw(2;3,t)-1$ (empirically) the run-time is at most the run-time needed to determine unsatisfiability for $n=\vdw(2;3,t)$; for numerical values of solution-counts see \cite{Kouril2012W34}.

\subsection{Look-ahead solvers}
\label{sec:lasolvers}

It is useful for the general picture to consider the general $r_k$-operations, as introduced in \cite{Ku99b} and further studied in \cite{GwynneKullmann2012SlurSOFSEM,GwynneKullmann2012SlurJ}. These operations transform a clause-set $F$ into a satisfiability-equivalent clause-set via application of some forced assignments (i.e., where the opposite assignments would yield an unsatisfiable clause-set). Let $\bot$ be the empty clause, which stands for a trivial contradiction. $r_0$ just maps $F$ to $\set{\bot}$ in case of $\bot \in F$, while otherwise $F$ is left unchanged. Now we can recognise $r_1$ as an operation which is applied recursively to the result of $F$ with literal $x$ set to true if setting $\ol{x}$ to true yields $\set{\bot}$ via $r_0$. This scheme yields also the general $r_k$ for $k \in \NN$: as long as there is a literal $x$ such that $F$ with $\ol{x}$ set to true yields $\set{\bot}$ via $r_{k-1}$, set $x$ to true and iterate. The final result, denoted by $r_k(F)$, is uniquely determined. Besides the ubiquitous unit-clause propagation $r_1$ also $r_2$, called ``failed literal elimination'', is popular for SAT solving, and even $r_3$, typically called ``double look-ahead'', is used in some solvers (always partially, testing the reductions only for selected variables).

The general scheme for a look-ahead solver (as stipulated in \cite{Kullmann2007HandbuchTau}) now generalises the DLL-procedure from Algorithm \ref{alg-revised-dpll}, by replacing the reduction $F \leadsto r_1(F)$ in Line 3 by the general $F \leadsto r_k(F)$ for some $k \ge 1$. Furthermore, for the inspection of a branching variable and the computation of the heuristical values $w(v)$ and $w(\ol{v})$, now the effects of setting $v$ resp.\ $\ol{v}$ to true and performing $r_{k-1}$ reduction are considered. This explains also the notion of ``look-ahead'': the $r_k$-reduction can be partially achieved at the time when running through all variables $v$, setting $v$ resp.\ $\ol{v}$ to true and applying $r_{k-1}$ --- if this yields $\set{\bot}$, then performing the opposite assignment is justified. Since $r_1$ is the standard for reduction of a branch, (partial) $r_2$ is the default for the reduction at a node.\footnote{The look-ahead solvers \satz{} and \march{} run through the variables once (actually also only considering ``interesting'' variables by some criterion), and so they do not compute $r_2$, but only an approximation. The only solver to completely compute $r_2$ is the \oksolver (while \satz{} and \march{} search also for some $r_3$ reductions on selected variables).}

We see that \tawsolver{} uses $k=1$ (so the ``look-ahead'' uses $k=0$, and in this sense \tawsolver{} is a ``look-ahead solver with zero look-ahead''). The prototypical solver for using $k=2$ is the \oksolver{} (\cite{Ku2002h}). In a rather precise sense the \tawsolver{} can be considered at the level-1-version of the \oksolver{} (or the latter as the level-2-version of the \tawsolver). Also for the branching heuristic, which is discussed in the following subsection, \tawsolver{} uses the same scheme as the \oksolver{}, appropriately simplified to the lower level. Both \tawsolver{} and \oksolver{} are solvers with a ``mathematical meaning'', precisely implementing an algorithm to full extent, with the only magical numbers the clause-weights used in the branching heuristic.

The general scheme for the branching heuristic of a look-ahead solver, as developed in \cite{Kullmann2007HandbuchTau} (Subsection 7.7.2), is as follows: For a clause-set $F$ and its direct successor $F'$ on a branch (applying the branching assignment and further reductions), a ``distance measure'' $d(F,F') \in \RR_{>0}$ is chosen, with the meaning the bigger this distance, the larger the decrease in complexity. The branching heuristic considers for each variable $v$ its two successor $F', F''$ and computes the distances $d(F,F'), d(F,F'')$. Then via a ``projection'' $p: \RR_{>0}^2 \rightarrow \RR_{>0}$ one heuristical value $h_v := p(d(F,F'), d(F,F''))$ is obtained. Finally some $v$ with maximal $h_v$ is chosen. Choosing which of $v$ or $\ol{v}$ to be processed first (important for satisfiable instances) is done via a second heuristic, estimating the satisfiability-probabilities of $F', F''$ in some way.

\subsection{From \otawsolver{} to \ntawsolver}
\label{sec:tawsolver1020}

We are now turning to the discussion of the branching heuristic in \ntawsolver{} (lines 6, 7 in Algorithm \ref{alg-revised-dpll}), the version developed for this article. For \otawsolver{} (used in \cite{ahmed2009,ahmed2010}) the ``Two-sided Jeroslaw-Wang'' (2sJW) rule by Hooker and Vinay \cite{hv95} was used, which chooses $v$ such that the weighted sum of the number of clauses of $F$ containing $v$ is maximal, where the weight of a clause of length $k$ is $2^{-k}$.\footnote{We do not care much here about the order of branching, since the algorithm is only effective on unsatisfiable problems, where the order does not matter (while on satisfiable problems local search is much faster).} As discussed in \cite{Kullmann2007HandbuchTau}, the ideas from \cite{hv95} are actually rather misleading, and this is demonstrated here again by obtaining a large speed-up through the replacement of the branching heuristic, as can be seen by the data in Section \ref{sec:remsat} (comparing \otawsolver{} with \ntawsolver).

For a literal $x$, a clause-set $F$ and $k \in \NN$ let $\ldeg_F^k(x) := \abs{\set{C \in F : x \in C \wedge \abs{C} = k}}$ be the ``literal degree'' of $x$ in the $k$-clauses of $F$. The 2sJW-rule consists of three components:
\begin{enumerate}
\item The weight $w_F(x)$ of literal $x$ is set as $w_F(x) := \sum_k 2^{-k} \cdot \ldeg_F^k(x)$.
\item A variable $v$ with maximal $p_+(w_F(v), w_F(\ol{v}))$ for $p_+(a,b) := a+b$ is chosen.
\item The literal $x \in \set{v,\ol{v}}$ to be set first to true is given by the condition $w_F(x) \ge w_F(\ol{x})$.
\end{enumerate}
This approach has the following fundamental flaws:
\begin{enumerate}
\item The choice of the first branch ($v$ or $\ol{v}$) is mixed up with the choice of $v$ itself, but very different heuristics are needed:
  \begin{enumerate}
  \item For the choice of the first branch, some form of approximated \emph{satisfiability}-probability must be maximised,
  \item while the branching-variable must minimise some approximated tree-size for the worst case, the \emph{unsatisfiable} case.
  \end{enumerate}
  In 2sJW the weights $2^{-k}$ are only motivated by satisfiability-probabilities, but are used for the choice of $v$ itself.
\item Once total weights $w_F(v), w_F(\ol{v})$ have been determined, one number (the projection) must be computed from this (to be maximised). 2sJW uses the sum, which, as demonstrated in \cite{Kullmann2007HandbuchTau}, corresponds to minimising a \emph{lower bound} on the DLL-tree-size --- much better is the product $p_*(a,b) := a \cdot b$, which corresponds to minimising an \emph{upper bound} on the tree-size.
\end{enumerate}
So the improved heuristic (which nowadays, when extended appropriately to take the look-ahead into account, is the basis for all look-ahead solvers) chooses clause-weights $w_2, w_3, \dots \in \RR_{>0}$, from which the total weight
\begin{displaymath}
 w_F(x) := \sum_k w_k \cdot \ldeg_F^k(x)
\end{displaymath}
is determined, and chooses a variable $v$ with maximal
\begin{displaymath}
  p_*(w_F(v), w_F(\ol{v})) = w_F(v) \cdot w_F(\ol{v}).
\end{displaymath}
The meaning of these weights is completely different from the argumentation in \cite{hv95}: as mentioned, satisfiability-probabilities have no place here. The underlying distance measure is $\sum_k w_k \cdot \nu^k(F')$, where $F'$ is the resulting clause-set after performing the branch-assignment and the subsequent $r_k$-reduction, while $\nu^k(F')$ is the number of \emph{new} $k$-clauses in $F'$. When setting literal $x$ to true, then $\ldeg_F^k(\ol{x})$ is an ``approximation'' of the number of new clauses of length $k-1$ (since in the clauses containing $\ol{x}$ this literal is removed).

The weights $w_k^{\mathrm{OK}}$ for the \oksolver{} have been experimentally determined as roughly $5^{-k}$. Since the value of the first weight is arbitrary, the weights are rescaled to $w_2^{\mathrm{OK}} = 1$, obtaining then each new weight by multiplication with $1/5$. Now $w_2$ for the \tawsolver{} is a stand-in for the number of new 1-clauses, which are handled in the \oksolver{} by the look-ahead; accordingly it seems plausible that now $w_2$ needs a relatively higher weight. We rescale here the weights to $w_3 = 1$ (note that for the \tawsolver{} the weight $w_k$ concerns new clauses of length $k-1$). Empirically we determined $w_2 = 4.85$, $w_4 = 0.354$, $w_5 = 0.11$, $w_6 = 0.0694$, and thereafter a factor of $\frac{1}{1.46}$; thus starting with $w_2$ the next weights are obtained by multiplying with (rounded) $1/4.85, 1/2.82, 1/3.22, 1/1.59, 1/1.46, \dots$.

For the choice of the first branch there are two main schemes, as discussed in \cite{Kullmann2007HandbuchTau} (Subsection 7.9). Roughly, the target now is to get rid off (satisfy) as many short clauses as possible (since shorter clauses are bigger obstructions for satisfiability).\footnote{While for a good branching variable we want to \emph{create} as many short clauses as possible (via setting literals to false)!} Both schemes amount to choose literal $x \in \set{v,\ol{v}}$ with $w'_F(x) \ge w'_F(\ol{x})$ for some weights $w'_k$. For the Franco-estimator we have $w'_k = -\log(1 - 2^{-k})$, while for the Johnson-estimator we have $w'_k = 2^{-k}$. In the \oksolver{} the Franco-estimator is used. But for the \tawsolver{} with its emphasis on unsatisfiable instances, while the computation of the heuristic is very time-sensitive (much more so than for the \oksolver), actually just the same weights $w'_k= w_k$ are used.

As one can see from the data in Section \ref{sec:remsat}, on ordinary van der Waerden problems the new heuristic yields a reduction in the size of the DLL-tree by a factor increasing from $2$ to $5$ for for $t = 12, \dots, 16$ (comparing \ntawsolver{} with \otawsolver), and for palindromic problems by a factor increasing from $5$ to $20$ for $t = 17, \dots, 23$.\footnote{\ntawsolver{} additionally has the implementation improved, so that nodes are processed now twice as fast as with \otawsolver.} We do not present the data, but most of the reduction in node-count is due to the replacement of the sum as projection by the product (the optimised clause-weights only further improve the node reduction by at most $50\%$ for the biggest instances, compared with a simple but reasonable scheme like $2^{-k}$).

\subsection{The implementation}
\label{sec:tawsolverimpl}

The \tawsolver{} is written in modern C++ (C++11, to be precise), with around 1000 lines of code, with complete input- and output-facilities, error handling and various compile-time options for implementations. The code is highly optimised for run-time speed, but at the same time expressing the concepts via appropriate abstractions, relying on the expressiveness of C++ both at the abstraction- and the implementation-level, so that the compiler can do a good job producing efficient code.

Look-ahead solvers are often ``eager'', that is, they represent the clause-set at each node of the DLL-tree in such a way, that the current (``residual'') clause-set is visible to the solver, and precisely the current clauses can be accessed. On the other hand, conflict-driven solvers are all ``lazy'', that is, the initial clause-set is not updated, and the state of the current clause-set has to be inferred via the current assignment to the variables. The representation of the input clause-set $F$ by the \tawsolver{} now is ``mostly lazy'':
\begin{enumerate}
\item Assignments to variables are entered into a global array,
\item Via the usual occurrence lists, for each literal $x$ one obtains access to all the clauses  $C \in F$ with $x \in C$.
\item This representation of $F$ is static (is not updated), and in this sense we have a lazy datastructure.
\item But the status of clauses, which is either inactive (when satisfied) or active, and their length (in the active case) is handled eagerly, by storing status and length for each clause and updating this information appropriately. So at each node, when running through the occurrence lists (still as in the input), for each clause we can see directly whether the clause is active and in this case its current length.
\item When doing an assignment, then the clause-lengths are updated: if a literal is falsified in a clause, the length is decreased by one, and if a literal is satisfied, the status of the clause is set to inactive.
\item For each active clause containing a variable which is assigned, there is exactly one change (either decrease in length or going from active to inactive). This change is entered into a change-list.
\item When backtracking, then the assignment is simply undone by going through the change-list in reverse order, and undoing the changes to the clauses.
\end{enumerate}

No counters are maintained for the literal degrees $\ldeg^k(x)$. Instead, the heuristic is computed by running through all literal occurrences in the original input for the unassigned literals, and adding the contributions of the clauses which are still active (this is the use of maintaining the length of a clause).

When doing unit-clause propagation, the basic choice is whether performing a BFS search, by using a first-in-first-out strategy for the processing of derived unit-clauses, of a DFS search, using a last-in-first-out strategy. BFS is slightly easier to implement, but on the palindromic vdW-instances needs roughly $10\%$ more unit-clauses to propagate\footnote{the final result is uniquely determined, but in general there are many ways to get there}, while on ordinary vdW-instances it uses less propagations, though the difference is less than $2\%$, and thus DFS is the default. This can also be motivated by the consideration that newly derived unit-clauses can be considered to be ``more expensive'', and thus should be treated as soon as possible.

Look-ahead solvers in general rely on the distance for branch-evaluation to be positive, while a zero value should indicate that a special reduction can be performed. And indeed, when counting new clauses, then the weighted sum being zero means that an autarky has been found, a partial assignment not creating new clauses, which means that all touched clauses are satisfied; see \cite{Kullmann2007HandbuchMU}.\footnote{The point about autarkies is that they can be applied satisfiability-equivalently.} Thus starting with the \oksolver{}, look-ahead solvers looked out for such autarkies, and applied them when found (\cite{HvM09HBSAT,Kullmann2007HandbuchTau}). Now for a zero-look-ahead solver like the \tawsolver, these autarkies are just pure literals (only occurring in one sign, not in the other). Their elimination causes a slight run-time increase, without changing much anything else, and so by default they are not eliminated but not chosen for branching (if there are still non-pure literals).

\subsection{The optimal projection: the $\tau$-function}
\label{sec:tauproj}

In \cite{Kullmann2007HandbuchTau} it is shown that the $\tau$-function is the best generic projection in the following sense:
\begin{itemize}
\item The $\tau$-function is defined for arbitrary tuples $a \in \RR^n$, $n \in \NN$, namely $\tau(a) \in \RR_{>0}$ is the unique $x \ge 1$ such that $\sum_{i=1}^n x^{-a_i} = 1$.
\item This projection induces a linear order on the set of all such ``branching tuples'' $a$ (of arbitrary length) by defining $a \le b$ if $\tau(a) \le \tau(b)$; here ``$a \le b$'' means that $a$ is better than $b$.
\item Theorem 7.5.3 in \cite{Kullmann2007HandbuchTau} shows that when imposing some general consistency-constraints on the comparison of branching tuples (where it is of importance that branching tuples can have arbitrary length), then there is precisely one such linear order on the set of branching tuples, namely the one induced by $\tau$.
\end{itemize}
Now specific solvers might have a special built-in bias, and, more importantly, the theorem is not applicable when considering only branching tuples of length $2$ (as it is the case for ordinary boolean SAT solving). But nevertheless, considering the $\tau$-function as projection (more precisely, since we maximised projection values, $1/\tau$ is used) is an interesting option, and leads to the \ttawsolver{} (with ``$\tau$'' in place of ``\texttt{t}''):
\begin{displaymath}
  p_{\tau}(w_F(v), w_F(\ol{v})) := 1 / \tau(w_F(v), w_F(\ol{v})).
\end{displaymath}
In this context it makes sense to definitely forbid distance-values $0$, and thus pure literals are now eliminated.\footnote{That is, only eliminating those literals (by setting $\ol{x}$ to true) with $w_F(x) = 0$; these eliminations might create further pure literals, which will be eliminated when in the child node the branching variable is computed, and so on.} In Section \ref{sec:remsat} we see that \ttawsolver{} is faster than \ntawsolver{} on large palindromic problems due to a much reduced node-count, but on ordinary problems the node-count stays basically the same, and then the overhead for computing $p_{\tau}$ makes the \ttawsolver{} slower.

The weights for \ttawsolver{} have been empirically determined as $w_2 = 7$, $w_4 = 0.31$, $w_5 = 0.19$, and then a factor of $\frac{1}{1.7}$; so starting with $w_2$ the next weights are obtained by multiplying with $1/7, 1/3.22, 1/1.63, 1/1.7, \dots$.\footnote{We consider the values for the weights as reasonable all-round values. A deeper understanding, based on the theory developed in \cite{Kullmann2007HandbuchTau}, is left for future investigations.}

\section{Computational results on $\vdw(2; 3, t)$}
\label{sec-comp}

This section is concerned with the numbers $\vdw(2; 3, t)$. The discussion of the computation of $\vdw(2; 3, 19)$ is the subject of Subsection \ref{sec:comp349}. Conjectures on the values of $\vdw(2; 3, t)$ for $20 \leq t \leq 30$ are presented in Subsection \ref{sec:Somenewconjectures}, and also further lower bounds for $31 \leq t \leq 39$ are given there. Finally in Subsection \ref{sec:conjupb}, we update the conjecture on the (quadratic) growth of $\vdw(2; 3, t)$.

\subsection{$\vdw(2; 3, 19)=349$}
\label{sec:comp349}

The lower bound $\vdw(2; 3, 19)\geq 349$ was obtained by Kullmann \cite{kullmann0}
using local search algorithms and it could not be improved any further using these incomplete algorithms (because, as we now know, the bound is tight).
An example of a good partition of the set $\set{1,2,\ldots,348}$ is as follows:

\scriptsize
$$1^{4}01^{6}01^{18}01^{3}01^{4}01^{5}01^{4}01^{11}01^{9}01^{3}01^{6}0
1^{7}01^{5}01^{14}01^{16}0101^{2}0^{2}1^{2}01^{15}01^{4}01^{12}0$$
$$1^{15}01^{2}01^{5}01^{7}01^{10}01^{13}01^{2}01^{15}01^{12}01^{4}0
1^{15}01^{2}0^{2}1^{2}0101^{9}01^{6}01^{14}01^{5}01^{14}01^{2}.$$

\normalsize

To finish the search, i.e., to decide that a current lower bound of a certain van der
Waerden number is exact, one might require many years of CPU-time. Discovering a new van der Waerden number has always been a challenge,
as it requires to explore the search space completely, which has a size
exponential in the number of variables in the corresponding satisfiability instance.
To prove that an instance with $n$ variables is unsatisfiable, the DLL algorithm has to implicitly
enumerate all the $2^n$ cases. So the algorithm systematically explores all possible cases, however without actually
explicitly evaluating all of them --- herein lies the strength (and the challenge) for SAT solving.

In Subsection \ref{sec:parallelSAT}, we gave an overview on the area of distributing hard SAT problems from a general SAT perspective, and we are concerned here with method (ii)(a), applied to \tawsolver. We find the simplest division of the computation of the search into parts, that have no inter-process communication among themselves, together with the observation of some patterns, very successful. Namely a level (depth) $L \in \NNZ$ of the DLL-tree is chosen, where the level considers only the decisions (ignoring the variables inferred via unit-clause propagation), and the $2^L$ subtrees rooted at that level are distributed among the processors.

To show the unsatisfiability of $F(3,19;349)$, we have used \otawsolver{} and 2.2 GHz AMD Opteron 64-bit processors (200 of them) from the {\tt cirrus} cluster at Concordia University for running the
distributed branches of the DLL-tree. The value $L=8$ was chosen, splitting the search space into $2^8=256$ independent parts (subtrees) $P_0, \dots, P_{255}$. The total CPU-time of all processor together was roughly 196 years (the first part $P_0$ alone has taken roughly 60 years of CPU-time).\footnote{Comparing \otawsolver{} with \ntawsolver, as we can see in Table \ref{tab:complsolvervdw}, the series of quotients $q_i = $ old-time / new-time, for $t = 12, \dots, 16$ is (rounded) $4.3,5.6,6.8,9.4,12.8$. This can be approximated well by the law $q_{i+1} = 1.3 \cdot q_i$, which would yield for $t=19$ the factor $12.8 \cdot 1.3^3 \approx 28.1$. So we would expect with \ntawsolver{} at least a speed-up by a factor $20$, which would reduce the 200 years to 10 years. Another approximation is obtained by considering Table \ref{tab:complsolvervdw}: we see that for each step from $t$ to $t+1$ the run-time always increases by less than a factor of $10$, while for $t=17$ we use less than five days, which would yield at most 500 days for $t=19$ with \ntawsolver.} For the prediction of run-times for the sub-tasks, the following observation made in Ahmed \cite{ahmed2010} was used. Recall that for \otawsolver{} (Algorithm \ref{alg-revised-dpll}) the branching rule was to select a variable with maximal $w_F(v) + w_F(\ol{v}) = \sum_k (\ldeg_F(v) + \ldeg_F(\ol{v})) \cdot 2^{-k}$, where for the first branch $x \in \set{v,\ol{v}}$ with $\sum_k \ldeg_F(v) \cdot 2^{-k} \geq \sum_k \ldeg_F(\ol{v})$ is chosen. Now the observation is that the parts (sub-trees of the DLL-tree) $P_0,P_1,P_2,P_4,P_8,P_{16},P_{32},P_{64},P_{128}$ are bigger than the others parts, and $P_0$ is the biggest.

Meanwhile our result $\vdw(2;3,19)=349$ has been reproduced in \cite{Kouril2012W34}, via an alternative SAT solving approach (see Subsection \ref{sec:infvsuninf}). At least at this time there seems to be no competitive alternative to SAT solving. See Section \ref{sec:remsat} for further remarks on SAT solving for these instances in general. It would be highly desirable to be able to substantially compress the resolution proofs obtained from the solver runs, so that a proof object would be obtained which could be verified by certified software (and hardware); see \cite{Cotton2010MinimizingResolution} for some recent literature.

\subsection{Some new conjectures}
\label{sec:Somenewconjectures}

In Table \ref{tab:conj3t}, we provide conjectured values of $\vdw(2; 3, t)$ for $t=20,21,\ldots,30$.
We have used the {\tt Ubcsat} suite \cite{ubcsat} of local-search based satisfiability
algorithms for generating good partitions, which provide a proof
of these lower bounds; see \ref{sec:Certificates3tp} for the certificates.
In Subsection \ref{sec:remsatincomp} we provide details of the algorithms used to find
the good partitions.
The characteristics of the searches were such that we believe these values to be optimal, namely with the right settings, these bounds can be found rather quickly, and in the past, all such conjectures turned out to be true (though, as discussed below, the situation gets weaker for $t=29,30$). However, since local search based algorithms are incomplete (they may fail to
deliver a satisfying assignment, and hence a good partition when there exists one), it remains to prove
exactness of the numbers using a complete satisfiability solver or some complete colouring algorithm.

\begin{table}[H]
  \centering
  \begin{tabular}{c|*{11}{c}}
    $t$ & 20 & 21 & 22 & 23 & 24 & 25 & 26 & 27 & 28 & 29 & 30\\
    \hline
    $\vdw(2;3,t) \geq$ & 389 & 416 & 464 & 516 & 593 & 656 & 727 & 770 & 827 & 868 & 903
  \end{tabular}
  \caption{Conjectured precise lower bounds for $\vdw(2;3,t)$}
  \label{tab:conj3t}
\end{table}

We observe that for $t=24,25,\ldots,30$ we have $\vdw(2; 3, t)>t^2$, which
refutes the possibility that $\forall\, t : \vdw(2; 3, t)\leq t^2$,
as suggested in \cite{blr2008}, based on the
exact values for $5\leq t \leq 16$ known by then. Further (strict) lower bounds we found are in Table \ref{tab:further3t} (where now we think it is likely that these bounds can be improved; see \ref{sec:Certificates3tf} for the certificates).

\begin{table}[H]
  \centering
  \begin{tabular}{c|*{9}{c}}
    $t$ & 31 & 32 & 33 & 34 & 35 & 36 & 37 & 38 & 39\\
    \hline
    $\vdw(2;3,t) >$ & 930 & 1006 & 1063 & 1143 & 1204 & 1257 & 1338 & 1378 & 1418
  \end{tabular}
  \caption{Further lower bounds for $\vdw(2;3,t)$}
  \label{tab:further3t}
\end{table}
That we conjecture the data of Table \ref{tab:conj3t} to be true, that is, that the used local-search algorithm is strong enough, while for the data of Table \ref{tab:further3t} that algorithm seems too weak to reach the solution, has the following background in the data: As we report in Subsection \ref{sec:remsatincomp}, in the range $24 \leq t \leq 33$ the local-search algorithm RoTS from the \texttt{Ubcsat} suite was found best-performing. This algorithm is used in an incremental fashion, initialising the search by known solutions for smaller $n$. This approach for $t=28$, with a cut-off $5 \cdot 10^6$ rounds, found a solution for $n=826$, and in 1000 independent runs (non-incremental) two solutions were found. But with cut-off $10^7$ in 1000 runs and cut-off $2 \cdot 10^7$ in 500 runs no solutions was found. From our experience this seems ``pretty safe'' for a conjecture. We are entering now a transition period. For $t=29$ the iterative approach with cut-off $5 \cdot 10^6$ found the solution for $n=867$, while cut-off $10^7$ found no solution for $n=868$ in 1000 runs. For $t=30$ the iterative approach managed to find a solution for $n=897$; restarting it with cut-off $10^8$ found a solution for $n=902$, while for $n=903$ no solution with that cut-off was found in 300 runs. So we see that already $t=30$ is stretching it. However for $t=31$ the iterative approach with cut-off $10^8$ only reached $n=919$ (despite restarts), while we happen to have a palindromic solution for $n=930$ (these are much easier to find; see Subsection \ref{sec:Palindromesconj}). So here now we believe we definitely over-stretched the abilities of the algorithm.

\subsection{A conjecture on the upper bound}
\label{sec:conjupb}

An important theoretical question is the growth-rate of $t \mapsto \vdw(2;3,t)$.
Although the precise relation ``$\vdw(2;r,t) \leq t^2$'' has been invalidated by our results,
quadratic growth still seems appropriate (see \cite{kullmann0} for a more general
conjecture on polynomial growth for van der Waerden numbers in certain directions of the parameter space; indeed in some directions linear growth is proven there):

\begin{conj}\label{conj-1}
There exists a constant $c>1$ such that $\vdw(2; 3, t) \leq ct^2$.
\end{conj}

See Conjecture \ref{conj-3a} for a strengthening. To determine the current best guess for $c$, and to give some heuristic justification for Conjecture \ref{conj-1}, we observe the known exact values and lower bounds, and we arrive at the following possible recursion:
\begin{displaymath}
  \vdw(2; 3, t) \leq \vdw(2; 3, t-1)+d(t-1),
\end{displaymath}
for $4 \leq t \leq 39$ and some $d > 0$, with $\vdw(2;3,3)=9$. So we make the Ansatz $\vdw(2;3,t) \leq w_t := 9 + \sum_{i=3}^{t-1} d \cdot i$, for $t \geq 3$, where $d := \max_{t=4}^{39} \frac{\vdw(2;3,t) - \vdw(2;3,t-1)}{t-1}$; in case $\vdw(2;3,t)$ is not known, we use the lower bounds from Tables \ref{tab:conj3t}, \ref{tab:further3t}. From our data we obtain $d = {593-516 \over 23} = \frac{77}{23}$ (see \ref{sec:OKlNumbers}).
We have (geometric sum) $w_t = \frac d2 t^2 - \frac 32 d t + 9 - 2d < \frac d2 t^2$, and so we obtain
\begin{displaymath}
  \vdw(2;3,t) \leq \frac d2 t^2 = \frac{77}{46} t^2 < 1.675 t^2,
\end{displaymath}
which satisfies all data regarding $\vdw(2; 3,t)$ presented so far.

\section{Patterns in the good partitions}
\label{sec-pat}

In this section, we investigate the set of all good partitions
corresponding to certain van der Waerden numbers $\vdw(2; 3, t)$ for
patterns. As described in Section \ref{sec:infvsuninf}, the motivation behind this section
is to obtain more problem-specific information on the solution-patterns, which may help to
design heuristics to reduce search-space
while computing specific van der Waerden numbers.

Let $S(t)$ denote the set of all binary strings each of which represents a good partition of the set
$\set{1,2,\ldots,\vdw(2; 3, t)-1}$.
Generating $S(t)$ involves traversing the respective search space completely.
Let $n_0(B)$, $n_1(B)$, and $n_{00}(B)$ be the number of zeros, ones, and double-zeros, respectively,
in a bitstring $B$ (note that three consecutive zeros are not possible in any $B \in S(t)$). Let $\epos(B)$ denote the sequence of powers of 1 in a bitstring $B$.
Let $\np(B)$ and $\nv(B)$ denote the number of peaks (local maxima) and valleys (local minima), respectively, in
$\epos(B)$ (not necessarily strict). For example, for the compact bitstring $1^8001^601^301^101^3001^501^801^5001^301^101^301^6001^8$ (with $n_0 = 16$, $n_1 = 60$ and $n_{00} = 4$),
we have the following $\epos$, with \textrm{p} and \textrm{v}, marking peaks and valleys, respectively,
corresponding to changes in magnitudes.

$${8 \atop{\textrm p}} {6 \atop } {3 \atop } {1 \atop {\textrm v}} {3 \atop } {5 \atop }
{8 \atop {\textrm p}} {5 \atop } {3 \atop } {1 \atop {\textrm v}} {3 \atop } {6 \atop } {8 \atop {\textrm p}}$$

And for $B = 1^101^101^201^201^301^3$ we have $n_0(B) = 5$, $n_1(B) = 12$, $n_{00}(B) = 0$, while
there is one valley followed by one peak, and thus $\nv(B) = \np(B) = 1$.

\subsection{Number of 0's and 00's}
\label{sec:mum0}

In this section, we determine the number $\min\set{n_0(B): B\in S(t)}$,
$\max\set{n_0(B): B\in S(t)}$, and $\max\set{n_{00}(B):B\in S(t)}$ for $3\leq t\leq 14$.
Observations in Table \ref{tab-zeros} lead us to Conjectures
\ref{conj-2} and \ref{conj-3}.

\begin{longtable}[c]{|c|c|c|}
\caption{Zeros in good partitions of $\set{1,2,\dots,\vdw(2; 3, t)-1}$}
\label{tab-zeros}\\
  \hline \hline
  $\vdw(2; 3, t)$ & $(\min\set{n_0(B): B\in S(t)}$, & $\max\set{n_{00}(B): B\in S(t)}$\\
 {} & $\max\set{n_0(B): B\in S(t)})$ & {} \\
  \hline
{$\vdw(2; 3, 3)$} & {(4, 4)} & {2}\\
{$\vdw(2; 3, 4)$} & {(6, 6)} & {2}\\
{$\vdw(2; 3, 5)$} & {(7, 9)} & {2}\\
{$\vdw(2; 3, 6)$} & {(8, 10)} & {4}\\
{$\vdw(2; 3, 7)$} & {(11, 12)} & {3}\\
{$\vdw(2; 3, 8)$} & {(14, 14)} & {1}\\
{$\vdw(2; 3, 9)$} & {(16, 16)} & {4}\\
{$\vdw(2; 3, 10)$} & {(19, 21)} & {5}\\
{$\vdw(2; 3, 11)$} & {(19, 22)} & {5}\\
{$\vdw(2; 3, 12)$} & {(22, 22)} & {1}\\
{$\vdw(2; 3, 13)$} & {(25, 29)} & {5}\\
{$\vdw(2; 3, 14)$} & {(29, 29)} & {4}\\
\hline
\end{longtable}

It seems that there is little variation concerning the total number of zeros:
\begin{conj}\label{conj-2}
There exists a constant $c > 0$ such that
$\abs{n_0(B)-n_0(B')} \leq ct$, $\forall{B,B'}\in S(t)$  with $t\geq 3$.
\end{conj}

And there seem to be very few consecutive zeros:
\begin{conj}\label{conj-3}
There exists a constant $c > 0$ such that
$n_{00}(B) < ct$, $\forall{B}\in S(t)$ with $t\geq 3$.
\end{conj}

\subsection{Number of 1's}
\label{sec:num1}

In this section, we determine $T=\min\set{\np(\epos(B))+\nv(\epos(B)): B\in S(t)}$, as well as
minimum and maximum values of $n_1(B)$ over all $B\in S(t)$.
The observations in Table \ref{tab-ones} lead
us to Conjectures \ref{conj-2v}, \ref{conj-3a}, and Questions \ref{ques-4} and \ref{ques-5}.

{\scriptsize
\begin{longtable}[c]{|c|c|c|c|}
\caption{ Selected good-partitions of $\set{1,2,\ldots,\vdw(2; 3, t)-1}$}
\label{tab-ones}\\
  \hline \hline
 {\tt $\vdw(2; 3, t)$} & {\tt A good partition $B$ corresponding to $T$} &{\tt $T$} & {$\min\set{n_1(B): B\in S(t)}$,}\\
 {} & {} &{} & {$\max\set{n_1(B): B\in S(t)}$}\\
  \hline
\endhead
\hline
  \multicolumn{4}{|r|}{{Continued on Next Page\ldots}} \\
\hline
\endfoot
\endlastfoot
{} & {\tt } & {} & {}\\
{$\vdw(2; 3, 3) = 9$} & {$1^2001^200$} & {$1$}& {(4, 4)}\\
{} & {\tt (2 2)} & {}& {}\\
{} & {\tt } & {}& {}\\
{$\vdw(2; 3, 4) = 18$} & {$1^3001^101^3001^101^3$} & {$5$}& {(11, 11)}\\
{} & {\tt (3 1 3 1 3)} & {}& {}\\
{} & {\tt } & {}& {}\\
{$\vdw(2; 3, 5) = 22$} & {$001^3001^101^4001^401^1$} & {$4$}& {(12, 14)}\\
{} & {\tt (3 1 4 4 1)} & {}& {}\\
{} & {\tt } & {}& {}\\
{$\vdw(2; 3, 6) = 32$} & {$01^5001^501^3001^5001^5$} & {$3$}& {(21, 23)}\\
{} & {\tt (5 5 3 5 5)} & {}& {}\\
{} & {\tt } & {}& {}\\
{$\vdw(2; 3, 7) = 46$} & {$1^101^101^401^201^501^401^1001^301^501^201^50$} & {$8$}& {(33, 34)}\\
{} & {\tt (1 1 4 2 5 4 1 3 5 2 5)} & {}& {}\\
{} & {\tt } & {}& {}\\
{$\vdw(2; 3, 8) = 58$} & {$1^401^201^401^101^401^301^5001^501^301^401^101^401^201^1$} & {$12$}& {(43, 43)}\\
{} & {\tt (4 2 4 1 4 3 5 5 3 4 1 4 2 1)} & {}& {}\\
{} & {\tt } & {}& {}\\
{$\vdw(2; 3, 9) = 77$} & {$1^8001^601^301^101^3001^501^801^5001^301^101^301^6001^8$} & {$5$}& {(60, 60)}\\
{} & {\tt (8 6 3 1 3 5 8 5 3 1 3 6 8)} & {}& {}\\
{} & {\tt } & {}& {}\\
{$\vdw(2; 3, 10) = 97$} & {$1^701^401^201^5001^2001^701^401^801^101^801^4001^6001^2001^801^9$} & {$13$}& {(75, 77)}\\
{} & {\tt (7 4 2 5 2 7 4 8 1 8 4 6 2 8 9)} & {}& {}\\
{} & {\tt } & {}& {}\\
{$\vdw(2; 3, 11) = 114$} & {$01^{10}01^{4}001^{6}01^{10}01^{2}001^{9}01^{6}01^{1}01^{9}001^{1}001^{10}01^{6}001^{10}01^{10}$} & {$11$}& {(91, 94)}\\
{} & {\tt (10 4 6 10 2 9 6 1 9 1 10 6 10 10)} & {}& {}\\
{} & {\tt } & {}& {}\\
{$\vdw(2; 3, 12) = 135$} &
{$1^{9}01^{8}01^{9}01^{2}01^{3}01^101^{7}01^{2}0101^{3}01^{11}0^{2}$} & {$17$}& {(112, 112)}\\
{} & {$1^{11}01^{3}0101^{2}01^{7}01^101^{3}01^{2}01^{9}01^{8}01^{9}$} & {$$}& {}\\
{} & {\tt (9 8 9 2 3 1 7 2 1 3 11 11 3 1 2 7 1 3 2 9 8 9)} & {}& {}\\
{} & {\tt } & {}& {}\\
{$\vdw(2; 3, 13) = 160$} &
{$1^101^601^{12}01^4001^{11}001^601^{10}01^201^401^{11}01^10$} & {$15$}& {(130, 134)}\\
{} & {$1^601^901^201^301^701^{10}01^1001^501^{12}01^501^401^2$} & {$$}& {}\\
{} & {\tt (1 6 12 4 11 6 10 2 4 11 1 6 9 2 3 7 10 1 5 12 5 4 2)} & {}& {}\\
{} & {\tt } & {}& {}\\
\hline
\hline
\end{longtable}
}

Again, there seems little variation concerning the total number of ones:
\begin{conj}\label{conj-2v}
There exists a constant $c > 0$ such that
$ \abs{n_1(B)-n_1(B')} \leq ct$, $\forall{B,B'}\in S(t)$ with $t\geq 3$.
\end{conj}

Stronger than Conjecture \ref{conj-2v}, the number of ones seems very close to the vdW-number for the previous $t$:
\begin{conj}\label{conj-3a}
There exists a constant $c > 0$ such that
$ \abs{\vdw(2; 3, t-1) - n_1(B)} < c t$, $\forall{B}\in S(t)$.
\end{conj}

This conjecture also implies the earlier conjecture on the quadratic growth of $\vdw(2;3,t)$:
\begin{lem}\label{lem:relconj}
  Conjecture \ref{conj-3a} implies Conjecture \ref{conj-2v} and Conjecture \ref{conj-1}.
\end{lem}
\begin{proof}
Conjecture \ref{conj-2v} follows by the triangle inequality. Conjecture \ref{conj-1} follows, if for $t$ large enough we can show $n_0(B) \le n_1(B)$ for all $B \in S(T)$, and this is a special case of Szemer\'edi's Theorem (\cite{Szemeredi1975AP}), which for arithmetic progressions of size $3$ was already proven in \cite{Roth1953vdW}\footnote{see \url{http://rothstheorem.wikidot.com/on-certain-sets-of-integers}}, namely that the relative size of maximum independent subsets of the hypergraph of arithmetic progressions of size $3$ in the numbers $1,\dots,t$ goes to $0$ with $t \rightarrow \infty$.
\end{proof}

We turn to the growth of the number of peaks and valleys:
\begin{ques}\label{ques-4}
For each positive constant $c$ does there exist a $t'$ such that for all $t\geq t'$,
$\np(\epos(B))+\nv(\epos(B)) \geq ct$, ($t\geq 3$) $\forall{B}\in S(t)$? (We conjecture yes).
\end{ques}

We conclude with the observation, that for $t > 3$ there do not seem to be long plateaus for the numbers of ones:
\begin{ques}\label{ques-5}
Is there a good partition $B\in S(t)$, ($t\geq 4$) with 3 consecutive
numbers equal in $\epos(B)$? (Note that, for $t=3$, the partition $1^101^1001^101^1$
has four consecutive exponents, which are the same.)
\end{ques}

\subsection{How can it help for SAT solving?}
\label{sec:canhelp}

If one of the above conjectures (or some other conjecture) turns out to be true, and if moreover the numerical constants have good estimates, then they can be used to restrict the search space. When using a general purpose SAT solver, this can be achieved by adding further constraints. It seems however that these constraints do not help with the search, even if we assume that they are true, since they are too difficult to handle for the solver. It seems the problem is that these constraints do not mix well with the original problem formulation, and a deeper integration is needed. Such an integration was achieved in the case of the palindromic constraint, which is the subject of the following section --- here an organic new problem formulation could be established, where the additional restriction doesn't appear as an ``add-on'', but establishes a natural new problem class.

\section{Palindromes}
\label{sec:Palindromes}

Recall Definitions \ref{def:vdwn}, \ref{def:goodpart}:
\begin{enumerate}
\item for given $k \in \NN$ (the number of ``colours''),
\item $t_0, \dots, t_{k-1}$ (the lengths of arithmetic progressions),
\item and $n \in \NN$ (the number of vertices)
\end{enumerate}
we consider block partitions $(P_0, \dots, P_{k-1})$ of $\set{1, \dots, n}$ such that no $P_i$ contains an arithmetic progression of length $t_i$ --- these are the ``good partitions'', and $\vdw(k; t_0, \dots, t_{k-1}) \in \NN$ is the smallest $n$ such that no good partition exists. If $(P_0, \dots, P_{k-1})$ is a good partition of $\set{1,\dots,n}$ w.r.t.\ $t_0, \dots, t_{k-1}$, then for $1 \leq n' \leq n$ we obtain a good partition of $\set{1,\dots,n'}$ w.r.t.\ $t_0, \dots, t_{k-1}$ by just removing vertices $n'+1, \dots, n$ from their blocks. Thus $\vdw(k; t_0, \dots, t_{k-1})$ completely determines for which $n \in \NN$ good partitions exist, namely exactly for $n < \vdw(k; t_0, \dots, t_{k-1})$.

\begin{defn}\label{def:mirror}
  For $n \in \NN$ let $\mir_n: \set{1,\dots,n} \rightarrow \set{1,\dots,n}$ (with ``m'' like ``mirror'') defined by $\mir_n(v) := n+1-v$. This map is extended to $S \subseteq \set{1,\dots,n}$ as usual: $\mir_n(S) := \set{\mir_n(v) : v \in S}$.
\end{defn}
Now if $(P_0, \dots, P_{k-1})$ is a good partition w.r.t.\ $n$, then also $(\mir_n(P_0), \dots, \mir_n(P_{k-1}))$ is a good partition w.r.t.\ $n$. So it is of interest to consider self-symmetric partitions (with $\mir_n(P_i) = P_i$ for all $i$):

\begin{defn}\label{def:goodpp}
  A \emph{good palindromic partition} of $\set{1,\dots,n}$ w.r.t.\ parameters $t_0,\dots,t_{k-1}$, where $n, t_0, \dots, t_{k-1} \in \NN$, is a good partition of $\set{1,\dots,n}$ w.r.t.\ $t_0,\dots,t_{k-1}$ such that for all $j \in \set{0,\dots,k-1}$ holds $\mir_n(P_j) = P_j$.
\end{defn}
We call these special good partitions ``palindromic'', since a block partition can be represented as a string of numbers over $\set{0, \dots, k-1}$, and then the block partition is palindromic iff the string is a palindrome (reads the same forwards and backwards). For example, the string $01^2001^20$ represents a good palindromic partition for $k=2$, $t_0=t_1=3$ and $n=8$, namely $(\set{1,4,5,8},\set{2,3,6,7})$, and so does $(\set{1,3,6,8},\set{2,4,5,7})$, represented by $0101^2010$, while $(\set{1,2,5,6},\set{3,4,7,8})$, represented by $001^2001^2$, is a good partition which is not palindromic.

For given $k$ and $t_0, \dots, t_{k-1}$ again we want to completely determine (in theory) for which $n$ do good palindromic partitions exist and for which not. The key is the following observation (which follows also from Lemmas \ref{lem:embpd}, \ref{lem:pdsat}).

\begin{lem}\label{lem:transferpdpartitions}
  Consider fixed $k, t_0, \dots, t_{k-1}$, and $n \geq 3$. From a good palindromic partition $(P_0, \dots, P_{k-1})$ of $\set{1,\dots,n}$ we obtain a good palindromic partition $(P_0', \dots, P_{k-1}')$ of $\set{1,\dots,n-2}$ by removing vertices $1, n$ and replacing the remaining vertices $v$ by $v-1$, that is, $P_i' := \set{v - 1 : v \in P_i \setminus \set{1,n}}$.
\end{lem}
\begin{proof}
The notion of a good partition of $\set{1,\dots,n}$ w.r.t.\ $\vdw(k; t_0,\dots,t_{k-1})$, as defined in Definition \ref{def:goodpart}, can be generalised to good partitions of arbitrary $T \subseteq \ZZ$ by demanding that for every block partition $(P_0,\dots,P_{k-1})$ of $T$ into $k$ parts no part $P_j$ contains an arithmetic progression of size $t_j$. In the remainder of the proof we omit the ``w.r.t.\ $t_0,\dots,t_{k-1}$''.

If $T$ has a good partition, then also every subset has a good partition, by restricting the blocks accordingly, and for every $d \in \ZZ$ also $d + T = \set{d + x : x \in T}$ has a good partition, by shifting the blocks as well.

We can also generalise the notion of a good palindromic partition to intervals $T = \set{a, a+1, \dots, b} \subset \ZZ$ for $a < b$, defining now the mirror-map $m_{a,b}: T \rightarrow T$ via $v \in T \mapsto b + a - v$ ($m_n$ in Definition \ref{def:mirror} is the special case $m_n = m_{1,n}$).

Again, if $T$ has a good palindromic partition, then $d + T$ for $d \in \ZZ$ has as well. But for subsets of $T$ we can only consider sub-intervals $T' = \set{a',\dots,b'}$, where from both sides we have taken away equal amounts. That is, for $a \leq a' < b' \leq b$ with $a' - a = b - b'$ we have, that from a good palindromic partition for $T$ we can obtain a good palindromic partition for $T'$ (by just restricting the blocks).

So from a good palindromic partition of $\set{1,\dots,n}$ we obtain a good palindromic partition of $\set{1,\dots,n-2}$ by first restricting to $\set{2,\dots,n-1}$ and then shifting by $-1$.
\end{proof}

\begin{cor}\label{cor:transferpdpartitions1}
  If  there is no good palindromic partition of $\set{1,\dots,n}$, then there is no good palindromic partition of $\set{1,\dots,n+2 \cdot i}$ for all $i \in \NNZ$.
\end{cor}
\begin{proof}
If there would be a good palindromic partition of $\set{1,\dots,n+2 \cdot i}$, then by repeated applications of Lemma \ref{lem:transferpdpartitions} we would obtain a good palindromic partition of $\set{1,\dots,n}$.
\end{proof}

Since by van der Waerden's theorem we know there always exists some $n$ such that for all $n' \geq n$ no good palindromic partition exists, we get that the existence of good palindromic partitions w.r.t.\ fixed $t_0, \dots, t_{k-1}$ is determined by two numbers, the endpoint $p$ of ``always exists'' resp.\ $q$ of ``never exists'', with alternating behaviour in the interval in-between:
\begin{cor}\label{cor:transferpdpartitions2}
  Consider the maximal $p \in \NNZ$ such that for all $n \leq p$ good palindromic partitions exist, and the minimal $q \in \NN$ such that for all $n \geq q$ no good palindromic partitions exist. Then $q - p$ is an odd natural number, where no good palindromic partition exists for $p+1$, but $p+2$ again has a good palindromic partition, and so on alternately, until from $q$ on no good palindromic partition exists anymore.
\end{cor}
\begin{proof}
By Corollary \ref{cor:transferpdpartitions1} there is no good palindromic partition for $p+1 + 2 i$ and all $i \in \NNZ$. Now for the first $i \in \NNZ$, such that $p+2 + 2 i$ has no good palindromic partition, we let $q' := (p+2 + 2 i) - 1 $. We have a good palindromic partition for $q-1$ by definition of $i$ (as the smallest such $i$) resp.\ in case of $i=0$ by definition of $p$. We have $q' + 2 j = (p+2 + 2 i) - 1 + 2 j = p+1 + 2 (i+j)$ for $j \in \NNZ$, and thus there is no good palindromic partition for $q' + 2j$. And if there would be a good palindromic partition for $q' + 1 + 2j = p+2 + 2 i + 2 j$, then by Corollary \ref{cor:transferpdpartitions1} there would be a good palindromic partition for $p+2 + 2i$. So we have $q' = q$.
\end{proof}

\begin{defn}\label{def:pdvdw}
  The \emph{palindromic van-der-Waerden number} $\vdwpd(k; t_0, \dots, t_{k-1}) \in \NNZ^2$ is defined as the pair $(p,q)$ such that $p$ is the largest $p \in \NNZ$ with the property, that for all $1 \leq n \leq p$ there exists a good palindromic partition of $\set{1,\dots,n}$, while $q$ is the smallest $q \in \NN$ such that for no $n \geq q$ there exists a good palindromic partition of $\set{1,\dots,n}$. We use $\vdwpd(k; t_0, \dots, t_{k-1})_1 = p$ and $\vdwpd(k; t_0, \dots, t_{k-1})_2 = q$. So $0 \leq \vdwpd(k; t_0, \dots, t_{k-1})_1 < \vdwpd(k; t_0, \dots, t_{k-1})_2 \leq \vdw(k; t_0, \dots, t_{k-1})$.

  The \emph{palindromic gap} is
  \begin{displaymath}
    \pdg(k; t_0, \dots, t_{k-1}) := \vdw(k; t_0, \dots, t_{k-1}) - \vdwpd(k; t_0, \dots, t_{k-1})_2 \in \NNZ,
  \end{displaymath}
 while the \emph{palindromic span} is defined as
 \begin{displaymath}
  \pds(k; t_0, \dots, t_{k-1}) :=  \vdwpd(k; t_0, \dots, t_{k-1})_2 - \vdwpd(k; t_0, \dots, t_{k-1})_1 \in \NN.
\end{displaymath}
\end{defn}

To certify that $\vdw(k;t_0,\dots,t_{k-1}) = n$ holds means to show that there exists a good partition of $\set{1,\dots,n-1}$ and that there is no good partition of $n$. For palindromic number-pairs we need to double the effort:
\begin{thm}\label{thm:certificatespd}
  To certify that $\vdwpd(k;t_0,\dots,t_{k-1}) = (p,q)$ holds, exactly the following needs to be shown for (arbitrary) $p \in \NNZ$, $q \in \NN$ with $p < q$:
  \begin{enumerate}
  \item[(i)] there are good palindromic partitions of $\set{1,\dots,p-1}$ and $\set{1,\dots,q-1}$ w.r.t.\ $t_0, \dots, t_{k-1}$;
  \item[(ii)] there are no good palindromic partitions of $\set{1,\dots,p+1}$ and $\set{1,\dots,q+1}$ w.r.t.\ $t_0, \dots, t_{k-1}$.
  \end{enumerate}
\end{thm}
\begin{proof}
The given conditions are necessary for $\vdwpd(k;t_0,\dots,t_{k-1}) = (p,q)$ by the defining properties of $p$ and $q$. We show that they are sufficient to establish $\vdwpd(k;t_0,\dots,t_{k-1}) = (p,q)$. First we have by Corollary \ref{cor:transferpdpartitions1} that $q - p$ is odd, since otherwise $p+1$ having no good palindromic partitions would yield that $q-1$ would have no good palindromic partition. Then, again by Corollary \ref{cor:transferpdpartitions1}, all $n \geq q+1$ have no good palindromic partition, while all $n \leq p-1$ have good palindromic partitions. By Corollary \ref{cor:transferpdpartitions2} we must now have $\vdwpd(k;t_0,\dots,t_{k-1}) = (p,q)$.
\end{proof}

\subsection{Palindromic vdW-hypergraphs}
\label{sec:palindvdwh}

Recall that a finite hypergraph $G$ is a pair $G = (V,E)$, where $V$ is a finite set (of ``vertices'') and $E$ is a set of subsets of $V$ (the ``hyperedges''); one writes $V(G) := V$ and $E(G) := E$. The essence of the (finite) van der Waerden problem (which we will now often abbreviate as ``vdW-problem'') is given by the hypergraphs $\arithp(t,n)$ of arithmetic progressions with progression length $t \in \NN$ and the number $n \in \NNZ$ of vertices:
\begin{itemize}
\item $V(\arithp(t,n)) := \set{1,\dots,n}$
\item $E(\arithp(t,n)) := \set{p \subseteq \set{1,\dots,n} : p \text{ arithmetic progression of length } t}$.
\end{itemize}
For example $\arithp(3,5) = (\set{1,2,3,4,5},\set{\set{1,2,3},\set{2,3,4},\set{1,3,5},\set{3,4,5}})$. Considering hypergraphs, the reader might wonder how determination of vdW-numbers fits with hypergraph colouring. While the determination of diagonal vdW-numbers is an ordinary hypergraph colouring problem, for general vdW-numbers a more general concept of hypergraph colouring is to be used, involving the simultaneous colouring of several hypergraphs in the following sense: The diagonal vdW-number $\vdw(k; t, \dots, t)$ for $k, t \in \NN$ is the smallest $n \in \NN$ such that the hypergraph $\arithp(t,n)$ is not $k$-colourable, where in general a $k$-colouring of a hypergraph $G$ is a map $f: V(G) \rightarrow \set{1,\dots,k}$ such that no hyperedge is ``monochromatic'', that is, every hyperedge gets at least two different values by $f$. For the general vdW-number $\vdw(k; t_0, \dots, t_{k-1})$ we now consider for each colour $i \in \set{0,\dots,k-1}$ the hypergraph $\arithp(t_i,n)$, and we forbid (to formulate ``good partition'') for each $i$ that there is a hyperedge in $\arithp(t_i,n)$ monocoloured with colour $i$ (while we do not care about the other colours here). Accordingly the SAT-encoding of ``$\vdw(2;3,t) > n$ ?'', as discussed in Subsection \ref{sec:usingsat}, exactly consists of the two hypergraphs $\arithp(3,n)$ and $\arithp(t,n)$ represented by positive resp.\ negative clauses.

The task now is to define the palindromic version $\pdarithp(t,n)$ of the hypergraph of arithmetic progressions, so that for diagonal palindromic vdW-numbers $\vdwpd(k; t, \dots, t) = (p,q)$ we have, that $q$ is minimal for the condition that for all $n \geq q$ the hypergraph $\pdarithp(t,n)$ is not $k$-colourable, while $p$ is maximal for the condition that for all $n \leq p$ the hypergraph is $k$-colourable. Furthermore we should have that for two-coloured problems (i.e., $k=2$) the SAT-encoding of ``$\vdwpd(2;t_0,t_1) > n$ ?'' (satisfiable iff the answer is yes) consists exactly of the two hypergraphs $\pdarithp(t_0,n)$, $\pdarithp(t_1,n)$ represented by positive resp.\ negative clauses (while for more than two colours generalised clause-sets can be used; see \cite{kullmann}).

Consider fixed $t \in \NN$ and $n \in \NNZ$. Obviously $\pdarithp(t,0) := \arithp(t,0) = (\set{},\set{})$, and so assume $n \geq 1$. Recall the permutation $m = m_n$ of $\set{1,\dots,n}$ from Definition \ref{def:mirror}. As every permutation, $m$ induces an equivalence relation $\sim$ on $\set{1,\dots,n}$ by considering the cycles, which here, since $m$ is an involution (self-inverse), just has the equivalence classes $\set{1,\dots,n} / \!\!\sim \:= \set{\set{v,f(v)}}_{v \in \set{1,\dots,n}}$ of size $1$ or $2$ comprising the elements and their images. Note that $m$ has a fixed point (an equivalence class of size $1$) iff $n$ is odd, in which case the unique fixed point is $\frac{n+1}2$. The idea now is to define $m': \set{1,\dots,n} \rightarrow \set{1,\dots,n}$, which chooses from each equivalence class one representative (so $m'(v) \in \set{v,m(v)}$ and $v \sim w \Leftrightarrow m'(v) = m'(w)$), and to let $\pdarithp(t,n)$ be the image of $\arithp(t,n)$ under $m'$, that is, $(m'(V(\arithp(t,n))),\set{m'(H)}_{H \in E(\arithp(t,n))})$. Naturally we choose $m'(v)$ to be the smaller of $v$ and $m(v)$. Now it occurs that images of arithmetic progressions under $m'$ can subsume each other, i.e., for $H_1, H_2 \in E(\arithp(t,n))$ with $H_1 \ne H_2$ we can have $m'(H_1) \subset m'(H_2)$, and so we define $\pdarithp(t,n)$ as the image of $\arithp(t,n)$ under $m'$, where also all subsumed hyperedges are removed (so we only keep the minimal hyperedges under the subset-relation).

\begin{defn}\label{def:pdarithp}
  For $t \in \NN$ and $n \in \NNZ$ the hypergraph $\pdarithp(t,n)$ is defined as follows:
  \begin{itemize}
  \item $V(\pdarithp(t,n)) := \set{1,\dots,\ceil{\frac n2}}$
  \item $E(\pdarithp(t,n))$ is the set of minimal elements w.r.t.\ $\subseteq$ of the set of $m'_n(H)$ for $H \in E(\arithp(t,n))$, where $m'_n: \set{1,\dots,n} \rightarrow V(\pdarithp(t,n))$ is defined by $m'_n(v) := v$ for $v \leq \ceil{\frac n2}$ and $m'_n(v) := n+1-v$ for $v > \ceil{\frac n2}$.
  \end{itemize}
\end{defn}
Using $\arithp(3,5) = (\set{1,2,3,4,5},\set{\set{1,2,3},\set{2,3,4},\set{1,3,5},\set{3,4,5}})$ as above, we have $m'(\set{1,2,3}) = \set{1,2,3}$, $m'(\set{2,3,4}) = \set{2,3}$, $m'(\set{1,3,5}) = \set{1,3}$ and $m'(\set{3,4,5}) = \set{1,2,3}$, whence $\pdarithp(3,5) = (\set{1,2,3}, \set{\set{1,3},\set{2,3}})$.

\begin{lem}\label{lem:embpd}
  Consider $t \in \NN$ and $n \in \NNZ$. The hypergraph $\pdarithp(t,n)$ is embedded into the hypergraph $\pdarithp(t,n+2)$ via the map $e: V(\pdarithp(t,n)) \rightarrow V(\pdarithp(t,n+2))$ given by $v \mapsto v+1$.
\end{lem}
\begin{proof}
  First we note that $\abs{V(\pdarithp(t,n+2))} = \abs{V(\pdarithp(t,n))}+1$, and so the range of $e$ is $V(\pdarithp(t,n+2)) \setminus \set{1}$. Let $G$ be the hypergraph with vertex set $V(\pdarithp(t,n+2)) \setminus \set{1}$, whose hyperedges are all those hyperedges $H \in E(\pdarithp(t,n+2))$ with $1 \notin H$.  We show that $e$ is an (hypergraph-)isomorphism from $\pdarithp(t,n)$ to $G$, which proves the assertion.

  Now obviously the underlying hypergraph $\arithp(t,n)$ is embedded into the underlying $\arithp(t,n+2)$ via the underlying map $v \in V(\arithp(t,n)) \mapsto v+1 \in V(\arithp(t,n+2))$, where the image of this embedding is given by the hypergraph with vertex set $V(\arithp(t,n+2)) \setminus \set{1,n+2}$, and where the hyperedges are those $H \in E(\arithp(t,n+2))$ with $1, n+2 \notin H$. Since $m'_{n+2}(n+2) = 1$ and $m'_n(v) = m'_{n+2}(v+1)-1$ for $v \in \set{1,\dots,n}$, the assertion follows from the fact that there are no hyperedges $H, H' \in E(\arithp(t,n+2))$ with $H \cap \set{1,n+2} \not= \emptyset$, $H' \cap \set{1,n+2} = \emptyset$ and $m'_{n+2}(H) \subset m'_{n+2}(H')$ (thus $m'_{n+2}(H')$ can only be removed from $\pdarithp(t,n+2)$ by subsumptions already at work in $\pdarithp(t,n)$), and this is trivial since $1 \in m'_{n+2}(H)$ but $1 \notin m'_{n+2}(H')$.
\end{proof}

The SAT-translation of ``Is there a good palindromic partition of $\set{1,\dots,n}$ w.r.t.\ $t_0,t_1$ ?'' is accomplished similar to the translation of ``$\vdw(2;t_0,t_1) > n$ ?'', now using $\pdarithp(t_0,n), \pdarithp(t_1,n)$ instead of $\arithp(t_0,n), \arithp(t_1,n)$:
\begin{lem}\label{lem:pdsat}
  Consider $t_0, t_1 \in \NN$, $t_0 \leq t_1$, and $n \in \NNZ$. Let the boolean clause-set $\Fpd(t_0,t_1,n)$ be defined as follows:
  \begin{itemize}
  \item the variable-set is $\set{1,\dots,\ceil{\frac n2}}$ ($ = V(\pdarithp(t_0,n)) = V(\pdarithp(t_1,n))$);
  \item the hyperedges of $\pdarithp(t_0,n)$ are directly used as positive clauses;
  \item the hyperedges $H$ of $\pdarithp(t_1,n)$ yield negative clauses $\set{\ol{v}}_{v \in H}$.
  \end{itemize}
  Then there exists a good palindromic partition if and only if $\Fpd(t_0,t_1,n)$ is satisfiable, where the satisfying assignments are in one-to-one correspondence to the good palindromic partitions of $\set{1,\dots,n}$ w.r.t.\ $(t_0,t_1)$. \qed
\end{lem}
For more than two colours, Lemma \ref{lem:pdsat} can be generalised by using generalised clause-sets, as in \cite{kullmann}, and there one also finds the ``generic translation'', a general scheme to translate generalised clause-sets (with non-boolean variables) into boolean clause-sets (see also \cite{Kullmann2007ClausalFormZI,Kullmann2007ClausalFormZII}).

\subsection{Precise values}
\label{sec:Palindromespv}

See Subsection \ref{sec:remsatcomp} for details of the computation.

{\renewcommand{\arraystretch}{1.4}
\begin{longtable}[c]{|c|c|c|c|}
\caption{Palindromic vdW-numbers $\vdwpd(2; 3, t)$}
\label{tab-pdprecise}\\
  \hline \hline
 $t$ & $\vdwpd(2;3,t)$ & $\pds(2;3,t)$ & $\pdg(2;3,t)$\\
  \hline
\endhead
\hline
  \multicolumn{4}{|r|}{{Continued on Next Page\ldots}} \\
\hline
\endfoot
\endlastfoot

$3$ & $(6,9)$ & $3$ & $0$ \\
\hline

$4$ & $(15,16)$ & $1$ & $2$\\
\hline
$5$ & $(16,21)$ & $5$  & $1$\\
\hline
$6$ & $(30,31)$ & $1$ & $1$\\
\hline
$7$ & $(41,44)$ & $3$ & $2$\\
\hline
$8$ & $(52,57)$ & $5$ & $1$\\
\hline
$9$ & $(62,77)$ & $15$ & $0$\\
\hline
$10$ & $(93,94)$ & $1$ & $3$\\
\hline
$11$ & $(110,113)$ & $3$ & $1$\\
\hline
$12$ & $(126,135)$ & $9$ & $0$\\
\hline
$13$ & $(142,155)$ & $13$ & $5$\\
\hline
$14$ & $(174, 183)$ & $9$ & $3$\\
\hline
$15$ & $(200, 205)$ & $5$ & $13$ \\
\hline
$16$ & $(232, 237)$ & $5$ & $1$ \\
\hline
$17$ & $(256, 279)$ & $23$ & $0$ \\
\hline
$18$ & $(299, 312)$ & $13$ & $0$ \\
\hline
$19$ & $(338, 347)$ & $9$ & $2$ \\
\hline
\hline
$20$ & $(380, 389)$ & $9$ & $\geq 0$ \\
\hline
$21$ & $(400, 405)$ & $5$ & $\geq 11$ \\
\hline
$22$ & $(444, 463)$ & $19$ & $\geq 1$ \\
\hline
$23$ & $(506, 507)$ & $1$ & $\geq 9$ \\
\hline
$24$ & $(568, 593)$ & $25$ & $\geq 0$ \\
\hline
$25$ & $(586, 607)$ & $21$ & $\geq 49$ \\
\hline
$26$ & $(634,643)$ & $9$ & $\geq 84$ \\
\hline
$27$ & $(664,699)$ & $35$ & $\geq 71$\\
\hline

\hline
\hline
\end{longtable}
}

\subsection{Conjectured values and bounds}
\label{sec:Palindromesconj}

For $28 \leq t \leq 39$ we have reasonable values on $\vdwpd(2;3,t)$, which are given in Table \ref{tab-pdconj}, and which we believe to be exact for $t \leq 35$. These values have been computed by local-search methods (see Subsection \ref{sec:remsatincomp}), and thus for sure we can only say that they present lower bounds. We obtain conjectured values for the palindromic span (which might however be too large or too small) and conjectured values for the palindromic gap (which additionally depend on the conjectured values from Subsection \ref{sec:Somenewconjectures}, while for $t \geq 31$ we only have the lower bounds from Subsection \ref{sec:Somenewconjectures}).

{\renewcommand{\arraystretch}{1.4}
\begin{longtable}{|c|c|c|c|}
\caption{\protect\makebox[23em]{Conjectured palindromic vdW-numbers $\vdwpd(2; 3, t)$}}
\label{tab-pdconj}\\
  \hline \hline
 $t$ & $\vdwpd(2;3,t) \geq$ & $\pds(2;3,t) \sim$ & $\pdg(2;3,t) \sim$\\
  \hline
\endhead
\hline
  \multicolumn{4}{|r|}{{Continued on Next Page\ldots}} \\
\hline
\endfoot
\endlastfoot

$28$ & $(728,743)$ & $15$  & $84$\\
\hline
$29$ & $(810,821)$ & $11$ & $47$\\
\hline
$30$ & $(844,855)$ & $11$ & $48$\\
\hline
\hline
$31$ & $(916,931)$ & $15$ & $0$\\
\hline
$32$ & $(958,963)$ & $5$ & $44$\\
\hline
$33$ & $(996,1005)$ & $9$ & $59$\\
\hline
$34$ & $(1054,1081)$ & $27$ & $63$\\
\hline
$35$ & $(1114,1155)$ & $41$ & $50$\\
\hline
$36$ & $(1186,1213)$ & $27$ & $45$\\
\hline
$37$ & $(1272, 1295)$ & $23$ & $44$\\
\hline
$38$ & $(1336, 1369)$ & $33$ & $10$ \\
\hline
$39$ & $(1406, 1411)$ & $5$ & $8$ \\
\hline

\hline
\hline
\end{longtable}
}

For the certificates for these lower bounds see \ref{sec:goodpp}.

\subsection{Open problems}
\label{sec:pdopen}

The relation between ordinary and palindromic vdW-numbers are of special interest:
\begin{itemize}
\item It seems the palindromic span can become arbitrarily large --- also in relative terms? Perhaps the span shows a periodic behaviour, oscillating between small and large?
\item Similar questions are to be asked for the gap. Does it attain value $0$ infinitely often?
\end{itemize}
Do the hypergraphs $\pdarithp(t,n)$ have interesting properties (more basic than their chromatic numbers)? A basic exercise would be to estimate the number of hyperedges and their sizes. In the subsequent Subsection \ref{sec:remsatcomp} we find data that SAT solvers behave rather different on palindromic vdW-problems (compared to ordinary problems). It seems that palindromic problems are more ``structured'' than ordinary problems --- can this be made more precise? Perhaps the hypergraphs $\pdarithp(t,n)$ show characteristic differences to the hypergraphs $\arithp(t,n)$, which could explain the behaviour of SAT solvers?

\subsection{Remarks on the use of symmetries}
\label{sec:remsym}

The heuristic use of symmetries for finding good partitions has been studied in \cite{herwig,HeuleWalsh2010SolutionsSymmetry,HeuleWalsh2010SolutionsSymmetryW} (while for symmetries in the context of general SAT solving see \cite{Sak09HBSAT}). Especially we find there an emphasis on ``internal symmetries'', which are not found in the problem, but are imposed on the solutions.

The good palindromic partitions introduced in this section are more restricted in the sense, that they are based on the symmetries $m$ from Subsection \ref{sec:palindvdwh} of the clause-sets $F$ expressing ``$\vdw(k; t_0, t_1,\ldots, t_{k-1}) > n$ ?'' (i.e., we have $m(F) = F$; recall Subsection \ref{sec:usingsat}), which then is imposed as an internal symmetry on the potential solution by demanding that the solutions be self-symmetric. In \cite{herwig} ``reflection symmetric'' certificates are mentioned, which for even $n$ are the same as good palindromic partitions, however for odd $n$ they ignore vertex $1$, not the mid-point $\ceil{\frac n2}$ as we do. This definition in \cite{herwig} serves to maintain monotonicity (i.e., a solution for $n+1$ yields a solution for $n$, while we obtain one only for $n-1$ (Lemma \ref{lem:transferpdpartitions}). We believe that palindromicity is a more natural notion, but further studies are needed here to compare these two notions.

Other internal symmetries used in \cite{herwig,HeuleWalsh2010SolutionsSymmetry,HeuleWalsh2010SolutionsSymmetryW} are obtained by modular additions and multiplications (these are central to the approaches there), based on the method from \cite{Rabung1979vdW} for constructing lower bounds for diagonal vdW-numbers. No generalisations are known for the mixed problems we are considering.

Finally we wish to emphasise that we do not consider palindromicity as a mere heuristic for finding lower bounds, but we get an interesting variation of the vdW-problem in its own right, which hopefully will help to develop a better understanding of the vdW-problem itself in the future.

\section{Experiments with SAT solvers}
\label{sec:remsat}

We conclude by summarising the experimental results and insights gained by running SAT solvers on the instances considered in this paper. All the solvers (plus build environments), generators and the data are available in the \OKlibrary{} (\cite{Kullmann2009OKlibrary}); see \ref{sec:OKl}
for more information.

For determining unsatisfiability we consider complete SAT solvers in Subsection \ref{sec:remsatcomp}.
In general, for (ordinary) vdW-problems look-ahead solvers seem to perform better than conflict-driven solvers, while for palindromic problems it seems to be the opposite. However \ntawsolver{} is the best (single) solver for both classes.

The hybrid approach, \cubeconq, was developed precisely on the instances of this paper, as discussed in \cite{HeuleKullmannWieringaBiere2011Cubism} (further developments one finds in \cite{TakHeuleBiere2012CC}). This approach is third-best on vdW-problems (after \ntawsolver{} and \ttawsolver), and best on palindromic vdW-problems (before \ttawsolver{} and \ntawsolver).

We conclude this section in Subsection \ref{sec:remsatincomp} by remarks on incomplete SAT solvers, used to obtain lower bounds (determine satisfiability).

For the experiments we used a 64-bit workstation with 32 GB RAM and Intel i5-2320 CPUs (6144 KB cache) running with 3 GHz, where we only employed a single CPU.

\subsection{Complete solvers}
\label{sec:remsatcomp}

Complete SAT solvers exist in mainly two forms, ``look-ahead solvers'' and ``conflict-driven solvers''; see \cite{MSLM09HBSAT,HvM09HBSAT} for general overviews on these solver paradigms. Besides the \tawsolver{} (see Section \ref{sec:tawsolver}), for our experimentation we use the following (publicly available) complete solvers, which give a good coverage of state of the art SAT solving and of the winners of recent SAT competitions and SAT races\footnote{The (parent) SAT competition homepage is at \url{http://www.satcompetition.org} with links to each individual competition.}:
\begin{itemize}
\item Look-ahead solvers:
  \begin{itemize}
  \item \oksolver{} (\cite{Ku2002h}), a solver with well-defined behaviour, no ad-hoc heuristics, and which applies complete $r_2$ (at every node). This solver won gold at the SAT 2002 competition.
  \item \satz{} (\cite{LA1996}), a solver which applies partial $r_2$ and $r_3$. In the \OKlibrary{} we maintain version 215, with improved/corrected in/output and coding standard.
  \item \march{} (\cite{HeuleDufourvanZwietenMaaren2004Reasoning}), a solver applying partial $r_2, r_3$, and resolution- and equivalence-preprocessing. \march{} contains the same underlying technology as its sibling solvers \texttt{march\_\{rw,hi,ks,dl,eq\}}, which won gold, silver and bronze at the 2004 to 2011 SAT competitions and SAT races. We use the \texttt{pl} (\underline{p}artial \underline{l}ookahead) version.
  \end{itemize}
\item Conflict-driven solvers:
  \begin{itemize}
  \item \texttt{MiniSat} family:
    \begin{itemize}
    \item \minisat{} (\cite{EenSoerensson2003Minisat}), version 2.0 and 2.2, the latest version of this well-established solver, used as starting point for many new conflict-driven solvers. Previous versions won gold at the SAT Race 2006 and 2008, as well as numerous bronze and silver awards at the SAT competition 2007.
    \item \cryptomini{} (\cite{CryptoSAT2009}), a \minisat{} derivative designed specifically to tackle hard cryptographic problems. This solver won gold at SAT Race 2010 and gold and silver at the SAT competition 2011. We use version 2.9.6.
    \item \glucose{} (\cite{AudemardSimon2009PredictingLearnt}), a \minisat{} derivative utilising a new clause scoring scheme and aggressive learnt-clause deletion. This solver won gold in both SAT 2011 competition and SAT Challenge 2012. We use versions 2.0 and 2.2.
    \end{itemize}
  \item \lingeling{} family:
    \begin{itemize}
    \item \picosat{} (\cite{Biere2008picosat}), a conflict-driven solver using an aggressive restart strategy, compact data-structures, and offering proof-trace options to allow for unsatisfiability checking. This solver won gold and silver at the SAT competition 2007. We use the latest version 913.
    \item \precosat{} (\cite{Biere2009Precosat}), integrates the \texttt{SATeLite} preprocessor into \picosat{}, applying various reductions including partial $r_2$ at certain nodes in the search tree. This solver won gold and silver at the SAT 2009 competition. We use the latest version 570.
    \item \lingeling (\cite{Biere2012LingelingDesc}), based on \precosat, focuses further on integrating preprocessing and search, introducing new algorithms and data-structures to speed up these techniques and reduce memory footprint. As with \precosat, this solver applies partial $r_2$ at specially chosen nodes in the search tree. This solver won bronze at the SAT 2011 competition and silver at the SAT Race 2010. We use the latest version \texttt{ala-b02aa1a-121013}.
    \end{itemize}
  \end{itemize}
\end{itemize}

\subsubsection{Cube-and-Conquer}
\label{sec:cubeconq}

The \cubeconq{} method uses a look-ahead solver as the ``cube-solver'', splitting the instance into subinstances small enough such that the ``conquer-solver'', a conflict-driven solver, can solve almost all sub-instances in at most a few seconds. We use the \oksolver{} as the cube-solver and \minisat{} as the conquer-solver. The main (and single) parameter is $D \in \NNZ$, the cut-off depth for the \oksolver: the DLL-tree created by the \oksolver{} is cut off when the number of assignments reaches $D$, where it is important that this includes \emph{all assignments} on the path, not just the decisions, but also the forced assignments found by $r_1$ and $r_2$ --- only in this way a relatively balanced load is guaranteed. The data reported in Tables \ref{tab:CCvdW}, \ref{tab:CCpdvdW} shows first data on the cube-phase, namely
\begin{itemize}
\item $D$ (cut-off depth),
\item the number of nodes in the (truncated) DLL-tree of the \oksolver,
\item the time needed (this includes writing the partial assignments representing the sub-instances to files),
\item and the number $N$ of sub-instances.
\end{itemize}
For the conquer-phase we have:
\begin{itemize}
\item the median and maximum time for solving the sub-instances by \minisat{},
\item the sum of conflicts over all sub-instances,
\item and the total time used by \minisat.
\end{itemize}
Finally the overall total time is reported, which does not include the time used by the processing-script, which applies the partial assignments to the original instance and produces so the sub-instances: this adds an overhead of nearly $20\%$ for the smallest problem, but this proportion becomes smaller for larger problems, and is less than $1\%$ for the largest problems.

\subsubsection{VdW-problems}
\label{sec:remsatcompvdw}

We consider the (unsatisfiable) instances to determine the upper bounds for $\vdw(2; 3,t)$ with $12 \le t \le 17$; in Table \ref{tab:vdwdata} we give basic data for these instances (plus $t=18,19$).

\begin{table}[H]
  \centering
  \begin{tabular}{c|c||c|c|c|c}
    $t$ & $n$ & $c$ & $c_3$ & $c_t$ & $\ell$\\
    \hline
    12 & 135 & 5,251 & 4,489 & 762 & 22,611\\
    13 & 160 & 7,308 & 6,320 & 988 & 31,804\\
    14 & 186 & 9,795 & 8,556 & 1,239 & 43,014\\
    15 & 218 & 13,362 & 11,772 & 1,590 & 59,166\\
    16 & 238 & 15,812 & 14,042 & 1,770 & 70,446\\
    17 & 279 & 21,616 & 19,321 & 2,295 & 96,978\\
    18 & 312 & 26,889 & 24,180 & 2,709 & 121,302\\
    19 & 349 & 33,487 & 30,276 & 3,211 & 151,837
  \end{tabular}
  \caption{Instance data for $F(3,t;n)$, where $n$ is the number of vertices as well as the number of variables, $c=c_3+c_t$ is the number of clauses, $c_i$ the number of clauses of length $i$, and $\ell = 3 c_3 + t c_t$ is the number of literal occurrences.}
  \label{tab:vdwdata}
\end{table}

In Table \ref{tab:complsolvervdw}, we see the running times and number of nodes/conflicts for the SAT solvers. We see that in general look-ahead solvers here have the upper hand over conflict-driven solvers, with the \ntawsolver{} with a large margin the fastest solver. Regarding conflict-driven solvers, we see that version 2.2 for \minisat{} is superior over version 2.0, while for \glucose{} it is the opposite. The low node-count for \march{} seems due to the preprocessing phase, which adds a large number of resolvents to the original instance: this reduces the node-count, but increases the run-time. Compared to the other look-ahead solvers, the strength of \ntawsolver{} is that the number of nodes is just larger by a factor of most $3$, while processing of each node happens much faster. Compared with the strongest conflict-driven solver, \minisat-2.2, we see that the node-count of \ntawsolver{} is considerably less than the number of conflicts used by \minisat, and that one node is processed somewhat faster than one conflict.

One aspect important here for the superiority of look-ahead solver is the ``tightness'' of the problem formulation. Consider for example $t=12$, not with $n=135$ as in Tables \ref{tab:vdwdata}, \ref{tab:complsolvervdw}, but with $n=1000$; this yields $c = 294{,}455$, $c_3 = 249{,}500$, $c_{12} = 44{,}955$, and $\ell = 1{,}287{,}960$, which is now a highly redundant problem instance. For \ntawsolver{} we obtain 1,311,511 nodes and 2,868 sec, and for \ttawsolver{} we get 935,475 nodes and 2,452 sec, while for \minisat-2.2 we get 1,140,616 conflicts and 159 sec. We see that \minisat-2.2 was able to utilise the additional clauses to determine unsatisfiability with fewer conflicts, and with a run-time not much affected by the large increase in problem size, while for \ntawsolver{} the run-time (naturally) explodes, and the number of nodes stayed the same.\footnote{The \oksolver{} yields a more extreme example: the run was aborted after 657,648 sec and 603,177 nodes, where yet only $\%49.2$ of the search space was visited (so that the solver was still working on completing the first branch at the root of the tree, making very slow progress towards $\%50$). This shows the big overhead caused by the $r_2$-reduction, and the danger of a heuristic which (numerically) sees opportunities ``all over the place'', and thus can not focus on one relevant part of the input.} If we consider a typical branching-heuristics for look-ahead solvers (as discussed in Subsection \ref{sec:tawsolver1020}), then we see that locality of the search process is not taken into consideration, and thus for non-tight problem formulations the solver can ``switch attention'' again and again. This is very different from heuristics for conflict-driven solvers, which via ``clause-activity'' have a strong focus on locality of reasoning. Furthermore, look-ahead solvers consider much more of the whole input, for example the \tawsolver{} considers always all remaining variables and their occurrences for the branching heuristic, while conflict-driven solvers do not use such global heuristics.

\begin{table}[H]
  \centering
  \begin{tabular}{c||c|c|c|c|c|c}
    $t=$ & $12$ & $13$ & $14$ & $15$ & $16$ & $17$\\
    \hline\hline
    \ntawsolver & 11 & 83 & 673 & 5,010 & 42,356 & 401,940\\
    & 961,949 & 5,638,667 & 35,085,795 & 194,035,915 & 1,462,429,351 & 10,258,378,909\\
    \hline
    \ttawsolver & 19 & 143 & 1,068 & 7,607 & 59,585\\
    & 953,179 & 5,869,055 & 35,668,687 & 200,208,507 & 1,479,620,647\\
    \hline
    \otawsolver & 47 & 463 & 4,577 & 47,006 & 532,416 \\
    & 1,790,733 & 13,722,975 & 102,268,511 & 774,872,707 & 8,120,609,615\\
    \hline
    \satz{} & 77 & 711 & 6,233 & 54,913 & 562,161\\
    & 262,304 & 1,698,185 & 10,822,316 & 66,595,028 & 599,520,428\\
    \hline
    \march & 185 & 1,849 & 17,018 & 175,614\\
    & 47,963 & 279,061 & 1,975,338 & 11,959,263 \\
    \hline
    \oksolver & 216 & 3,806 & 47,598\\
    & 281,381 & 2,970,723 & 22,470,241\\
    \hline\hline
    \minisat-2.2 & 107 & 1,716 & 16,836 & 190,211\\
    & 5,963,349 & 63,901,998 & 463,984,635 & 3,205,639,994\\
    \hline
    \minisat-2.0 & 273 & 3,022 & 33,391 & 274,457\\
    & 1,454,696 & 9,298,288 & 60,091,581 & 314,678,660\\
    \hline
    \precosat & 211 & 2,777 & 47,624\\
    & 2,425,722 & 16,978,254 & 140,816,236\\
    \hline
    \picosat & 259 & 4,258 & 48,372\\
    & 9,643,671 & 82,811,468 & 576,692,221\\
    \hline
    \glucose-2.0 & 58 & 781 & 84,334\\
    & 1,263,087 & 8,377,487 & 163,500,051\\
    \hline
    \lingeling & 519 & 7,651 & 107,243\\
    & 1,659,607 & 24,124,525 & 176,909,499\\
    \hline
    \cryptomini & 212 & 4,630 & 141,636\\
    & 2,109,106 & 18,137,202 & 205,583,043\\
    \hline
    \glucose-2.2 & 94 & 1,412 & $>$940,040\\
    & 1,444,017 & 10,447,051 & aborted\\
  \end{tabular}
  \caption{Complete solvers on unsatisfiable instances $F(3,t;n)$ for computing $\vdw(2;3,t)$ (with $t=12,\dots,16$ and $n=135, 160, 186, 218, 238$). The first line is run-time in seconds, the second line is the number of nodes for look-ahead solvers resp.\ number of conflicts for conflict-driven solvers.}
  \label{tab:complsolvervdw}
\end{table}

Finally we consider \cubeconq, with the \oksolver{} as Cube-solver and \minisat-2.2 as Conquer-solver, in Table \ref{tab:CCvdW}. We see that the combination is vastly superior to each of the two solvers involved, and approaches in performance the best solver, the \ntawsolver{} (but still slower by a factor of two).

\begin{table}[H]
  \centering
  \begin{tabular}{c||c|c|c|c|c}
    $t=$ & $13$ & $14$ & $15$ & $16$ & $17$\\
    \hline\hline
    $D$ & 20 & 30 & 35 & 40 & 50\\
    nds & 3,197 & 27,053 & 64,663 & 209,593 & 1,399,505\\
    t & 10 & 146 & 821 & 3,248 & 23,546\\
    N & 1599 & 13,527 & 32,331 & 104,797 & 699,751\\
    \hline
    $t$: med, max & $0.06$, $0.49$ & $0.06$, $0.68$ & $0.16$, $3.9$ & $0.46$, $29.6$ & $0.8$, $199$\\
    $\Sigma$ cfs & 8,479,987 & 59,402,586 & 361,511,501 & 3,723,995,162 & 35,931,491,146\\
    $\Sigma$ t & 120 & 961 & 6,888 & 80,056 & 1,006,718\\
    \hline
    total $t$ & 130 & 1,107 & 7,709 & 83,304 & 1,030,264\\
    \hline
    factor & $13.2$ & $15.2$ & $24.7$ & NA & NA
  \end{tabular}
  \caption{\cubeconq, via the \oksolver{} as the cube-solver, and \minisat-2.2 as the conquer-solver. Times are in seconds. ``factor'' is run-time of \minisat-2.2, divided by total time of \cubeconq. The run-times of the \oksolver{} includes writing all data-files (the partial assignments), the run-times of \minisat{} include reading the files. $10^6$ seconds are roughly $11.6$ days.}
  \label{tab:CCvdW}
\end{table}

\subsubsection{Palindromic vdW-problems}
\label{sec:remsatcomppdvdw}

The data for the palindromic problems we considered is shown in Table \ref{tab:pdvdwdata}. Recall that for palindromic problems, that is, the determination of $\vdwpd(2;3,t) = (n_1,n_2)$, we have to determine two numbers: the $n_1$ such that all $\Fpd(3,t,n)$ with $n \le n_1$ are satisfiable, while $\Fpd(3,t,n_1+1)$ is unsatisfiable, and $n_2$ for which $\Fpd(3,t,n)$ is unsatisfiable for all $n \ge n_2$, while $\Fpd(3,t,n_2-1)$ is satisfiable. In order to do so, as shown in Theorem \ref{thm:certificatespd}, the main unsatisfiable instances are for $n_1$ and $n_2+1$. To reduce the amount of data, we don't show the data for these two critical points, but for $n_2$, which is easier than $n_2+1$ (in our range by a factor of around five; possible due to the fact that except of one case $n_2$ happens to be odd here, as discussed in the next paragraph), and harder than $n_1$.

For $\Fpd(3,t,n)$ with odd $n$ we can determine that the middle vertex $\frac{n+1}2$ can not be element of the first block of the partition (belonging to progression-size $3$), since then no other vertex could be in the first block (due to the palindromic property and the symmetric position of the middle vertex), and then we would have an arithmetic progression of size $t$ in the second block. Due to this (and there might be other reasons), palindromic problems for odd $n$ are easier (running times can go up by a factor of $10$ for even $n$).\footnote{We remark that while for example \precosat{} determines this forced variable right at the beginning, this is not the case for the \minisat{} versions, which infer that fact rather late, and they are helped by adding the corresponding unit-clause to the instance.}

\begin{table}[H]
  \centering
  \begin{tabular}{c|c||c|c|c|lllll}
    $t$ & $n$ & $v$ & $c$ & $\ell$ & $c_2$ & $c_3$ & $c_{\ceil{t/2}}$ & $c_{\ceil{t/2}+1}$ & $c_t$\\
    \hline
    17 & 279 & 140 & 10,536 & 45,139 & 185 & 9,357 & 25 & 0 & 969\\
    18 & 312 & 156 & 13,277 & 58,763 & 52 & 11,954 & 9 & 0 & 1,262\\
    19 & 347 & 174 & 16,208 & 70,414 & 230 & 14,586 & 28 & 0 & 1,364\\
    20 & 389 & 195 & 20,327 & 88,944 & 258 & 18,393 & 10 & 19 & 1,647\\
    21 & 405 & 203 & 21,950 & 96,305 & 269 & 19,958 & 29 & 0 & 1,694\\
    22 & 463 & 232 & 28,650 & 126,560 & 308 & 26,171 & 11 & 21 & 2,139\\
    23 & 507 & 254 & 34,289 & 152,236 & 337 & 31,448 & 34 & 0 & 2,470\\
    24 & 593 & 297 & 46,881 & 209,792 & 394 & 43,156 & 12 & 24 & 3,295\\
    25 & 607 & 304 & 48,979 & 219,525 & 404 & 45,237 & 37 & 0 & 3,301\\
    26 & 643 & 322 & 54,843 & 246,503 & 428 & 50,813 & 12 & 24 & 3,566\\
    27 & 699 & 350 & 64,719 & 292,102 & 465 & 60,133 & 38 & 0 & 4,083
  \end{tabular}
  \caption{Instance data for $\Fpd(3,t,n)$, where $v$ is the number of variables, $c=c_2 + c_3 + c_{\ceil{t/2}} + c_{\ceil{t/2}+1} + c_t$ is the number of clauses, $c_i$ the number of clauses of length $i$, and $\ell$ is the number of literal occurrences.}
  \label{tab:pdvdwdata}
\end{table}

First we consider the look-ahead solvers in Table \ref{tab:pdvdwlasolvers}. Comparing \tawsolver{} with the other solvers, we see a similar behaviour as with (ordinary) vdW-problems, but more extreme so. The node-count of \ntawsolver{} and \ttawsolver{} is not much worse than the ``real'' look-ahead solvers, with exception of \march{} (where again a large number of inferred clauses is added by the solver). The weak performance of the \oksolver{} is (likely) explained by the instances not having many $r_2$-reductions (recall that \oksolver{} is \emph{completely} eliminating failed literals, as the only solver), and so the overhead is prohibitive (the savings in node-count don't pay off). \satz{} only investigates $\%10$ of the most promising variables for $r_2$-reductions, and additionally looks for some $r_3$-reductions. This strategy here works far better than \oksolver's ``strategy'' (but the \oksolver{} deliberately doesn't employ a ``strategy'' here, since the aim is to have a stable and ``mathematical meaningful'' solver); nevertheless still the overhead is too large.

An interesting aspect is that for larger $t$ the more complex heuristic (i.e., projection) of \ttawsolver{} compared to \ntawsolver{} pays off. This is different from ordinary vdW-problems. And as the comparison with \otawsolver{} shows, the heuristic (mostly the projection) is of great importance here (more pronounced than for ordinary vdW-problems).

\begin{table}[H]
  \centering
  \begin{tabular}{c||c|c|c|c|c|c}
    $t$ & \ttawsolver & \ntawsolver & \satz & \otawsolver &   \march & \oksolver{}\\
    \hline\hline
    17 & 1 & 0.8 & 12 & 7 & 35 & 18\\
    & 32,855 & 32,697 & 16,466 & 143,319 & 1,448 & 5,023\\
    \hline
    18 & 11 & 8 & 182 & 60 & 269 & 335\\
    & 276,249 & 279,309 & 208,873 & 1,063,979 & 12,289 & 100,803\\
    \hline
    19 & 13 & 10 & 143 & 134 & 500 & 322\\
    & 283,229 & 285,037 & 123,199 & 2,009,635 & 12,423 & 62,009\\
    \hline
    20 & 48 & 39 & 701 & 738 & 1,980 & 1,419\\
    & 894,777 & 897,529 & 459,899 & 9,076,261 & 39,681 & 206,617\\
    \hline
    21 & 115 & 101 & 2,592 & 2,541 & 5,053 & 3,536\\
    & 2,144,743 & 2,239,371 & 1,567,736 & 30,470,349 & 99,493 & 490,841\\
    \hline
    22 & 564 & 525 & 9,418 & 18,306 & 25,841 & 47,593\\
    & 8,427,503 & 8,683,035 & 4,393,139 & 170,414,771 & 376,285 & 3,197,173\\
    \hline
    23 & 1,547 & 1,695 & 35,633 & 86,869 & 77,763 & 132,150\\
    & 19,858,971 & 21,565,129 & 12,587,868 & 573,190,251 & 876,315 & 7,461,907\\
    \hline
    24 & 8,558 & 26,724\\
    & 79,790,419 & 198,685,857\\
    \hline
    25 & 22,841\\
    & 219,575,127
  \end{tabular}
  \caption{Look-ahead solvers on unsatisfiable instances $\Fpd(3,t;n)$ for computing $\vdw(2;3,t)$ (with $t=17,\dots,25$ and $n=279, \dots, 607$). The first line is run-time in seconds, the second line is the number of nodes.}
  \label{tab:pdvdwlasolvers}
\end{table}

The conflict-driven solvers are shown in Table \ref{tab:pdvdwcdsolvers}. We see that they are not competitive with \ntawsolver{} or \ttawsolver, however now most of them are better than the ``real'' look-ahead solvers. Here \minisat-2.2 is better than \minisat-2.0, and \glucose-2.2 is better than \glucose-2.0, so we show only data for the newest versions. With \glucose{} we see a pattern which we observed also at other (hard) instance classes: for smaller instances \glucose{} is better than \minisat{}, but from a certain point on the performance of \glucose{} becomes very bad. This is likely due to the more aggressive restart strategy, which pays off for smaller instances, but from a certain point on the solver becomes essentially incomplete.

\begin{table}[H]
  \centering
  \begin{tabular}{c||c|c|c|c|c}
    $t$ & \minisat & \glucose & \precosat & \lingeling & \cryptomini\\
    \hline\hline
    17 & 0.8 & 0.8 & 1.2 & 3.7 & 3.6\\
    & 34,426 & 34,826 & 41,961 & 57,306 & 59,443\\
    \hline
    18 & 19 & 14 & 25 & 59 & 78\\
    & 607,908 & 340,568 & 506,793 & 919,123 & 871,916\\
    \hline
    19 & 19 & 15 & 24 & 61 & 72\\
    & 568,924 & 336,861 & 485,357 & 915,107 & 765,301\\
    \hline
    20 & 118 & 66 & 131 & 355 & 384\\
    & 2,852,150 & 1,132,012 & 1,799,145 & 3,633,502 & 3,071,462\\
    \hline
    21 & 423 & 228 & 445 & 1,060 & 1,418\\
    & 9,179,642 & 2,903,573 & 4,687,589 & 8,672,073 & 8,458,496\\
    \hline
    22 & 3,151 & 1,631 & 2,825 & 8,428 & 14,321\\
    & 51,582,064 & 13,397,451 & 22,283,651 & 41,696,062 & 49,716,762\\
    \hline
    23 & 8,191 & 6,817 & 9,280 & 28,543 & 55,544\\
    & 108,028,217 & 36,314,064 & 54,951,563 & 104,007,799 & 141,249,316\\
    \hline
    24 & 54,678 & $>$ 992,540 & 82,750 & 152,076\\
    & 476,716,936 & $>$ 1,100,664,795 & 261,084,988 & 285,546,948 & \\
    && aborted
  \end{tabular}
  \caption{Conflict-driven solvers on unsatisfiable instances $\Fpd(3,t;n)$ for computing $\vdw(2;3,t)$ (with $t=17,\dots,24$ and $n=279, \dots, 593$). The first line is run-time in seconds, the second line is the number of conflicts.}
  \label{tab:pdvdwcdsolvers}
\end{table}

Finally we consider \cubeconq{} in Table \ref{tab:CCpdvdW}. We see that this is now the fastest solver overall. \glucose-2.2 is $\%10$ faster, but since this is only a small amount, for consistency we stick with \minisat-2.2.

\begin{table}[H]
  \centering
  \begin{tabular}{c||c|c|c|c|c}
    $t=$ & $23$ & $24$ & $25$ & $26$ & $27$\\
    \hline\hline
    $D$ & 25 & 35 & 45 & 55 & 65\\
    nds & 1,717 & 5,559 & 17,633 & 77,161 & 220,069\\
    t & 106 & 500 & 1,752 & 7,889 & 25,478\\
    N & 859 & 2,780 & 8,817 & 38,581 & 110,032\\
    \hline
    $t$: med, max & $0.95$, $17.6$ & $1.2$, $27$ & $0.81$, $47$ & $0.95$, $58$ & $0.82$, $125$\\
    $\Sigma$ cfs & 27,308,572 & 93,831,664 & 258,829,555 & 1,231,383,588 & 3,423,841,749\\
    $\Sigma$ t & 1,095 & 4,466 & 11,822 & 55,306 & 172,033\\
    \hline
    total $t$ & 1,201 & 4,966 & 13,574 & 63,195 & 197,511\\
    \hline
    factor & $1.3$ & $1.7$ & $1.7$ & NA & NA
  \end{tabular}
  \caption{\cubeconq, via the \oksolver{} as the cube-solver, and \minisat-2.2 as the conquer-solver. Times are in seconds. ``factor'' is run-time of best solver, i.e., \ttawsolver, divided by total time of \cubeconq. $10^5$ seconds are roughly $1.2$ days.}
  \label{tab:CCpdvdW}
\end{table}

\subsection{Incomplete solvers (stochastic local search)}
\label{sec:remsatincomp}

In the \OKlibrary{} we use the \texttt{Ubcsat} suite (see \cite{ubcsat}) of local-search algorithms in version 1-2-0. The considered algorithms are GSAT, GWSAT, GSAT-TABU, HSAT, HWSAT, WALKSAT, WALKSAT-TABU, WALKSAT-TABU-NoNull, Novelty, Novelty+, Novelty++, Novelty+p, Adaptive Novelty+, RNovelty, RNovelty+, SAPS, RSAPS, SAPS/NR, PAWS, DDFW, G2WSAT, Adaptive G2WSat, VW1, VW2, RoTS, IRoTS, SAMD. The performance of local-search algorithms is very much instance-dependent, and so a good choice of algorithms is essential. Our experiments yield the following selection criteria:
\begin{itemize}
\item For standard problems (Section \ref{sec-comp}) the best advice seems to use GSAT-TABU for $t \leq 23$, to use RoTS for $t > 23$, and to use Adaptive G2WSat for $t > 33$ (also trying DDFW then).
\item For the palindromic problems (Section \ref{sec:Palindromes}) GSAT-TABU is the best algorithm.
\end{itemize}
For a given $t$ in principle we let these algorithm run for $n=t+1,t+2, \dots, $, until the search seems unable to find a solution. But running these algorithms from scratch on these vdW-problems is much less effective than using an incremental approach, based on a solution found for $n-1$, respectively for palindromic vdW-problems on a solution found for $n-2$ (according to Lemma \ref{lem:transferpdpartitions}), as initial guess, and repeating this process for the next $n$: this helps to go much quicker through the easier part of the search space (of possible $n$), and also seems to help for the harder problems. Finally, we recall that in Subsection \ref{sec:Somenewconjectures} we explained how we made the distinction between lower bounds we conjecture to be exact and sheer lower bounds.

\section{Conclusion}
\label{sec:Conclusion}

This article presented the following contributions to the fields of Ramsey theory and SAT solving:
\begin{itemize}
\item Study of $\vdw(2;3,t)$:
  \begin{enumerate}
  \item determination of $\vdw(2;3,19) = 349$;
  \item lower bounds for $\vdw(2;3,t)$ with $20 \leq t \leq 30$, conjectured to be exact;
  \item further lower bounds for $31 \leq t \leq 39$;
  \item improved conjecture on the growth rate of $\vdw(2;3,t)$;
  \item various observations on structural properties of good partitions.
  \end{enumerate}
\item Introduction and study of $\vdwpd(2;3,t)$:
  \begin{enumerate}
  \item basic definitions and properties;
  \item determination of $\vdwpd(2;3,t)$ for $t \leq 27$;
  \item lower bounds for $\vdwpd(2;3,t)$ with $28 \leq t \leq 35$, conjectured to be exact;
  \item further lower bounds for $36 \leq t \leq 39$.
  \end{enumerate}
\item SAT solving:
  \begin{enumerate}
  \item introduced the new SAT-solver \tawsolver, with the basic implementation given by \otawsolver, and the versions with improved heuristic by \ntawsolver{} and \ttawsolver;
  \item experimental comparison with current look-ahead and conflict-driven solvers;
  \item comparison and data for the new \cubeconq{} method;
  \item experimental determination of good local-search algorithms for lower bounds.
  \end{enumerate}
\end{itemize}
We hope that these investigations contribute to a better understanding of the connections between Ramsey theory and SAT solving. The following seem relevant research directions for future investigations:
\begin{itemize}
\item Showing $\vdw(2;3,20) = 389$ (recall Subsection \ref{sec:Somenewconjectures}) should be in reach with \ntawsolver, while showing $\vdw(2;3,21) = 416$ seems to require new (algorithmic) insight (when using similar computational resources).
\item Conjecture \ref{conj-1} states that the lower bound from \cite{blr2008} for $\vdw(2;3,t)$ is tight up to a small factor.
\item In Section \ref{sec-pat} four conjectures on patterns in good partitions are presented (one implying Conjecture \ref{conj-1}).
\item In Subsection \ref{sec:pdopen} various open problems on palindromic van der Waerden numbers are stated.
\item Considering SAT solving:
  \begin{enumerate}
  \item Understand the differences between ordinary and palindromic problems:
    \begin{itemize}
    \item Why is the projection relatively more important for the palindromic problems? (So that the difference between \ntawsolver{} and \otawsolver{} is more pronounced, and \ttawsolver{} becomes faster than \ntawsolver{} on bigger instances.)
    \item Why do we have different behaviour of look-ahead versus conflict-driven solvers?
    \end{itemize}
  \item Can the branching heuristic of \tawsolver{} for the instances of this paper be much further improved? Especially can we gain some understanding of the weights?
  \item How to understand the success of \cubeconq{} ? Does its success indicate that there are important dag-like structures in good resolution refutations of the instances of these classes, which are dispersed locally, so that ordinary conflict-driven solvers have problems exploiting them?
  \end{enumerate}
\end{itemize}

\section*{Acknowledgements}
The authors would like to thank Donald Knuth, the Editor and the anonymous referees for their valuable suggestions and helpful comments.

\bibliographystyle{plain}


\begin{appendix}

\section{Certificates}
\label{sec:Certificates}

\subsection{Conjectured precise lower bounds for $\vdw(2;3,t)$}
\label{sec:Certificates3tp}

The certificate for $t=20$ is also palindromic (while for $t=24$ a palindromic certificate is given in Subsection \ref{sec:goodpp}).

\paragraph{$\vdw(2; 3, 20)\geq 389$}

$$1^{19}01^{11}01^{4}01^{7}0101^{4}01^{13}01^{9}01^{4}0101^{14}0^{2}1^{3}0^{2}101^{9}01^{18}01^{9}01^{4}0101^{5}01^{3}01^{10}01^{16}01^{8}01^{16}01^{10}01^{3}01^{5}0101^{4}01^{9}$$
$$01^{18}01^{9}010^{2}1^{3}0^{2}1^{14}0101^{4}01^{9}01^{13}01^{4}0101^{7}01^{4}01^{11}01^{19}$$

\paragraph{$\vdw(2; 3, 21)\geq 416$}

$$1^{8}01^{17}0101^{6}0101^{9}01^{12}0^{2}1^{19}01^{18}01^{3}0101^{2}01^{12}0^{2}1^{5}0^{2}101^{10}01^{18}01^{3}010^{2}1^{8}01^{12}01^{20}01^{18}01^{3}0101^{7}01^{8}01^{12}0^{2}$$
$$101^{17}01^{18}01^{3}010^{2}1^{8}01^{5}01^{6}0^{2}1^{12}01^{6}01^{15}0101^{4}0101^{11}01^{20}$$

\paragraph{$\vdw(2; 3, 22)\geq 464$}

$$1^{2}0^{2}1^{17}0^{2}1^{9}01^{12}0101^{12}01^{15}01^{2}01^{10}01^{4}01^{7}01^{5}01^{12}01^{9}010
1^{3}01^{8}01^{3}0101^{9}01^{12}01^{5}01^{7}01^{4}01^{10}01^{2}01^{15}01^{12}$$
$$0101^{12}01^{9}0^{2}1^{17}0^{2}1^{9}01^{12}0101^{12}01^{15}01^{2}01^{10}01^{4}01^{7}01^{5}0
1^{12}01^{9}0101^{3}01^{8}01^{5}01^{9}01^{12}01^{5}01^{7}01^{11}$$

\paragraph{$\vdw(2; 3, 23)\geq 516$}

$$1^{4}0^{2}1^{2}01^{17}01^{4}0101^{15}01^{16}01^{4}01^{5}01^{20}0101^{2}01^{8}0^{2}101^{4}0
1^{15}01^{2}01^{4}01^{16}01^{9}01^{10}0101^{9}01^{10}01^{7}01^{17}01^{6}0101^{19}01^{16}0$$
$$1^{21}0^{2}1^{19}01^{6}01^{2}01^{12}010^{2}1^{4}0101^{20}01^{13}0^{2}1^{11}01^{9}0^{2}1^{6}0
1^{4}01^{13}0101^{3}01^{8}01^{9}01^{20}01^{5}01^{18}01^{3}01$$

\paragraph{$\vdw(2; 3, 24)\geq 593$}

$$1^{21}01^{18}01^{16}01^{4}01^{7}01^{6}0101^{14}01^{3}0^{2}1^{8}01^{7}01^{3}01^{2}01^{20}0^{2}
1^{3}01^{7}01^{15}01^{7}01^{3}0^{2}1^{20}01^{2}01^{3}01^{7}01^{9}01^{18}0101^{6}01^{21}01^{7}0$$
$$1^{10}01^{7}01^{21}01^{6}0101^{18}01^{9}01^{7}01^{3}01^{2}01^{20}0^{2}1^{3}01^{7}01^{15}01^{7}0
1^{3}0^{2}1^{20}01^{2}01^{3}01^{7}01^{8}0^{2}1^{3}01^{14}0101^{6}01^{7}01^{4}01^{16}01^{18}01^{21}$$

\paragraph{$\vdw(2; 3, 25)\geq 656$}
$$1^{16}01^{2}01^{19}01^{8}01^{7}01^{19}01^{8}01^{4}0101^{5}01^{17}01^{10}01^{21}0^{2}
1^{2}0101^{10}01^{7}01^{12}01^{4}01^{3}01^{6}01^{7}01^{11}01^{3}01^{13}01^{4}01^{17}0101^{3}$$
$$01^{6}0^{2}1^{6}01^{17}01^{8}01^{7}01^{13}01^{14}01^{2}01^{4}01^{13}0^{2}1^{8}01^{7}0
1^{19}01^{8}01^{4}0101^{7}01^{15}01^{10}01^{9}01^{11}0^{2}1^{2}0101^{10}01^{5}01^{14}0$$
$$1^{4}01^{3}01^{6}01^{7}01^{11}01^{3}01^{13}01^{4}01^{17}0101^{3}01^{6}0^{2}1^{24}01^{8}01^{9}$$

\paragraph{$\vdw(2; 3, 26)\geq 727$}
$$1^{10}01^{23}01^{10}0101^{20}01^{4}01^{11}01^{6}01^{11}0^{2}1^{2}01^{5}01^{6}01^{5}01^{23}0
1^{4}0101^{7}01^{17}01^{16}0101^{11}0^{2}1^{2}0101^{3}01^{4}01^{2}01^{18}$$
$$01^{3}01^{5}01^{14}01^{12}01^{16}01^{4}01^{19}01^{8}010^{2}1^{4}01^{13}01^{14}0101^{20}0
1^{4}01^{18}01^{11}0^{2}1^{2}01^{5}01^{6}01^{5}01^{23}01^{4}0101^{7}01^{15}0101^{16}$$
$$0101^{11}0^{2}1^{2}0101^{3}01^{4}01^{2}01^{18}01^{9}01^{14}01^{12}0
1^{16}01^{4}01^{19}01^{8}01^{2}01^{18}01^{3}01^{25}$$

\paragraph{$\vdw(2; 3, 27)\geq 770$}
$$1^{24}01^{3}01^{5}01^{18}01^{17}0101^{2}01^{21}01^{3}01^{7}01^{2}01^{20}0^{2}1^{12}0
1^{15}01^{10}01^{11}01^{3}01^{9}01^{6}01^{13}01^{22}01^{3}0^{2}$$
$$1^{8}01^{5}01^{20}01^{6}0101^{16}01^{7}01^{3}01^{5}010^{2}101^{21}01^{2}0^{2}1^{14}01^{9}0
1^{17}0101^{2}01^{3}01^{17}01^{3}01^{5}0101^{2}01^{20}0^{2}1^{12}01^{15}01^{22}$$
$$01^{3}01^{9}01^{6}01^{13}01^{22}01^{3}0^{2}1^{8}01^{5}01^{20}01^{8}01^{16}01^{7}0
1^{3}01^{5}010^{2}101^{24}01^{6}01^{8}01^{20}01^{15}01^{16}01^{5}$$

\paragraph{$\vdw(2; 3, 28)\geq 827$}
$$1^{27}01^{10}01^{22}0101^{16}01^{13}01^{16}01^{20}01^{4}01^{16}0101^{11}0^{2}1^{2}010
1^{3}01^{6}0^{2}1^{18}01^{3}01^{5}01^{14}01^{12}01^{21}01^{14}0$$
$$1^{13}010^{2}1^{4}01^{23}01^{4}0101^{25}01^{18}01^{11}01^{3}0101^{3}01^{4}01^{7}0
1^{17}01^{10}0101^{7}01^{17}01^{11}01^{4}01^{13}0^{2}1^{2}0101^{3}01^{7}$$
$$01^{18}01^{3}01^{5}01^{14}01^{12}01^{16}01^{4}01^{14}01^{13}010^{2}1^{18}01^{9}01^{4}010
1^{25}01^{18}01^{11}01^{9}01^{4}0101^{23}01^{10}01^{9}01^{12}01^{16}01^{10}$$

\paragraph{$\vdw(2; 3, 29)\geq 868$}
$$1^{15}01^{21}01^{18}01^{17}01^{18}01^{21}01^{2}01^{7}01^{25}0101^{12}01^{17}01^{6}01^{7}0^{2}
1^{2}01^{17}01^{5}0^{2}101^{10}01^{6}01^{13}0101^{2}01^{17}01^{16}01^{7}01^{4}$$
$$01^{20}0101^{2}01^{5}01^{11}01^{25}01^{10}01^{25}01^{2}0^{2}1^{16}0
1^{3}01^{10}01^{4}01^{20}0101^{2}01^{6}01^{20}01^{14}0^{2}1^{2}01^{17}01^{5}0^{2}101^{10}01^{6}01^{13}0101^{2}$$
$$01^{9}01^{7}01^{9}01^{6}01^{7}0
1^{4}01^{20}0101^{2}01^{17}01^{25}01^{2}01^{7}01^{25}01^{3}01^{16}01^{6}01^{12}01^{25}01^{6}01^{27}01^{8}$$

\paragraph{$\vdw(2; 3, 30) \geq 903$}

$$1^{22}01^{16}01^{22}01^{26}01^{7}0101^{8}0101^{22}01^{12}01^{16}01^{7}01^{6}01^{9}01^{11}01^{6}0^{2}1^{15}0^{2}1^{13}01^{7}01^{8}0101^{23}01^{6}01^{28}01^{7}01^{3}01^{4}01^{22}$$
$$01^{8}0101^{2}0^{2}1^{11}0101^{22}01^{11}01^{14}01^{3}01^{19}01^{4}01^{16}0^{2}1^{11}0101^{5}01^{28}01^{6}0^{2}101^{10}01^{2}01^{14}01^{7}01^{10}01^{23}01^{6}01^{28}$$
$$01^{7}01^{5}01^{2}0^{2}1^{18}01^{2}01^{8}0101^{3}01^{11}0101^{16}01^{25}0101^{4}01^{23}01^{16}01^{4}0^{2}1^{13}01^{12}01^{3}01^{27}01^{10}01^{14}$$

\normalsize

\subsection{Further lower bounds for $\vdw(2;3,t)$}
\label{sec:Certificates3tf}

The certificate for $t=31$ is also palindromic.

\paragraph{$\vdw(2; 3, 31) > 930$}

\begin{gather*}
  1^{12}01^{19}01^{16}01^{10}01^{9}01^{3}01^{4}01^{14}01^{12}01^{14}01^{18}0^{2}1^{10}01^{4}01^{13}01^{8}01^{5}0101^{10}01^{17}01^{15}01^{16}01^{8}01^{5}01^{8}0101^{4}01^{8}01^{19}\\
  01^{16}01^{5}01^{18}01^{4}01^{25}01^{14}01^{16}01^{3}01^{9}01^{6}01^{13}01^{11}01^{2}01^{11}01^{13}01^{6}01^{9}01^{3}01^{16}01^{14}01^{25}01^{4}01^{18}01^{5}01^{16}01^{19}01^{8}\\
  01^{4}0101^{8}01^{5}01^{8}01^{16}01^{15}01^{17}01^{10}0101^{5}01^{8}01^{13}01^{4}01^{10}0^{2}1^{18}01^{14}01^{12}01^{14}01^{4}01^{3}01^{9}01^{10}01^{16}01^{19}01^{12}
\end{gather*}

\paragraph{$\vdw(2; 3, 32) > 1006$}

\begin{gather*}
  1^{15}01^{26}01^{5}01^{11}01^{4}01^{14}01^{4}01^{16}01^{28}01^{22}01^{2}01^{5}01^{7}01^{19}01^{2}0101^{16}01^{13}01^{5}01^{3}01^{7}01^{4}01^{13}01^{9}01^{14}01^{28}01^{7}01^{15}01^{10}\\
  010^{2}1^{18}0^{2}1^{2}01^{13}01^{6}01^{21}01^{27}01^{2}0101^{8}01^{7}01^{14}01^{13}01^{2}01^{24}01^{10}01^{5}0101^{17}01^{10}0^{2}1^{4}01^{28}01^{9}01^{11}0^{2}1^{8}01^{6}01^{7}0\\
  1^{11}01^{3}01^{10}010^{2}1^{28}01^{14}01^{17}01^{3}01^{12}01^{13}01^{5}01^{16}01^{13}01^{3}01^{10}01^{22}01^{2}01^{12}01^{7}01^{28}01^{20}01^{2}01^{5}01^{4}01^{22}01^{8}01^{28}
\end{gather*}

\paragraph{$\vdw(2; 3, 33) > 1063$}

\begin{gather*}
  1^{29}01^{6}01^{14}01^{11}01^{21}01^{14}01^{2}01^{18}0101^{15}01^{12}01^{25}01^{16}0101^{11}0^{2}1^{4}01^{10}01^{5}01^{13}01^{16}01^{20}01^{2}01^{13}01^{22}01^{5}01^{4}0101^{19}01^{9}0\\
  1^{4}01^{9}01^{15}0101^{18}01^{11}0^{2}1^{4}01^{3}01^{4}01^{2}01^{28}01^{6}01^{7}01^{29}01^{4}01^{19}01^{10}0^{2}1^{18}01^{14}0101^{5}01^{14}01^{16}01^{28}01^{6}01^{20}01^{8}01^{4}01^{7}\\
  01^{14}01^{4}01^{18}01^{11}01^{9}01^{4}01^{2}01^{4}01^{17}0101^{3}01^{14}01^{29}01^{18}0^{2}1^{2}01^{10}01^{2}01^{19}01^{2}01^{10}01^{27}01^{18}01^{12}01^{4}01^{20}01^{9}01^{24}01^{13}
\end{gather*}

\paragraph{$\vdw(2; 3, 34) > 1143$}

\begin{gather*}
  1^{32}01^{7}0101^{8}01^{29}01^{23}01^{3}0^{2}1^{10}01^{3}01^{17}01^{2}01^{9}01^{5}01^{16}01^{15}01^{20}0^{2}1^{2}01^{8}0^{2}101^{31}01^{18}01^{6}01^{12}01^{6}01^{17}01^{3}01^{7}0^{2}1^{26}01^{6}\\
  01^{2}0101^{24}0^{2}1^{7}01^{2}01^{3}01^{21}01^{24}01^{10}01^{11}01^{20}01^{10}01^{18}01^{6}0^{2}1^{28}01^{7}01^{3}01^{14}01^{27}010^{2}1^{20}0101^{4}01^{7}01^{2}01^{20}0^{2}1^{28}01^{7}01^{18}0\\
  1^{9}01^{10}01^{8}01^{28}01^{5}010^{2}101^{3}01^{27}01^{2}01^{5}01^{32}01^{24}01^{6}01^{25}01^{10}01^{2}01^{23}01^{2}0101^{9}01^{13}01^{10}01^{11}01^{20}01^{12}01^{16}01^{7}01^{3}01^{33}
\end{gather*}

\paragraph{$\vdw(2; 3, 35) > 1204$}

\begin{gather*}
  1^{34}01^{24}01^{8}01^{22}01^{3}0^{2}1^{10}01^{29}01^{7}01^{4}01^{14}01^{15}01^{7}01^{21}01^{6}01^{5}01^{13}01^{22}0101^{12}01^{2}01^{12}01^{7}01^{11}0^{2}1^{4}01^{3}01^{7}01^{18}01^{5}01^{3}01^{4}0\\
  1^{2}01^{6}01^{34}01^{19}01^{8}0101^{5}01^{17}01^{10}01^{2}01^{18}01^{3}01^{18}0101^{29}01^{18}01^{3}01^{32}0^{2}1^{2}01^{5}01^{27}01^{16}01^{24}01^{3}01^{12}0101^{17}01^{8}0101^{20}01^{13}01^{22}\\
  01^{8}01^{6}01^{7}01^{21}01^{12}01^{19}01^{11}01^{22}0101^{8}0101^{5}01^{13}01^{14}01^{9}01^{15}0101^{10}01^{7}01^{15}01^{13}01^{4}01^{13}01^{24}01^{12}01^{21}01^{23}01^{6}0^{2}1^{24}01^{34}
\end{gather*}

\paragraph{$\vdw(2; 3, 36) > 1257$}

\begin{gather*}
  1^{10}01^{33}01^{16}01^{12}01^{6}01^{11}01^{25}0101^{4}01^{5}01^{2}01^{12}01^{6}01^{29}0101^{4}01^{17}01^{26}01^{21}0101^{12}01^{3}0101^{2}0^{2}1^{13}01^{5}01^{16}01^{32}01^{2}01^{19}0^{2}10\\
  1^{4}01^{3}01^{25}01^{10}01^{32}01^{5}01^{2}01^{14}01^{7}01^{10}01^{17}01^{16}01^{4}01^{20}01^{16}01^{13}01^{5}01^{12}01^{3}01^{11}0101^{34}01^{9}01^{10}01^{30}01^{3}01^{5}01^{26}01^{9}01^{11}0\\
  1^{6}01^{2}01^{33}01^{2}01^{10}01^{2}01^{20}01^{6}01^{3}01^{24}01^{16}01^{19}0^{2}101^{4}01^{3}01^{17}01^{7}01^{16}01^{21}01^{26}0101^{5}01^{16}01^{26}01^{3}01^{4}0^{2}1^{13}01^{4}01^{2}01^{8}010\\
  1^{22}01^{16}01^{20}01^{28}0101^{33}
\end{gather*}

\paragraph{$\vdw(2; 3, 37) > 1338$}

\begin{gather*}
  1^{5}01^{33}01^{22}01^{2}01^{12}0^{2}1^{18}01^{14}01^{9}01^{12}01^{15}01^{24}01^{19}01^{17}01^{7}0^{2}1^{7}0101^{12}01^{2}01^{10}01^{22}01^{8}01^{32}01^{3}01^{6}01^{19}01^{14}01^{21}0^{2}1^{2}01^{12}\\
  01^{7}01^{26}01^{22}01^{2}0101^{3}01^{32}01^{3}01^{14}01^{11}01^{17}01^{18}0^{2}1^{2}0101^{3}01^{7}01^{22}01^{32}01^{22}0101^{28}0^{2}1^{4}01^{31}01^{6}01^{17}01^{3}01^{14}01^{21}01^{4}01^{7}\\
  01^{14}01^{35}01^{5}01^{16}01^{19}01^{3}01^{4}0101^{23}01^{10}01^{31}01^{12}0101^{11}01^{3}01^{12}01^{5}01^{14}01^{2}01^{10}01^{27}01^{30}0^{2}1^{2}0101^{8}01^{2}01^{4}01^{27}010^{2}1^{7}0\\
  1^{29}01^{4}01^{14}01^{16}01^{18}01^{14}01^{21}0^{2}1^{11}01^{5}
\end{gather*}

\paragraph{$\vdw(2; 3, 38) > 1378$}

\begin{gather*}
  1^{34}01^{14}01^{7}01^{13}01^{22}0101^{12}01^{23}01^{12}01^{13}010^{2}1^{28}01^{4}01^{2}01^{24}01^{10}01^{7}01^{29}01^{4}01^{19}01^{8}01^{2}01^{18}0^{2}1^{21}01^{14}01^{13}01^{7}01^{17}\\
  0^{2}1^{4}01^{26}01^{3}0101^{10}01^{20}01^{16}01^{18}01^{9}01^{12}0101^{21}01^{6}01^{37}01^{6}01^{15}01^{13}01^{6}01^{12}01^{8}01^{32}01^{4}01^{5}01^{19}01^{3}01^{10}01^{21}0^{2}1^{4}01^{18}\\
  01^{28}01^{37}01^{6}01^{11}01^{3}01^{10}01^{7}01^{13}0^{2}1^{4}01^{34}01^{5}01^{4}01^{25}01^{5}0^{2}1^{3}01^{27}01^{10}01^{5}01^{23}01^{4}01^{22}01^{8}01^{12}01^{16}01^{13}0^{2}1^{4}01^{20}0\\
  1^{10}01^{33}01^{7}0101^{14}01^{5}01^{27}01^{7}0101^{6}0^{2}1^{19}01^{16}01^{30}
\end{gather*}

\paragraph{$\vdw(2; 3, 39) > 1418$}

\begin{gather*}
  1^{2}01^{4}01^{13}01^{34}0101^{20}01^{21}01^{19}01^{7}01^{2}01^{36}01^{20}01^{2}0^{2}1^{8}01^{24}01^{12}01^{5}01^{27}01^{18}01^{11}01^{7}01^{6}01^{30}01^{5}01^{16}01^{4}01^{26}01^{4}\\
  0^{2}1^{28}01^{2}01^{4}01^{2}01^{6}01^{18}01^{26}01^{17}01^{3}01^{7}01^{16}01^{11}01^{5}01^{2}01^{22}01^{3}01^{6}01^{4}01^{36}01^{24}01^{8}01^{16}01^{6}01^{13}0101^{34}01^{26}01^{36}\\
  01^{6}01^{2}01^{5}01^{16}01^{15}01^{3}01^{9}01^{7}01^{11}01^{2}01^{24}01^{15}01^{27}0101^{2}01^{3}01^{9}01^{17}010^{2}1^{36}01^{22}0^{2}1^{11}01^{24}01^{28}01^{6}01^{29}0^{2}1^{5}0^{2}\\
  1^{3}01^{32}01^{2}0^{2}1^{5}01^{2}01^{15}01^{4}01^{7}01^{4}01^{29}01^{8}01^{22}01^{4}01^{32}01^{12}01^{30}
\end{gather*}

\normalsize

\subsection{Good palindromic partitions}
\label{sec:goodpp}

Table \ref{tab-pal} gives good palindromic partitions for $n_1-1, n_2 - 1$ with $(n_1,n_2) = \vdwpd(2;3,t)$, $3 \leq t \leq 39$, according to the values in Tables \ref{tab-pdprecise}, \ref{tab-pdconj}. Due to the palindromic property, for the corresponding $n$ we only show the partition of $\set{1,\dots,\ceil{\frac n2}}$, so that for example the good palindromic partition $01^2001^20$ for $t=3$ and $n=8$ is compressed to $01^20$. Note that such compressed partitions correspond exactly to the solution of $\Fpd(3,t,n)$ as defined in Lemma \ref{lem:pdsat}.

{
\scriptsize
\renewcommand{\arraystretch}{2}
\begin{longtable}[c]{|c|c|c|}
\caption{Good palindromic partitions for $\vdwpd(2; 3, t)-1$ according to Theorem \ref{thm:certificatespd}, Part (i)}
\label{tab-pal}\\
  \hline \hline
 $t$ & $\vdwpd(2;3,t)-1$ &{\tt Good palindromic partitions}\\
  \hline
\endhead
\hline
  \multicolumn{3}{|r|}{{Continued on Next Page\ldots}} \\
\hline
\endfoot
\endlastfoot
$3$ & $(5,8)$ & $1^{2}0, \ 10^{2}1$\\
\hline
$4$ & $(14,15)$ & $0101^{3}0, \ 1^{2}010^{2}1^{2}$\\
\hline
$5$ & $(15,20)$ & $1^{4}0^{2}1^{2}, \ 1^{2}0101^{4}0$\\
\hline
$6$ & $(29,30)$ & $1^{2}01^{5}0^{2}1^{3}01, \ 1^{5}0^{2}1^{5}0^{2}1$\\
\hline
$7$ & $(40,43)$ & $01^{4}0101^{3}0^{2}1^{5}01, \ 0^{2}1^{4}01^{2}01^{5}01^{4}01$\\
\hline
$8$ & $(51,56)$ & $1^{7}01^{2}01^{3}0^{2}101^{4}01^{3}, \ 1^{2}01^{2}01^{4}0101^{4}01^{3}01^{5}0$\\
\hline
$9$ & $(61,76)$ & $1^{4}01^{3}0^{2}1^{5}01^{2}01^{5}01^{4}01, \ 1^{8}0^{2}1^{6}01^{3}0101^{3}0^{2}1^{5}01^{4}$\\
\hline
$10$ & $(92,93)$ & $1^{9}01^{4}010^{2}1^{9}0^{2}10^{2}1^{9}01^{4}, \ 1^{9}01^{8}0^{2}1^{2}0^{2}1^{6}0^{2}1^{4}01^{8}01$\\
\hline
$11$ & $(109,112)$ & $1^{3}01^{3}01^{9}010^{2}1^{8}0^{2}1^{3}0101^{8}01^{4}01^{4}, \ 1^{2}0101^{3}0101^{4}01^{7}0^{2}1^{5}01^{3}01^{6}0101^{10}01$\\
\hline
$12$ & $(125,134)$ & $1^{11}0101^{8}010^{2}101^{10}01^{8}01^{5}01^{4}0101^{2}, \ 1^{9}01^{8}01^{9}01^{2}01^{3}0101^{7}01^{2}0101^{3}01^{11}0$\\
\hline
$13$ & $(141,154)$ & $101^{2}01^{11}01^{10}010^{2}1^{4}01^{10}01^{11}010^{2}1^{6}01^{2}, \ 1^{10}01^{4}01^{11}01^{4}01^{10}010^{2}1^{3}0^{2}101^{10}01^{4}01^{2}01^{4}$\\
\hline
$14$ & $(173,182)$ & $1^{2}01^{7}01^{6}01^{3}01^{7}01^{6}01^{10}0101^{10}01^{5}0101^{6}01^{2}0^{2}1^{7}, \ 1^{4}0^{2}1^{2}01^{8}0101^{7}01^{8}01^{5}01^{8}01^{7}0101^{8}01^{2}0^{2}1^{13}01^{2}$\\
\hline
$15$ & $199$ & $1^{5}01^{10}01^{7}0^{2}1^{4}01^{2}01^{5}01^{7}01^{8}01^{10}01^{3}01^{12}01^{5}01^{2}01^{4}01$\\[-1ex]
& $204$ & $1^{11}01^{6}01^{9}01^{2}01^{11}01^{3}01^{5}01^{2}0^{2}1^{2}01^{11}0101^{10}01^{13}01^{2}$\\
\hline
$16$ & $231$ & $1^{14}01^{9}01^{6}01^{5}01^{2}0^{2}1^{4}01^{12}01^{3}01^{7}01^{2}0^{2}1^{12}01^{8}01^{5}01^{10}01$\\[-1ex]
& $236$ & $1^{15}01^{9}01^{2}0^{2}1^{4}01^{9}010^{2}1^{9}01^{14}01^{9}0^{2}1^{11}01^{5}01^{2}01^{9}01^{3}$\\
\hline
$17$ & $255$ & $1^{14}01^{16}0101^{2}01^{16}0101^{5}01^{6}01^{3}01^{9}01^{2}01^{11}0101^{7}01^{6}01^{3}0^{2}1^{8}$\\[-1ex]
& $278$ & $1^{4}01^{6}01^{5}01^{3}01^{10}01^{6}01^{5}01^{8}0^{2}1^{10}0101^{6}01^{7}0101^{8}01^{12}0101^{16}01^{9}0^{2}1$\\
\hline
$18$ & $298$ & $1^{2}01^{2}01^{16}01^{6}0101^{16}0^{2}1^{4}01^{17}01^{2}01^{5}0101^{13}01^{2}01^{6}01^{2}01^{5}01^{12}0^{2}1^{10}01^{7}$\\[-1ex]
& $311$ & $1^{6}01^{8}01^{3}01^{17}0^{2}1^{3}0^{2}1^{12}01^{16}01^{7}01^{9}0^{2}10^{2}1^{4}01^{9}01^{7}01^{4}01^{14}01^{10}01^{6}$\\
\hline
$19$ & $337$ & $1^{11}01^{2}0101^{15}0^{2}1^{3}01^{2}01^{16}01^{17}0^{2}101^{8}010^{2}1^{17}01^{16}01^{2}01^{3}0^{2}101^{15}01^{2}01^{10}01^{3}$\\[-1ex]
& $346$ & $1^{11}01^{18}01^{5}0101^{6}01^{7}0101^{3}01^{2}01^{18}01^{9}0101^{2}0^{2}1^{14}01^{3}01^{2}0101^{12}01^{3}0101^{12}01^{13}01^{5}$\\
\hline
$20$ & $379$ & $1^{16}0^{2}101^{10}01^{18}01^{5}01^{6}0^{2}1^{2}01^{5}01^{15}01^{10}0101^{16}01^{3}0101^{16}01^{2}01^{9}0^{2}1^{19}0101^{2}01^{9}$\\[-1ex]
& $388$ & $1^{19}01^{11}01^{4}01^{7}0101^{4}01^{13}01^{9}01^{4}0101^{14}0^{2}1^{3}0^{2}101^{9}01^{18}01^{9}01^{4}0101^{5}01^{3}01^{10}01^{16}01^{4}$\\
\hline
$21$ & $399$ & $1^{6}0101^{6}0^{2}1^{11}01^{3}01^{4}01^{6}01^{10}01^{19}01^{7}01^{2}01^{3}01^{10}01^{6}0^{2}1^{4}01^{6}01^{2}01^{12}01^{16}0^{2}1^{19}0101^{5}01^{8}01^{7}$\\[-1ex]
& $404$ & $1^{18}0101^{5}0^{2}1^{3}01^{13}01^{9}0^{2}1^{19}0101^{8}0101^{16}01^{7}0101^{14}01^{8}0101^{20}0101^{3}01^{10}0101^{17}01$\\
\hline
$22$ & $443$ & $1^{3}01^{8}01^{7}01^{8}01^{19}01^{8}01^{7}0101^{12}01^{7}0101^{8}01^{16}0^{2}1^{16}01^{5}01^{2}0101^{20}0101^{16}01^{19}01^{3}01^{2}0101^{6}$\\[-1ex]
& $462$ & $1^{18}01^{17}01^{6}01^{4}01^{7}010^{2}1^{5}01^{15}01^{10}01^{14}0^{2}101^{6}0^{2}1^{2}01^{8}01^{14}0^{2}1^{7}01^{4}01^{10}01^{21}01^{10}01^{11}01^{4}0^{2}1^{9}$\\
\hline
$23$ & $505$ & $1^{11}01^{22}01^{20}0^{2}1^{9}01^{4}01^{8}0^{2}1^{6}01^{22}0^{2}1^{3}0^{2}1^{6}01^{10}01^{11}01^{2}01^{11}01^{10}01^{6}0^{2}1^{3}0^{2}1^{22}01^{6}01^{9}01^{4}01^{10}01^{10}$\\[-1ex]
& $506$ & $1^{22}0^{2}1^{18}01^{2}01^{19}01^{3}01^{2}01^{8}0^{2}1^{18}01^{9}01^{12}01^{2}0^{2}1^{9}01^{19}01^{2}0101^{6}01^{10}01^{4}01^{10}01^{19}010^{2}1^{6}01^{14}01^{10}$\\
\hline
$24$ & $567$ & $01^{18}01^{8}0^{2}1^{15}01^{8}01^{21}0^{2}1^{8}01^{16}010^{2}1^{6}01^{2}01^{14}01^{10}01^{20}01^{10}01^{14}01^{2}01^{6}0^{2}101^{16}01^{8}0^{2}1^{21}01^{8}0^{2}$\\[-2ex]
&& $1^{14}0^{2}1^{2}01^{3}$\\[-1ex]
& $592$ & $1^{21}01^{18}01^{16}01^{4}01^{7}01^{6}0101^{14}01^{3}0^{2}1^{8}01^{7}01^{3}01^{2}01^{20}0^{2}1^{3}01^{7}01^{15}01^{7}01^{3}0^{2}1^{20}01^{2}01^{3}01^{7}01^{9}01^{18}$\\[-2ex]
&& $0101^{6}01^{21}01^{7}01^{5}$\\
\hline
$25$ & $585$ & $1^{19}01^{9}01^{18}0^{2}1^{21}01^{15}0^{2}101^{10}01^{14}01^{13}01^{5}01^{9}010^{2}1^{8}01^{6}01^{8}0^{2}101^{9}01^{5}01^{13}01^{14}01^{10}010^{2}1^{15}01^{21}$\\[-2ex]
&& $0^{2}1^{4}01^{9}01^{2}$\\[-1ex]
& $ 606$ & $1^{24}01^{4}01^{10}01^{17}01^{8}01^{4}01^{23}0101^{12}01^{5}01^{12}01^{9}0^{2}1^{3}01^{8}01^{6}01^{21}01^{6}01^{8}01^{3}0^{2}1^{9}01^{12}01^{5}01^{12}0101^{23}$\\[-2ex]
&& $01^{4}01^{7}0^{2}1^{3}01^{10}01$\\
\hline
$26$ & $\geq 633$ & $1^{3}01^{13}01^{4}01^{12}01^{20}01^{21}01^{6}01^{18}01^{7}01^{6}01^{7}0^{2}1^{12}01^{14}0^{2}1^{4}01^{2}01^{14}01^{25}01^{2}01^{4}0^{2}1^{14}01^{12}0^{2}$\\[-2ex]
&& $1^{14}01^{4}01^{2}01^{13}01^{4}01^{9}01^{18}01$\\[-1ex]
& $\geq 642$ & $1^{15}01^{23}01^{6}01^{13}01^{4}01^{14}01^{2}01^{13}01^{11}01^{15}0^{2}1^{6}01^{12}01^{5}01^{7}01^{14}0^{2}1^{11}0^{2}101^{13}01^{4}01^{17}01^{4}01^{13}01^{6}$\\[-2ex]
&& $01^{7}01^{8}01^{4}01^{6}01^{13}01^{17}01^{5}$\\
\hline
$27$ & $\geq 663$ & $1^{18}01^{18}01^{22}01^{5}01^{13}0^{2}1^{21}01^{4}01^{9}0^{2}1^{2}01^{8}01^{2}01^{3}01^{8}0^{2}1^{12}01^{18}01^{9}01^{14}01^{5}01^{9}01^{18}01^{12}0^{2}$\\[-2ex]
&& $1^{12}01^{14}0^{2}1^{9}01^{2}01^{25}01^{4}01^{4}$\\[-1ex]
& $\geq 698$ & $1^{22}01^{6}01^{25}01^{10}01^{25}01^{11}0^{2}1^{9}0101^{5}01^{3}01^{7}0^{2}1^{9}01^{14}01^{23}01^{3}0^{2}1^{19}01^{17}0^{2}1^{3}01^{5}01^{12}01^{11}01^{7}0$\\[-2ex]
&& $1^{9}0^{2}1^{7}01^{21}01^{12}01^{13}01^{6}01$\\
\hline
$28$ & $\geq 727$ & $1^{26}01^{5}01^{13}01^{8}01^{19}01^{13}01^{10}01^{19}01^{12}01^{15}01^{7}0^{2}1^{12}0101^{12}0^{2}1^{4}01^{9}01^{2}01^{15}01^{22}01^{15}01^{2}01^{9}0$\\[-2ex]
&& $1^{5}01^{12}0101^{13}01^{14}01^{8}01^{12}01^{4}01^{10}01^{2}$\\[-1ex]
& $\geq 742$ & $1^{21}01^{22}0101^{4}01^{25}01^{10}0101^{23}01^{11}0^{2}1^{9}01^{7}01^{3}01^{7}0^{2}1^{9}01^{14}01^{27}0^{2}1^{19}01^{17}0^{2}1^{3}01^{5}01^{12}01^{11}01^{7}$\\[-2ex]
&& $01^{9}0^{2}1^{7}01^{21}01^{12}01^{13}01^{6}01$\\
\hline
$29$ & $\geq 809$ & $1^{22}01^{20}01^{28}01^{13}01^{3}01^{6}01^{4}01^{16}01^{22}01^{16}01^{11}01^{12}0^{2}1^{2}01^{7}01^{5}01^{10}01^{11}01^{3}01^{5}01^{22}0101^{3}01^{24}0$\\[-2ex]
&& $1^{11}01^{3}01^{6}01^{9}01^{11}01^{12}01^{11}01^{2}01^{24}0^{2}1^{3}01^{12}$\\[-1ex]
& $\geq 820$ & $1^{13}010^{2}1^{16}01^{15}01^{14}0101^{4}01^{2}01^{27}01^{4}0^{2}1^{14}01^{12}01^{16}01^{6}01^{7}01^{10}01^{7}01^{6}01^{16}01^{12}01^{20}01^{9}01^{4}01^{21}0$\\[-2ex]
&& $1^{15}01^{7}01^{6}01^{13}01^{6}0^{2}1^{11}01^{20}01^{21}0101^{5}01^{6}01^{4}$\\
\hline
$30$ & $\geq 843$ & $1^{29}01^{25}0^{2}1^{17}01^{9}01^{5}0^{2}1^{22}01^{7}01^{10}01^{9}01^{4}01^{9}0^{2}1^{17}01^{13}01^{8}01^{28}01^{5}01^{10}0101^{10}01^{5}01^{8}01^{19}0$\\[-2ex]
&& $1^{12}01^{24}01^{13}01^{4}01^{10}01^{2}01^{15}01^{9}010^{2}1^{15}01^{11}$\\[-1ex]
& $\geq 854$ & $1^{29}0101^{2}0101^{7}01^{4}01^{16}01^{18}01^{14}01^{21}0^{2}1^{7}01^{3}01^{2}01^{13}01^{14}01^{4}01^{16}01^{3}01^{11}01^{24}01^{4}01^{10}01^{7}01^{25}01^{10}$\\[-2ex]
&& $0^{2}1^{4}01^{18}0^{2}1^{15}01^{11}01^{5}0101^{4}01^{3}01^{21}01^{18}01^{9}01^{13}$\\
\hline
$31$ & $\geq 915$ & $1^{12}01^{3}01^{4}01^{7}01^{11}01^{8}01^{16}01^{10}01^{4}01^{3}01^{4}01^{19}01^{29}01^{6}01^{26}01^{3}0101^{24}01^{8}01^{16}010^{2}1^{25}01^{3}01^{8}0$\\[-2ex]
&& $1^{2}01^{16}01^{2}0101^{3}01^{24}01^{7}01^{3}01^{24}01^{2}0101^{28}01^{20}01^{11}0101^{16}01^{5}$\\[-1ex]
& $\geq 930$ & $1^{12}01^{19}01^{16}01^{10}01^{9}01^{3}01^{4}01^{14}01^{12}01^{14}01^{18}0^{2}1^{10}01^{4}01^{13}01^{8}01^{5}0101^{10}01^{17}01^{15}01^{16}01^{8}01^{5}01^{8}$\\[-2ex]
&& $0101^{4}01^{8}01^{19}01^{16}01^{5}01^{18}01^{4}01^{25}01^{14}01^{16}01^{3}01^{9}01^{6}01^{13}01^{11}01$\\
\hline
$32$ & $\geq 957$ & $1^{24}01^{3}01^{7}01^{19}01^{6}01^{4}01^{17}01^{30}0^{2}1^{6}01^{5}01^{24}0^{2}1^{11}01^{18}01^{4}01^{17}01^{19}01^{10}01^{11}01^{7}01^{17}0^{2}1^{3}0$\\[-2ex]
&& $1^{23}01^{12}01^{11}01^{7}0^{2}1^{3}01^{25}01^{4}01^{5}01^{7}01^{22}01^{7}01^{12}01^{30}01^{2}01^{6}01$\\[-1ex]
& $\geq 962$ & $1^{11}0^{2}1^{18}01^{21}01^{5}01^{2}01^{17}01^{19}0101^{16}01^{13}01^{3}01^{24}01^{5}01^{2}01^{23}01^{12}01^{23}0^{2}1^{5}01^{24}0101^{5}0^{2}101^{9}01^{17}$\\[-2ex]
&& $01^{15}01^{6}01^{19}01^{9}01^{24}01^{5}01^{2}01^{5}0^{2}1^{23}01^{27}01^{8}01^{19}01^{2}$\\
\hline
$33$ & $\geq 995$ & $1^{31}0101^{20}01^{7}0^{2}1^{12}01^{4}01^{30}01^{7}01^{9}01^{8}01^{11}01^{3}01^{13}01^{18}01^{12}01^{11}01^{3}0101^{10}0^{2}1^{4}01^{7}01^{21}01^{14}01^{28}$\\[-2ex]
&& $0101^{11}01^{3}0^{2}1^{17}01^{30}01^{6}0101^{28}01^{32}01^{2}01^{8}01^{5}01^{3}01^{8}01^{3}01^{13}$\\[-1ex]
& $\geq 1004$ & $1^{2}01^{6}01^{22}01^{4}01^{32}01^{2}01^{9}0^{2}101^{7}01^{5}01^{23}01^{18}01^{30}010^{2}1^{16}01^{13}01^{9}01^{6}01^{9}01^{19}01^{11}01^{4}01^{17}01^{10}0$\\[-2ex]
&& $1^{26}0^{2}1^{18}01^{10}01^{23}01^{12}01^{15}01^{9}01^{12}01^{6}0^{2}1^{8}0^{2}1^{3}01^{31}01^{12}$\\
\hline
$34$ & $\geq 1053$ & $1^{9}01^{14}01^{21}0101^{4}01^{27}01^{8}01^{24}01^{32}01^{3}01^{4}01^{16}01^{13}0101^{2}01^{33}01^{3}01^{4}01^{3}01^{16}0101^{13}01^{3}01^{16}01^{6}$\\[-2ex]
&& $01^{33}0^{2}1^{9}01^{4}01^{31}01^{2}0101^{16}01^{4}01^{31}01^{3}01^{2}01^{7}0^{2}1^{22}0101^{8}01^{30}01^{3}$\\[-1ex]
& $\geq 1080$ & $1^{23}0^{2}1^{9}01^{22}01^{13}01^{25}0^{2}1^{21}01^{8}01^{15}01^{11}01^{18}0^{2}1^{4}01^{20}0101^{2}01^{5}01^{21}01^{8}01^{4}01^{23}01^{4}01^{8}01^{21}$\\[-2ex]
&& $01^{10}01^{20}01^{4}0^{2}1^{2}01^{27}01^{10}01^{4}01^{2}01^{26}0101^{25}01^{8}01^{9}01^{17}01^{4}0^{2}1^{12}01^{21}0^{2}1^{3}01^{3}$\\
\hline
$35$ & $\geq 1113$ & $1^{7}01^{12}01^{21}01^{28}01^{3}01^{16}0101^{23}01^{4}0101^{10}01^{8}01^{14}01^{28}01^{9}01^{12}01^{21}01^{23}0^{2}101^{5}01^{3}01^{21}01^{2}01^{23}0$\\[-2ex]
&& $1^{5}01^{14}01^{15}01^{3}010^{2}101^{5}01^{17}01^{10}01^{23}0101^{3}01^{2}01^{9}01^{21}01^{14}01^{30}01^{25}01^{10}01^{7}$\\[-1ex]
& $\geq 1154$ & $1^{28}01^{8}01^{18}01^{6}01^{9}01^{29}01^{6}01^{9}01^{6}01^{12}01^{5}0101^{10}01^{25}01^{28}01^{13}01^{8}01^{9}01^{3}01^{25}01^{20}01^{3}01^{21}01^{5}0$\\[-2ex]
&& $1^{10}01^{2}01^{34}0101^{3}01^{2}01^{15}01^{17}01^{2}01^{27}01^{19}01^{8}0^{2}1^{6}0101^{2}0^{2}1^{19}01^{30}01^{21}01^{7}$\\
\hline
$36$ & $\geq 1185$ & $1^{6}01^{28}01^{10}01^{16}01^{28}01^{35}01^{4}01^{7}01^{23}0101^{2}01^{8}01^{2}01^{3}01^{5}01^{15}01^{2}0101^{19}01^{30}01^{25}01^{12}01^{3}01^{14}0$\\[-2ex]
&& $1^{8}01^{11}01^{7}01^{4}01^{18}0101^{7}01^{3}01^{2}01^{4}01^{6}01^{21}01^{3}01^{28}01^{24}01^{23}01^{3}0^{2}101^{32}01^{5}01^{11}01^{10}01^{15}$\\[-1ex]
& $\geq 1212$ & $1^{3}01^{25}01^{13}01^{24}01^{5}01^{4}01^{17}01^{13}01^{10}01^{8}01^{28}010^{2}1^{32}01^{10}01^{13}01^{20}0101^{4}01^{12}0101^{13}01^{20}0101^{8}01^{4}01^{32}$\\[-2ex]
&& $0101^{28}01^{4}01^{19}01^{32}01^{6}01^{18}01^{2}01^{6}0^{2}1^{3}01^{24}01^{6}01^{20}01^{13}01^{14}01^{7}01^{23}01^{4}01^{8}$\\
\hline
$37$ & $\geq 1271$ & $1^{30}01^{6}01^{15}01^{13}01^{23}010^{2}1^{15}0^{2}1^{28}01^{7}01^{28}01^{2}01^{3}01^{13}01^{22}0^{2}1^{12}01^{3}010^{2}1^{36}01^{28}01^{16}0^{2}101^{16}$\\[-2ex]
&& $01^{23}01^{13}01^{8}01^{3}01^{9}01^{12}01^{7}01^{29}01^{16}01^{5}0101^{15}010^{2}1^{33}01^{30}0^{2}1^{6}01^{12}01^{17}01^{5}01^{9}01^{2}01^{4}01^{6}$\\[-1ex]
& $\geq 1294$ & $1^{7}01^{24}01^{31}01^{6}01^{3}01^{30}0101^{3}01^{4}01^{22}01^{34}0101^{8}01^{32}01^{8}01^{22}01^{2}01^{7}01^{2}01^{3}01^{23}01^{6}01^{22}01^{10}0101^{4}$\\[-2ex]
&& $01^{9}01^{19}01^{4}01^{3}01^{7}01^{15}01^{29}01^{6}01^{3}0101^{2}01^{10}01^{22}01^{28}01^{12}01^{8}01^{16}01^{8}0101^{8}01^{22}01^{23}01^{4}01^{12}01^{9}$\\
\hline
$38$ & $\geq 1335$ & $1^{10}01^{35}01^{12}01^{8}01^{12}01^{7}0101^{33}01^{14}0101^{5}01^{13}01^{2}01^{7}01^{22}01^{34}0^{2}1^{13}01^{7}01^{26}0101^{4}01^{37}01^{13}0101^{4}$\\[-2ex]
&& $01^{36}0^{2}101^{27}01^{13}0101^{7}01^{18}01^{9}0^{2}1^{6}01^{29}01^{18}0^{2}1^{6}01^{28}0101^{21}0^{2}1^{19}01^{9}01^{36}01^{6}01^{6}$\\[-1ex]
& $\geq 1368$ & $1^{11}01^{7}01^{36}01^{32}0^{2}1^{3}0^{2}1^{8}01^{20}0101^{11}01^{33}01^{4}01^{16}01^{8}01^{28}01^{7}01^{4}01^{26}01^{10}01^{22}01^{4}01^{8}01^{4}01^{11}01^{5}$\\[-2ex]
&& $01^{4}01^{11}01^{13}01^{24}01^{17}01^{13}01^{4}01^{22}01^{7}01^{17}01^{8}01^{9}0^{2}1^{4}01^{27}01^{20}01^{18}01^{22}01^{3}01^{9}0^{2}1^{25}01^{10}01^{27}01$\\
\hline
$39$ & $\geq 1405$ & $101^{16}01^{5}01^{6}01^{2}01^{6}01^{33}01^{4}01^{10}0^{2}1^{28}01^{6}01^{4}01^{29}01^{18}01^{6}01^{10}0101^{36}0^{2}1^{22}01^{33}0101^{11}01^{2}0101^{11}01^{22}$\\[-2ex]
&& $01^{33}01^{4}01^{11}01^{21}0^{2}101^{13}01^{3}01^{32}01^{19}01^{13}01^{8}01^{11}01^{24}0^{2}1^{15}01^{36}01^{13}01^{6}01^{26}01^{2}01^{3}010^{2}1^{15}01^{11}01^{4}$\\[-1ex]
& $\geq 1410$ & $1^{3}01^{34}01^{24}01^{3}01^{32}01^{4}01^{18}0^{2}1^{35}01^{12}01^{7}01^{16}01^{11}01^{6}01^{9}01^{20}01^{6}01^{9}01^{19}01^{5}01^{10}0101^{23}01^{16}0^{2}1^{31}$\\[-2ex]
&& $01^{3}01^{12}01^{7}01^{9}01^{6}01^{28}01^{6}01^{17}010^{2}1^{9}01^{18}01^{6}0^{2}1^{8}01^{2}01^{17}01^{3}01^{32}01^{5}01^{12}01^{34}0101^{24}01^{15}01^{13}01^{3}01^{7}$\\
\hline
\hline
\end{longtable}
}

\section{Using the \OKlibrary}
\label{sec:OKl}

The \OKlibrary, available at \OKinternet, is an open-source research and development platform for SAT-solving and related areas (attacking hard problems); see \cite{Kullmann2009OKlibrary} for some general information. For the purpose of reproduction of all results, one can use the Git ID ``4cea9abf851424ca56f2ad0e4b8be2d707b041c2'' (package 00147).\footnote{Via the Git ID one can identify the versions of programs used in the article. The package provides the sources and a build system. Since building depends on the environment (the operating system to start with), there can not be a guarantee for the build to succeed, but perhaps later (or earlier) packages need to be used.} For the purposes of this article the following components are directly relevant:
\begin{itemize}
\item The \OKlibrary{} provides an already rather extensive library of functions for the computer algebra system \texttt{Maxima}\footnote{\url{http://maxima.sourceforge.net/}}. For example all hypergraph generators discussed in this article, and all vdW- and palindromic vdW-numbers can be computed and investigated at this level.
\item For computations which take more time, C++ implementations are available.
\item The \OKlibrary{} provides easy access to (original) SAT solvers and related tools (as ``external sources'').\footnote{The aim is to serve as a comprehensive collection, also maintaining ``historical'' versions.}
\item Finally these components are integrated into tools for running and evaluating experiments.\footnote{In general we use the \texttt{R} system for statistical evaluation.}
\end{itemize}
In the following sections we demonstrate the use of these tools. Some general technical remarks:
\begin{enumerate}
\item The installed \OKlibrary{} lives inside directory \texttt{OKplatform}.
\item Inside this directory the \texttt{Maxima}-installation is called via \texttt{oklib --maxima} on the (Linux) command-line.
\item The C++ programs as well as the external sources, here the various SAT solvers, are placed on the path of the (Linux) user, and are thus callable by their name on the command-line (anywhere).
\end{enumerate}

\subsection{Numbers and certificates}
\label{sec:OKlNumbers}

All known vdW-numbers and palindromic vdW-numbers and known bounds are available at the computer-algebra level in the \OKlibrary{} (using \texttt{Maxima}). For example the (known) numbers $\vdw(2;3,t)$ and $\vdwpd(2;3,t)$ are printed as follows (where inside the \OKlibrary{} we typically use the letter ``k'' for the length of an arithmetical progression, not ``t'' as in this article):
\begin{verbatim}
OKplatform> oklib --maxima
(%i1) oklib_load_all();
(%i2) output(N) := block([L],
  print(" k  vdw          pdvdw                       span         gap"),
  for k : 3 thru N do (L : [3,k],
    printf(true, "~2,d  ~12,a ~27,a ~12,a ~12,a~%",
      k, vanderwaerden(L), pdvanderwaerden(L), pd_span(L), pd_gap(L))))$
(%i3) output(40);
 k  vdw          pdvdw                       span         gap
 3  9            [6,9]                       3            0
 4  18           [15,16]                     1            2
 5  22           [16,21]                     5            1
 6  32           [30,31]                     1            1
 7  46           [41,44]                     3            2
 8  58           [52,57]                     5            1
 9  77           [62,77]                     15           0
10  97           [93,94]                     1            3
11  114          [110,113]                   3            1
12  135          [126,135]                   9            0
13  160          [142,155]                   13           5
14  186          [174,183]                   9            3
15  218          [200,205]                   5            13
16  238          [232,237]                   5            1
17  279          [256,279]                   23           0
18  312          [299,312]                   13           0
19  349          [338,347]                   9            2
20  [389,inf-1]  [380,389]                   9            [0,inf-390]
21  [416,inf-1]  [400,405]                   5            [11,inf-406]
22  [464,inf-1]  [444,463]                   19           [1,inf-464]
23  [516,inf-1]  [506,507]                   1            [9,inf-508]
24  [593,inf-1]  [568,593]                   25           [0,inf-594]
25  [656,inf-1]  [586,607]                   21           [49,inf-608]
26  [727,inf-1]  [634,643]                   9            [84,inf-644]
27  [770,inf-1]  [664,699]                   35           [71,inf-700]
28  [827,inf-1]  [[728,inf-1],[743,inf-1]]   [15,0]       [84,0]
29  [868,inf-1]  [[810,inf-1],[821,inf-1]]   [11,0]       [47,0]
30  [903,inf-1]  [[844,inf-1],[855,inf-1]]   [11,0]       [48,0]
31  [931,inf-1]  [[916,inf-1],[931,inf-1]]   [15,0]       [0,0]
32  [1007,inf-1] [[958,inf-1],[963,inf-1]]   [5,0]        [44,0]
33  [1064,inf-1] [[996,inf-1],[1005,inf-1]]  [9,0]        [59,0]
34  [1144,inf-1] [[1054,inf-1],[1081,inf-1]] [27,0]       [63,0]
35  [1205,inf-1] [[1114,inf-1],[1155,inf-1]] [41,0]       [50,0]
36  [1258,inf-1] [[1186,inf-1],[1213,inf-1]] [27,0]       [45,0]
37  [1339,inf-1] [[1272,inf-1],[1295,inf-1]] [23,0]       [44,0]
38  [1379,inf-1] [[1336,inf-1],[1369,inf-1]] [33,0]       [10,0]
39  [1419,inf-1] [[1406,inf-1],[1411,inf-1]] [5,0]        [8,0]
40  unknown      unknown                     unknown      unknown
\end{verbatim}
As one can see, if only bounds are known instead of a precise number $n$ resp. number-pair $(p,q)$, then the numbers $x \in \set{n,p,q}$ are replaced by pairs $(a,b)$ with $a \leq x \leq b$. Here $b = \mathtt{inf} - 1$ indicates that the number is finite, but no more precise upper bounds are known.\footnote{In principle there exist theoretical upper bounds, but for practical purposes these bounds are completely useless.} So for example we only know currently that $\vdw(2;3,20) \geq 389$, and this is shown by the interval $[389, \mathtt{inf}-1]$. Span and gap are simply computed according to definition, where in \texttt{maxima} the symbol \texttt{inf} is treated here like an unknown. That implies $\mathtt{inf} - \mathtt{inf} = 0$, and thus for palindromic span and gap the ``0'' in the second position indicate that the numbers in the first positions could go up or down. For example $\vdw(2;3,20) \geq 389$ and $\vdwpd(2;3,20) = (380,389)$, whence nothing can be said about $\pdg(2;3,20) = \vdw(2;3,20) - \vdwpd(2;3,20)_2$ except the trivialities that it is at least $0$ and less than infinity; the latter becomes $(\mathtt{inf} - 1) - 389 =  \mathtt{inf} - 390$.

Also the certificates (good partitions) are available, in various representations. First the certificate for $\vdw(2;3,20) \ge 389$ (see Subsection \ref{sec:Certificates3tp}), for which we check that it is in fact a palindromic certificate:
\begin{verbatim}
(%i4) full_certificate_string_vdw_3k(20);
(%o4) ["1^{19}01^{11}01^{4}01^{7}0101^{4}01^{13}01^{9}01^{4}0101^{14}
  0^{2}1^{3}0^{2}101^{9}01^{18}01^{9}01^{4}0101^{5}01^{3}01^{10}01^{16}
  01^{8}01^{16}01^{10}01^{3}01^{5}0101^{4}01^{9}01^{18}01^{9}010^{2}1^{3}
  0^{2}1^{14}0101^{4}01^{9}01^{13}01^{4}0101^{7}01^{4}01^{11}01^{19}"]
(%i5) certificate_pdvdw_p([3,20],388,full_certificate_vdw_3k(20)[1]);
(%o5) true
\end{verbatim}
And here certificates for palindromic number-pairs:
\begin{verbatim}
(%i6) cfull_certificate_string_pdvdw_3k(34);
(%o6) [["1^{9}01^{14}01^{21}0101^{4}01^{27}01^{8}01^{24}01^{32}01^{3}0
  1^{4}01^{16}01^{13}0101^{2}01^{33}01^{3}01^{4}01^{3}01^{16}0101^{13}
  01^{3}01^{16}01^{6}01^{33}0^{2}1^{9}01^{4}01^{31}01^{2}0101^{16}01^{4}
  01^{31}01^{3}01^{2}01^{7}0^{2}1^{22}0101^{8}01^{30}01^{3}"],
       ["1^{23}0^{2}1^{9}01^{22}01^{13}01^{25}0^{2}1^{21}01^{8}01^{15}0
  1^{11}01^{18}0^{2}1^{4}01^{20}0101^{2}01^{5}01^{21}01^{8}01^{4}01^{23}
  01^{4}01^{8}01^{21}01^{10}01^{20}01^{4}0^{2}1^{2}01^{27}01^{10}01^{4}0
  1^{2}01^{26}0101^{25}01^{8}01^{9}01^{17}01^{4}0^{2}1^{12}01^{21}0^{2}
  1^{3}01^{3}"]]
(%i7) extract_data_certificates_pdvdw_3k(34);
(%o7) [[3,34],1054,1081,
        [[10,25,47,49,54,82,91,116,149,153,158,175,189,191,194,228,232,
          237,241,258,260,274,278,295,302,336,337,347,352,384,387,389,
          406,411,443,447,450,458,459,482,484,493,524]],
        [[24,25,35,58,72,98,99,121,130,146,158,177,178,183,204,206,209,
          215,237,246,251,275,280,289,311,322,343,348,349,352,380,391,
          396,399,426,428,454,463,473,491,496,497,510,532,533,537]]]
\end{verbatim}
With the first command we get the representation of the good partitions as used in this paper (where now for the palindromic situation we have two good partition according to Theorem \ref{thm:certificatespd}), while the second command yields a list with five elements: first the parameter tuple, then the two components of the palindromic number-pair, and then two lists with the good partitions available, now represented via the block in the partition for the second colour.

Analysing the patterns according to Section \ref{sec-pat}, and applying these measurements to the certificates stored in the \OKlibrary{} for $20 \leq t \leq 39$ is done as follows:
\begin{verbatim}
(%i8) for k : 20 thru 39 do
  print(k,firste(vanderwaerden3k(k)),
        map(analyse_certificate,full_certificate_vdw_3k(k)));
20 389 [[[48,340],[44,45],[4,37],[5,27],[20,1]]]
21 416 [[[50,365],[43,44],[7,34],[13,26],[8,1]]]
22 464 [[[54,409],[51,52],[3,47],[5,40],[27,1]]]
23 516 [[[59,456],[53,54],[6,45],[12,36],[17,1]]]
24 593 [[[63,529],[57,58],[6,54],[13,37],[20,1]]]
25 656 [[[74,581],[69,70],[5,64],[11,45],[16,2]]]
26 727 [[[78,648],[72,72],[6,64],[13,42],[21,1]]]
27 770 [[[79,690],[72,73],[7,65],[15,58],[11,2]]]
28 827 [[[79,747],[74,75],[5,64],[11,44],[19,1]]]
29 868 [[[81,786],[76,77],[5,69],[11,57],[27,1]]]
30 903 [[[83,819],[76,77],[7,67],[13,57],[15,1]]]
31 931 [[[82,848],[80,81],[2,77],[5,53],[58,1]]]
32 1007 [[[87,919],[82,83],[5,78],[9,62],[29,1]]]
33 1064 [[[89,974],[85,86],[4,80],[9,58],[25,1]]]
34 1144 [[[96,1047],[87,88],[9,80],[19,63],[23,2]]]
35 1205 [[[95,1109],[91,92],[4,84],[9,67],[41,1]]]
36 1258 [[[101,1156],[97,98],[4,88],[9,72],[42,1]]]
37 1339 [[[105,1233],[97,98],[8,90],[17,65],[30,2]]]
38 1379 [[[104,1274],[96,97],[8,91],[17,73],[26,1]]]
39 1419 [[[105,1313],[98,99],[7,95],[13,72],[46,1]]]
\end{verbatim}
Per line we print out three items: $t$, the lower bound on $\vdw(2;3,t)$ and the list of data for each stored certificate. Now currently we have only stored one certificate for each $20 \leq t \leq 39$, and thus the third item contains just one list, with five pairs for the different statistics.\footnote{We found more than one solution in each case, but always very similar to the one stored; there seems to be a clustering of solutions, and perhaps there is always only one (or very few) cluster.} These five pairs have the following meaning:
\begin{enumerate}
\item First come $n_0$ and $n_1$.
\item Then come the numbers of terms $0^s$ and $1^s$ (we don't use ``$00$'' here, and so these terms alternate, and thus their numbers differ at most by one).
\item Then from these counts the cases with $s=1$ are excluded; thus the first element of the pair is $n_{00}$.
\item Now these exponents $s$ are put in the list, and the sums of the numbers of peaks and valleys are computed; again for block $0$ and block $1$ of the partition, and thus now the second element of the pair is $\np + \nv$.
\item Finally for these lists of exponents the maximal size of an interval with constant values is computed; thus if there were a second element of the pair with value $3$ or greater, then Question \ref{ques-5} would have been answered in the positive.
\end{enumerate}

The value of $d$ from Subsection \ref{sec:conjupb} is computed as follows:
\begin{verbatim}
(%i9) lmax(Delta_l(map(firste,create_list(vanderwaerden([3,k]),k,3,39)))
           /create_list(k,k,3,38));
(%o9) 77/23
(%i10) round_fdd(77/23/2,3);
(%o10) 1.674
\end{verbatim}

\subsection{Hypergraphs}
\label{sec:OKlHypergraphs}

The hypergraphs are available at Maxima-level, and the computationally expensive palindromic hypergraph also at C++ level:
\begin{verbatim}
(%i11) arithprog_hg(3,5);
(%o11) [{1,2,3,4,5},{{1,2,3},{1,3,5},{2,3,4},{3,4,5}}]
(%i12) arithprog_pd_hg(3,5);
(%o12) [{1,2,3},{{1,3},{2,3}}]

> PdVanderWaerden-O3-DNDEBUG 3 5
c Palindromised hypergraph with arithmetic-progression length 3
 and 5 vertices.
p cnf 3 2
1 3 0
2 3 0
\end{verbatim}

\subsection{SAT instances}
\label{sec:OKlSATinstances}

The SAT-instance for considering $\vdw(2;3,t)$ with $n$ vertices is created by the program call
\begin{center}
  \texttt{VanderWaerdenCNF-O3-DNDEBUG 3 t n},
\end{center}
for example for $t = 4$ and $n=6$
\begin{verbatim}
> VanderWaerdenCNF-O3-DNDEBUG 3 4 6
> cat VanDerWaerden_2-3-4_6.cnf
c Van der Waerden numbers with partitioning into 2 parts;
 SAT generator written by Oliver Kullmann, Swansea, May 2004, October 2010.
c Arithmetical progression sizes k1 = 3, k2 = 4.
c Number of elements n = 6.
c Iterating through the arithmetic progressions in colexicographical order.
p cnf 6 9
1 2 3 0
2 3 4 0
1 3 5 0
3 4 5 0
2 4 6 0
4 5 6 0
-1 -2 -3 -4 0
-2 -3 -4 -5 0
-3 -4 -5 -6 0
\end{verbatim}

The SAT-instance for considering $\vdwpd(2;3,t)$ with $n$ vertices is created by the program call
\begin{center}
  \texttt{PdVanderWaerdenCNF-O3-DNDEBUG 3 t n},
\end{center}
for example for $t = 4$ and $n=9$
\begin{verbatim}
> PdVanderWaerdenCNF-O3-DNDEBUG 3 4 9
> cat VanDerWaerden_pd_2-3-4_9.cnf
c Palindromic van der Waerden problem: 2 parts, arithmetic progressions of
 size 3 and 4, and 9 elements, yielding 5 variables.
p cnf 5 10
1 2 3 0
2 4 0
1 3 4 0
1 5 0
2 5 0
3 5 0
4 5 0
-2 -4 0
-1 -3 -5 0
-3 -4 -5 0
\end{verbatim}

\subsection{The SAT solvers}
\label{sec:OKlsolvers}

All solvers are installed via the \OKlibrary; the \oksolver{}\footnote{\url{https://github.com/OKullmann/oklibrary/tree/master/Satisfiability/Solvers/OKsolver/SAT2002}} is a solver specific to the \OKlibrary, \ntawsolver{}\footnote{\url{https://github.com/OKullmann/oklibrary/blob/master/Satisfiability/Solvers/TawSolver/tawSolver.cpp}} was developed in it (starting with version 1.0), and \satz{}\footnote{\url{https://github.com/OKullmann/oklibrary/blob/master/Satisfiability/Solvers/Satz/satz215.2.c}} as well as \march{} are maintained in the \OKlibrary{}. Example output for the column $t=12$ in Table \ref{tab:complsolvervdw}, with the instance produced by
\begin{verbatim}
> VanderWaerdenCNF-O3-DNDEBUG 3 12 135
\end{verbatim}
resp.
\begin{verbatim}
> VanderWaerdenCNF-O3-DNDEBUG 3 12 134
\end{verbatim}
for the satisfiable case, is provided in the following.

\subsubsection{\tawsolver}
\label{sec:okltaw}

First \ntawsolver{} (output with one additional line-break for the url):
\begin{verbatim}
> tawSolver -v
tawSolver:
 authors: Tanbir Ahmed and Oliver Kullmann
 url's:
  http://sourceforge.net/projects/tawsolver/
  https://github.com/OKullmann/oklibrary/blob/master/
    Satisfiability/Solvers/TawSolver/tawSolver.cpp
 Version: 2.6.6
 Last change date: 17.8.2013
 Mapping k -> weight, for clause-lengths k specified at compile-time:
   2->4.85  3->1  4->0.354  5->0.11  6->0.0694
 Divisor for open weights: 1.46
 Option summary = ""
 Macro settings:
  LIT_TYPE = std::int32_t (with 31 binary digits)
  UCP_STRATEGY = 1
 Compiled without TAU_ITERATION
 Compiled without ALL_SOLUTIONS
 Compiled without PURE_LITERALS
 Compiled with NDEBUG
 Compiled with optimisation options
 Compilation date: Aug 17 2013 21:38:43
 Compiler: g++, version 4.7.3
 Provided in the OKlibrary http://www.ok-sat-library.org
 Git ID = 237cbfc4d9b772a29e125928959af14cb4495d3e

> tawSolver VanDerWaerden_2-3-12_135.cnf
s UNSATISFIABLE
c number_of_variables                   135
c number_of_clauses                     5251
c maximal_clause_length                 12
c number_of_literal_occurrences         22611
c running_time(sec)                     10.58
c number_of_nodes                       961949
c number_of_binary_nodes                480974
c number_of_1-reductions                11312180
c reading-and-set-up_time(sec)          0.004
c file_name                             VanDerWaerden_2-3-12_135.cnf
c options                               ""
\end{verbatim}
A ``binary node'' is one with two children, i.e., where the second branch was not explored since the first branch was found satisfiable. And a ``1-reduction'' is one assignment of a literal $x$ to true due to a unit-clause $\set{x}$. Calling \ttawsolver{} happens via \texttt{ttawSolver}, and the counting version is called \texttt{ctawSolver}.
\begin{verbatim}
> ttawSolver VanDerWaerden_2-3-12_135.cnf
s UNSATISFIABLE
c number_of_variables                   135
c number_of_clauses                     5251
c maximal_clause_length                 12
c number_of_literal_occurrences         22611
c running_time(sec)                     19.29
c number_of_nodes                       953179
c number_of_binary_nodes                476589
c number_of_1-reductions                11285634
c number_of_pure_literals               1317
c reading-and-set-up_time(sec)          0.005
c file_name                             VanDerWaerden_2-3-12_135.cnf
c options                               "PT5"

> ctawSolver VanDerWaerden_2-3-12_135.cnf
s UNSATISFIABLE
c number_of_variables                   135
c number_of_clauses                     5251
c maximal_clause_length                 12
c number_of_literal_occurrences         22611
c running_time(sec)                     10.64
c number_of_nodes                       961949
c number_of_binary_nodes                480974
c number_of_1-reductions                11312180
c number_of_solutions                   0
c reading-and-set-up_time(sec)          0.005
c file_name                             VanDerWaerden_2-3-12_135.cnf
c options                               "A19"
\end{verbatim}
Options are reported via acronyms: ``P'' for pure literals, ``T'' for the tau-heuristics, followed by the number of iterations of the Newton-Raphson method, and ``A'' for all solutions, followed by the number of decimal digits for counting. Instead of just counting, we can also output all solutions, for example to standard output:
\begin{verbatim}
> ctawSolver VanDerWaerden_2-3-12_134.cnf -cout
v 1 2 3 4 5 6 7 8 9 -10 11 12 13 14 15 16 17 18 -19 20 21 22 23 24 25 26 27 28
-29 30 31 -32 33 34 35 -36 37 -38 39 40 41 42 43 44 45 -46 47 48 -49 50 -51 52
53 54 -55 56 57 58 59 60 61 62 63 64 65 66 -67 -68 69 70 71 72 73 74 75 76 77
78 79 -80 81 82 83 -84 85 -86 87 88 -89 90 91 92 93 94 95 96 -97 98 -99 100
101 102 -103 104 105 -106 107 108 109 110 111 112 113 114 115 -116 117 118 119
120 121 122 123 124 -125 126 127 128 129 130 131 132 133 134 0
s SATISFIABLE
c number_of_variables                   134
c number_of_clauses                     5172
c maximal_clause_length                 12
c number_of_literal_occurrences         22266
c running_time(sec)                     10.56
c number_of_nodes                       968509
c number_of_binary_nodes                484254
c number_of_1-reductions                11308431
c number_of_solutions                   1
c reading-and-set-up_time(sec)          0.004
c file_name                             VanDerWaerden_2-3-12_134.cnf
c options                               "A19"
\end{verbatim}
The solution is given in the DIMACS format for partial assignments, with positive literals setting the underlying variable to \texttt{true}, and negative literals setting them to \texttt{false} (so positive literals are the elements of the partition for $t=12$ here). For all options, use \texttt{tawSolver} without arguments, or see the source code. Finally we note that \ttawsolver{} and \texttt{ctawSolver} are just compilations of the \tawsolver{} with specific options set\footnote{see \url{https://github.com/OKullmann/oklibrary/blob/master/Satisfiability/Solvers/TawSolver/makefile} for the \texttt{makefile} in the \OKlibrary}, namely:
\begin{verbatim}
ttawSolver:
 -DPURE_LITERALS -DTAU_ITERATION=5

ctawSolver:
 -DALL_SOLUTIONS

cttawSolver:
 -DTAU_ITERATION=5 -DALL_SOLUTIONS
\end{verbatim}

\subsubsection{\satz}
\label{sec:oklsatz}

Now \satz:
\begin{verbatim}
> satz215 VanDerWaerden_2-3-12_135.cnf
s UNSATISFIABLE
c sat_status                            0
c number_of_variables                   135
c initial_number_of_clauses             5251
c reddiff_number_of_clauses             0
c running_time(sec)                     76.73
c number_of_nodes                       262304
c number_of_binary_nodes                133373
c number_of_pure_literals               55
c number_of_1-reductions                5482044
c number_of_2-look-ahead                30069498
c number_of_2-reductions                1196400
c number_of_3-look-ahead                563872
c number_of_3-reductions                257097
c file_name                             VanDerWaerden_2-3-12_135.cnf
\end{verbatim}
Here ``reddiff'' is the ``difference due to reduction'' in the number of clauses: clauses can be removed by subsumption (not applicable here), while clauses can be added by resolution (does not happen here).

\subsubsection{\march}
\label{sec:oklmarch}

Now \march:
\begin{verbatim}
> march_pl VanDerWaerden_2-3-12_135.cnf
c main():: nodeCount: 47963
c main():: dead ends in main: 110
c main():: lookAheadCount: 10456897
c main():: unitResolveCount: 274045
c main():: time=184.539993
c main():: necessary_assignments: 5287
c main():: bin_sat: 0, bin_unsat 0
c main():: doublelook: #: 421439, succes #: 321732
c main():: doublelook: overall 4.150 of all possible doublelooks executed
c main():: doublelook: succesrate: 76.341, average DL_trigger: 273.489
s UNSATISFIABLE
\end{verbatim}

\subsubsection{\oksolver}
\label{sec:okloks}

And to conclude the complete solvers, the \oksolver:
\begin{verbatim}
> OKsolver_2002-O3-DNDEBUG VanDerWaerden_2-3-12_135.cnf
s UNSATISFIABLE
c sat_status                            0
c initial_maximal_clause_length         12
c initial_number_of_variables           135
c initial_number_of_clauses             5251
c initial_number_of_literal_occurrences 22611
c number_of_initial_unit-eliminations   0
c reddiff_maximal_clause_length         0
c reddiff_number_of_variables           0
c reddiff_number_of_clauses             0
c reddiff_number_of_literal_occurrences 0
c number_of_2-clauses_after_reduction   0
c running_time(sec)                     215.8
c number_of_nodes                       281381
c number_of_single_nodes                0
c number_of_quasi_single_nodes          0
c number_of_2-reductions                2049274
c number_of_pure_literals               29
c number_of_autarkies                   0
c number_of_missed_single_nodes         0
c max_tree_depth                        36
c proportion_searched                   1.000000e+00
c proportion_single                     0.000000e+00
c total_proportion                      1
c number_of_table_enlargements          0
c number_of_1-autarkies                 490
c number_of_new_2-clauses               0
c maximal_number_of_added_2-clauses     0
c file_name                             VanDerWaerden_2-3-12_135.cnf
\end{verbatim}

\subsubsection{\texttt{Ubcsat}}
\label{sec:OKlubc}

If we want to run an algorithm from the \texttt{Ubcsat}-suite\footnote{see \url{http://ubcsat.dtompkins.com/}} on its own (while running it in the iterative fashion, as discussed in Subsection \ref{sec:remsatincomp}, is shown in the following \ref{sec:OKlexperiments}), for example \texttt{gsat-tabu}, then this can be done as follows (using an additional line-break in the command-line, and four additional line-breaks in the first output-line), for a cut-off $10^6$, ten runs, and initial seed $0$ (for reproducibility):
\begin{verbatim}
> ubcsat-okl -alg gsat-tabu -cutoff 1000000 -runs 10 -seed 0
 -i VanDerWaerden_2-3-12_134.cnf
# -rclean -r out stdout run,found,best,beststep,steps,seed -r stats stdout
 numclauses,numvars,numlits,fps,beststep[mean],steps[mean+max],percentsolve,
 best[min+max+mean+median] -runs 10 -cutoff 100000 -rflush
 -alg gsat-tabu -cutoff 1000000 -runs 10 -seed 0
 -i VanDerWaerden_2-3-12_134.cnf
       sat  min               osteps               msteps       seed
      1 0     1                 3588              1000000          0
      2 1     0               543154               543154 1492175541
      3 0     1                 5687              1000000  367425000
      4 0     1                 3152              1000000 3611176606
      5 0     1               164885              1000000  388711246
      6 0     1                50599              1000000 4160687068
      7 0     1                 3533              1000000  533276301
      8 0     1                94759              1000000 1146607069
      9 0     1                 2921              1000000 3903233437
     10 0     1                 8071              1000000  127100396

Clauses = 5172
Variables = 134
TotalLiterals = 22266
FlipsPerSecond = 513073
BestStep_Mean = 88034.9
Steps_Mean = 954315.4
Steps_Max = 1000000
PercentSuccess = 10.00
BestSolution_Mean = 0.9
BestSolution_Median = 1
BestSolution_Min = 0
BestSolution_Max = 1
\end{verbatim}
Here we use the wrapper-script \texttt{ubcsat-okl}\footnote{see \href{https://github.com/OKullmann/oklibrary/blob/master/Experimentation/ExperimentSystem/ControllingLocalSearch/ubcsat-okl}{link to shell script}}, which outputs the output for the runs in a style typical for statistical data (easily readable for example by the tool R\footnote{\url{http://www.r-project.org/}}, as used in the \OKlibrary):
\begin{enumerate}
\item First a comment-line, starting with ``\#'', showing the parameters passed to the \texttt{ubcsat}-program (everything until ``\texttt{-rflush}'' is the default, coming from \texttt{ubcsat-okl}, and after that come the parameters from the command-line (possibly overriding the defaults)).
\item Then a line with the headings for the six output columns (\texttt{osteps} is for the number of rounds for reaching the optimum, while \texttt{msteps} is for the maximum number of steps).
\item Followed by data for the runs (above, one of the ten runs was successful).
\item Finally summary statistics (this is not readable by tools like R, and needed to be removed; however for a quick human-readable overview it is useful).
\end{enumerate}

\subsection{Running experiments}
\label{sec:OKlexperiments}

For running \texttt{Ubcsat}-algorithm to determine lower bounds on $\vdw(2;3,t)$ and $\vdwpd(2;3,t)$, also providing ``conjectures'' on the precise values, we have the following tools (using no parameters here serves to print some basic helper-information). First the general tool for $\vdw(2;t_0,t_1)$:
\begin{verbatim}
> RunVdWk1k2
ERROR[RunVdWk1k2]: Six parameters k1, k2, n0, alg, runs, cutoff
  are needed: The progression-lengths k1,k2, the starting number n0 of
  vertices, the ubcsat-algorithm, the number of runs, and the cutoff.
  An optional seventh parameter is a path for the file containing an
  initial assignment for the first ubcsat-run.
\end{verbatim}
The special version with k1=3, handling our case $\vdw(2;3,t)$:\footnote{see \href{https://github.com/OKullmann/oklibrary/blob/master/Experimentation/Investigations/RamseyTheory/VanderWaerdenProblems/RunVdW3k}{link to shell script}}
\begin{verbatim}
> RunVdW3k
ERROR[RunVdW3k]: Five parameters k, n0, alg, runs, cutoff
  are needed: The progression-length k, the starting number n0 of vertices,
  the ubcsat-algorithm, the number of runs, and the cutoff.
  An optional sixth parameter is a path for the file containing an
  initial assignment for the first ubcsat-run.
\end{verbatim}
For example
\begin{verbatim}
> RunVdW3k 27 678 gsat-tabu 1000 10000000
\end{verbatim}
starts the investigation of $\vdw(2;3,27)$ with $n=678$ (ad-hoc, no solution given), where the cut-off value (the number of rounds for stochastic local search) is $10^6$, and $1000$ runs are executed; from $n=679$ on the first three runs will use the solution found for $n-1$, while further runs use a random initial assignment.

Handling palindromic instances is done similarly\footnote{see \href{https://github.com/OKullmann/oklibrary/blob/master/Experimentation/Investigations/RamseyTheory/VanderWaerdenProblems/RunPdVdWk1k2}{link to shell script}}:
\begin{verbatim}
> RunPdVdWk1k2
ERROR[RunPdVdWk1k2]: Five parameters k1, k2, alg, runs, cutoff
  are needed: The progression-lengths k1,k2, the ubcsat-algorithm,
  the number of runs, and the cutoff.
\end{verbatim}

And for running complete solvers on palindromic instances we have\footnote{see \href{https://github.com/OKullmann/oklibrary/blob/master/Experimentation/Investigations/RamseyTheory/VanderWaerdenProblems/CRunPdVdWk1k2}{link to shell script}}:
\begin{verbatim}
> CRunPdVdWk1k2
ERROR[CRunPdVdWk1k2]: Three parameters k1, k2, solver, are needed:
  The progression-lengths k1, k2 and the SAT solver.
\end{verbatim}

\end{appendix}

\end{document}